%% file: ppaper-master.tex
\documentclass[11pt]{amsart}
\input{ppaper-macro.tex}

\author[Beaudry]{Agn\`es Beaudry}
\address{Department of Mathematics, University of Colorado, Boulder, CO 80309, USA}
\email{agnes.beaudry@colorado.edu}
\author[Lewis]{Chloe Lewis}
\address{Department of Mathematics, University of Wisconsin-Eau Claire, Eau Claire, WI 54701, USA}
\email{lewischl@uwec.edu}
\author[May]{Clover May}
\address[May]{School of Mathematics and The Heilbronn Institute for Mathematical Research,
University of Bristol, 
Bristol, BS8 1UG, UK}
\email{clover.may@bristol.ac.uk}
\author[Pauli]{Sabrina Pauli}
\address{Fachbereich Mathematik\\TU Darmstadt\\64289 Darmstadt, Germany}
\email{pauli@mathematik.tu-darmstadt.de}
\author[Tatum]{Elizabeth Tatum}
\address{Department of Mathemtatics\\University of Rochester\\Rochester, NY 14627}
\email{etatum@ur.rochester.edu}

\title[Parametrized $\underline{\mathbb{F}}_2$-Cohomology of $B_{C_2}O(1)$]{Parametrized $\underline{\mathbb{F}}_2$-Cohomology of $B_{C_2}O(1)$}

\begin{document}

\begin{abstract}
We compute the equivariant parametrized ordinary cohomology of the classifying space for real $C_2$-line bundles, $B_{C_2}O(1)$ with coefficients in the constant Mackey functor $\underline{\mathbb{F}}_2$. 
Equivariant parametrized cohomology is graded by representations of the equivariant fundamental groupoid, $RO(\Pi B)$, extending $RO(G)$-graded cohomology by assembling equivariant cohomology for all local coefficients into a single graded ring. It was developed by Costenoble and Waner to incorporate Thom isomorphisms for all equivariant vector bundles, and these play a critical role in our computation. As an important component, we compute the $RO(C_2)$-graded cohomology of all Thom spaces of real $C_2$-vector bundles over $B_{C_2}O(1)$.

To facilitate the computation, we prove many general results that can be applied to equivariant parametrized cohomology. We introduce a collection of characteristic classes, including orientation classes generalizing the $RO(G)$-graded classes of Hill--Hopkins--Ravenel. We investigate some consequences of base change on parametrized cohomology, and for $G=C_2$ import equivariant Steenrod squares and other tools to the parametrized context.
\end{abstract}

\maketitle

\setcounter{tocdepth}{1}  
\tableofcontents

\input{ppaper-intro}
\input{ppaper-eulerorientation}

\input{ppaper-ropip}

\input{ppaper-cohthom}

\input{ppaper-pcomputation}

\input{ppaper-fullrocomputation}

\input{ppaper-restriction}

\bibliographystyle{amsalpha}
\bibliography{ppaper-bib}
\end{document}

%% file: ppaper-macro.tex
\usepackage{amsmath}
\usepackage{amsfonts}
\usepackage{amssymb}
\usepackage{amsthm}
\usepackage{mathrsfs}
\usepackage{mathtools}
\usepackage{dsfont}
\usepackage{tikz-cd}
\usepackage[all]{xy}
\usepackage{pdfpages}
\usepackage{enumerate}
\usepackage{float}
\usepackage{todonotes}
\usepackage{graphicx}
\usepackage{xcolor}
\usepackage[colorlinks,linkcolor=blue,urlcolor=blue,citecolor=blue,destlabel=true]{hyperref}

\usepackage[capitalise]{cleveref}
\numberwithin{equation}{subsection}

\newtheorem{thm}{Theorem}[subsection]
\newtheorem*{thm*}{Theorem}
\newtheorem{theorem}[thm]{Theorem}
\newtheorem{cor}[thm]{Corollary}

\newtheorem{prop}[thm]{Proposition}
\newtheorem{proposition}[thm]{Proposition}
\newtheorem{lem}[thm]{Lemma}
\newtheorem{lemma}[thm]{Lemma}

\theoremstyle{definition}
\newtheorem{defn}[thm]{Definition}
\newtheorem*{defn*}{Definition}
\newtheorem{definition}[thm]{Definition}

\newtheorem{notation}[thm]{Notation}

\newtheorem{example}[thm]{Example}

\newtheorem{rem}[thm]{Remark}
\newtheorem{remark}[thm]{Remark}
\newtheorem*{rem*}{Remark}

\newtheorem*{conventions*}{Conventions}
\newtheorem*{ex*}{Example}
\newtheorem*{exs*}{Examples}

\newcommand{\Ab}{\mathrm{Ab}}

\newcommand{\N}{\mathbb{N}}

\newcommand{\id}{\mathrm{id}}

\DeclareMathOperator{\coker}{coker}

\DeclareMathOperator{\Rep}{Rep}

\DeclareMathOperator*{\colim}{colim}

\DeclareMathOperator{\Hom}{Hom}

\DeclareMathOperator{\Map}{Map}

\newcommand{\op}{\mathrm{op}}

\newcommand{\mF}{\underline{\mathbb{F}}_2}

\newcommand{\res}{\mathrm{res}}

\newcommand{\Sq}{\mathrm{Sq}}

\newcommand{\Z}{\mathbb{Z}}
\newcommand{\R}{\mathbb{R}}

\newcommand{\C}{\mathbb{C}}

\newcommand{\M}{\mathbb{M}}
\newcommand{\F}{\mathbb{F}}
\newcommand{\pt}{\mathrm{pt}}

\newcommand{\uF}{\underline{\mathbb{F}}}
\newcommand{\mM}{\underline{M}}
\newcommand{\mR}{\underline{R}}

\newcommand{\cO}{\mathcal{O}}

\newcommand{\RO}{RO}

\newcommand{\im}{\mathrm{im}}

\newcommand{\vV}{v\mathcal{V}}
\newcommand{\Th}{\mathrm{Th}}

\newcommand{\RP}{\mathbb{R}\mathrm{P}}

\newcommand{\wH}{\widetilde{H}}

\newcommand{\uR}{\underline{R}}

\newcommand{\hPTopb}[2]{\mathrm{hTop}_{#2}^{#1}}

\newcommand{\PSH}[2]{\mathcal{SH}_{#2}^{#1}}

\newcommand{\tr}{\mathrm{tr}}

\newcommand{\uZ}{\underline{\mathbb{Z}}}

\newcommand{\triv}{\varepsilon_{1,0}}
\newcommand{\sign}{\varepsilon_{1,1}}
\newcommand{\taut}{\gamma_{1,0}}
\newcommand{\staut}{\gamma_{1,1}}
\newcommand{\bun}{\gamma}

\newcommand{\pars}[1]{\left( #1 \right)}

\newcommand{\Span}{\mathrm{Span}}

\definecolor{darkspringgreen}{rgb}{0.09, 0.45, 0.27}

\def\Saone (#1,#2,#3){
\filldraw [#3] (#1+.6,#2+.5) circle (2pt);
\draw [#3, thick] (#1+.6,-3) -- (#1+.6,6);
\draw [#3, thick] (#1+1.4,-3) -- (#1+1.4,6);
\draw [#3, thick] (#1+.6,-2.5) -- (#1+1.4,-1.5);
\draw [#3, thick] (#1+.6,-1.5) -- (#1+1.4,-0.5);
\draw [#3, thick] (#1+.6,-0.5) -- (#1+1.4,0.5);
\draw [#3, thick] (#1+.6,0.5) -- (#1+1.4,1.5);
\draw [#3, thick] (#1+.6,1.5) -- (#1+1.4,2.5);
\draw [#3, thick] (#1+.6,2.5) -- (#1+1.4,3.5);
\draw [#3, thick] (#1+.6,3.5) -- (#1+1.4,4.5);
\draw [#3, thick] (#1+.6,4.5) -- (#1+1.4,5.5);
}

\def\Sqoneladdereven (#1,#2,#3){
\draw [#3, thick] (#1+.6,#2-2.5) -- (#1+1.4,#2-2.5);
\draw [#3, thick] (#1+.6,#2-0.5) -- (#1+1.4,#2-0.5);
\draw [#3, thick] (#1+.6,#2+1.5) -- (#1+1.4,#2+1.5);
\draw [#3, thick] (#1+.6,#2+3.5) -- (#1+1.4,#2+3.5);
\draw [#3, thick] (#1+.6,#2+5.5) -- (#1+1.4,#2+5.5);
}

\def\Sqoneladderodd (#1,#2,#3){
\draw [#3, thick] (#1+.6,#2-2.5) -- (#1+1.4,#2-2.5);
\draw [#3, thick] (#1+.6,#2-0.5) -- (#1+1.4,#2-0.5);
\draw [#3, thick] (#1+.6,#2+1.5) -- (#1+1.4,#2+1.5);
\draw [#3, thick] (#1+.6,#2+3.5) -- (#1+1.4,#2+3.5);
}

\def\mtwoonearrow (#1,#2,#3){
\draw [thick, -stealth, #3] (#1,#2) -- (#1,#2+1);
\draw [thick, -stealth, #3] (#1,#2) -- (#1+1,#2+1);
     \fill[#3] (#1,#2) circle (3pt);
     \draw [thick, -stealth, #3] (#1,#2-2) -- (#1,#2-3);
      \draw [thick, -stealth, #3] (#1,#2-2) -- (#1-1,#2-3);
       \fill[#3] (#1,#2-2) circle (3pt);
}

\def\mtwoone (#1,#2,#3){
\draw [thick, #3] (#1,#2) -- (#1,#2+1);
\draw [thick, #3] (#1,#2) -- (#1+1,#2+1);
     \fill[#3] (#1,#2) circle (3pt);
     \draw [thick, #3] (#1,#2-2) -- (#1,#2-3);
      \draw [thick, #3] (#1,#2-2) -- (#1-1,#2-3);
       \fill[#3] (#1,#2-2) circle (3pt);
}

\def\mtwoonehalf (#1,#2,#3){
\draw [thick, #3] (#1,#2) -- (#1,#2+1);
\draw [thick, #3] (#1,#2) -- (#1+1,#2+1);
     \fill[#3] (#1,#2) circle (3pt);
}

\def \dropell (#1,#2,#3){
    \draw[#3, thick, dashed, stealth-] (#1-.4,#2+.1) .. controls (#1+0.2,#2+1) .. (#1,#2+2);
}

\newcommand{\autP}{\chi}

\newcommand{\crush}{\rho}
\newcommand{\gen}{\tau}

\newcommand{\uu}{u}
\renewcommand{\aa}{a}
\newcommand{\eee}{e}
\newcommand{\kk}{g}
\newcommand{\ee}{e}
\newcommand{\vv}{\nu}

\newcommand{\VV}{\rho_G}

\newcommand{\bock}{\Sq^{1,0}}

\newcommand{\tautB}{\gamma}
\newcommand{\tautBchi}{\chi\gamma}

%% file: ppaper-intro.tex

\section{Introduction}

Our main goal of this work is to study $P = B_{C_2}O(1)$, the classifying space of $C_2$-equivariant line bundles. 
A strong tool in the equivariant toolkit is $RO(G)$-graded cohomology, but it fails in this instance to capture the full equivariant structure of $P$. 
In particular, the Thom isomorphism fails to hold, even for the tautological bundle.  
The theory of equivariant parametrized cohomology, also known as $RO(\Pi B)$-graded cohomology, was developed by Costenoble and Waner in \cite{CW_book}, building on \cite{CW_Duality}, \cite{CW_Thom}, and \cite{CMW}. This theory was created to incorporate Thom isomorphisms for all equivariant vector bundles and thus better capture equivariant characteristic classes. 

In this paper, we compute the equivariant parametrized cohomology of $P$ with $\mF$-coefficients. 
In order to carry out this technical computation, we recall some of the main constructions of \cite{CW_book} (some of which are also described in \cite{witpaper}), investigate some of the nuances, explore the theory of characteristic classes in this setting, and compute the $RO(C_2)$-graded cohomology of stunted projective spaces. In addition to further developing the theory of characteristic classes and ironing out a number of details to make computations in equivariant parametrized cohomology more accessible, our computation uses the theory in a novel way. 

Given a $G$-space $B$, parametrized cohomology is graded by $RO(\Pi B)$, representations of the equivariant fundamental groupoid $\Pi B$, and takes coefficients in a parametrized Mackey functor. 
For $\gamma$ any virtual $G$-vector bundle over $B$ and $\alpha \in RO(\Pi B)$, the parametrized Thom isomorphism takes the form: 
\[ {H}^{\alpha}_B(B) \xrightarrow[\cong]{t_\gamma} \widetilde{H}^{\alpha+\gamma}_B(\Th_B(\gamma)).\]
Part of the strength of this theory is that this Thom isomorphism holds for any choice of coefficients, effectively incorporating local coefficients into the grading. Our computation treats this Thom isomorphism as fundamental.

In the first part of the paper, we review and elaborate on the theory of equivariant parametrized cohomology for $G$ a finite group following \cite{CW_book}. We investigate the parametrized Thom isomorphism for several different kinds of bundles. We then go on to develop a theory of Euler classes and orientation classes, generalizing these classes in $RO(G)$-graded cohomology from \cite{HHR}. We identify a number of other classes in cohomology, such as homogeneity units and triviality classes, extend equivariant Steenrod squares to parametrized cohomology (for $G=C_2$ and $\mF$-coefficients), and lay the foundations for our computation.

While the Euler classes are defined in the usual way, the orientation classes are trickier because orientation theory for $G$-bundles is subtle. For homogeneous bundles, also called $V$-bundles, there are many treatments (see  \cite{BhattZou} for a modern approach). 
A goal of Costenoble--Waner--May \cite{CMW} was to develop a theory of orientations for more general equivariant bundles and their Definition 2.8 gives a notion of orientability. 
The notion of orientability we need differs from both of these approaches, but naturally generalizes the notion of orientability in the nonequivariant setting. We develop this theory and prove several nice properties of orientation classes in \cref{sec:orientation}. 

In the second part, we turn to $G = C_2$ and $P = B_{C_2}O(1)$. To compute the $C_2$-equivariant parametrized cohomology of the classifying space $P$, we first compute the ring structure of the grading $RO(\Pi P)$.  Let $KO(\Pi P)$ denote the image of virtual vector bundles in $RO(\Pi P)$.
 
\begin{thm*}[c.f.\ \cref{thm:ROPiP}, \cref{thm:KOPiP}]
   There is an isomorphism 
   \[RO(\Pi P)\cong \Z^3 \times \Z/2\times \Z/2.\]
   If we denote an element by a coordinate $(p,q,n,\epsilon,\mu)$,
     the degrees $(p,q,0,0,0)$ correspond to the inclusion $\crush^* \colon RO(C_2) \to RO(\Pi P)$. The degrees in which the last coordinate $\mu$ vanishes corresponds to the subgroup 
   \[KO(\Pi  P)\cong \Z^3 \times \Z/2.\]
\end{thm*}

\begin{rem}\label{rem:KOC2intro}
There are four $C_2$-line bundles on $P$ up to isomorphism, namely, the trivial bundle $\triv$, the constant bundle $\sign$ with fibers the one-dimensional sign representation of $C_2$, the tautological bundle $\taut$ and the tensor product $\staut=\sign\otimes \taut$. See \cref{rem:linebundles} for more details. The $C_2$-equivariant real $K$-theory of $P$ is free abelian. In fact,
\[KO_{C_2}(P) \cong  \Z\{\triv, \sign, \taut, \staut\},\]
see \cite{BZ}.  However, in $KO(\Pi P)$ the classes corresponding to these line bundles satisfy  
\[2\dim(\triv+\sign-(\taut+\staut))=0. \] 
This is related to the fact that the bundle $\taut+\staut$ is homogeneous, and so admits a Thom isomorphism in $RO(G)$-graded cohomology. See \cref{def:homogeneous} and the surrounding discussion in \cref{sec:KOPiB}. This relation explains the appearance of the 2-torsion in $KO(\Pi P)$.
\end{rem}

Our main computational result is the parametrized cohomology of $P$ as an algebra over $\M_2$, the $RO(C_2)$-graded cohomology of a point. With $\mF$-coefficients, we have
\[\M_2 = H^{*,*}(\pt,\mF) = H^{*,*}(C_2/C_2 ,\mF) \cong \F_2[\uu,\aa]\oplus \F_2\{\theta/(\uu^i\aa^j) : i,j\geq 0\}.\]
  We use the motivic grading, where $(p,q)$ denotes representation $\R^{p,q}$ which is a direct sum of $q$ one-dimensional sign representations and $p-q$ trivial representations. Let $\sigma=\R^{1,1}$ denote the one-dimensional sign representation. The class $\uu = \uu_\sigma$ in degree $(0,1)$ and the class $\aa=\aa_\sigma$ in degree  $(1,1)$ (often called $\tau$ and $\rho$) are the orientation class and Euler class of the bundle $\sigma$ over a point. 
\begin{thm*}[c.f.\ \cref{thm:ROcohomologyofP}]
    The $RO(\Pi P)$-graded cohomology of $P$ is isomorphic to the $\M_2$-algebra
    \[H^{*,*,*,*,*}_{P}(P,\uF_2)\cong \M_2[\uu_{10},\uu_{11},\aa_{10},\aa_{11},\eee,\vv]/\sim\]
with relations given by
\[ \uu_{10}\uu_{11}-\uu \eee,\aa_{10}\uu_{11}+\aa_{11}\uu_{10}-\aa \eee, \quad \eee^2-1, \quad \vv^2-1.\]
The classes $\aa_{10}$ and $\uu_{10}$ are the Euler and orientation classes of $\taut$ and $\aa_{11}$ and $\uu_{11}$ those of $\staut$. 
The class $\ee$ is the homogeneity unit  of the bundle $\taut+ \staut$. 
The degrees of the elements are
\begin{align*}
|\eee| &= (0,0,0,1,0) & |\vv| &= (0,0,0,0,1)   \\ 
|\uu_{10}| &= (0,0,1,0,0)  & |\uu_{11}| &= (0,1,-1,1,0)  \\ 
|\aa_{10}|&=(1,0,1,0,0) & |\aa_{11}| &= (1,1,-1,1,0).
\end{align*}
In particular, the cohomology is free as an $\M_2$-module.
\end{thm*}

The $KO(\Pi P)$-graded cohomology of $P$ in \cref{thm:KOcohomologyofP} is the subring corresponding to the vanishing of the last coordinate, and can be read off by setting the unit $\vv = 1$.

We first compute the $KO(\Pi P)$-graded portion. One of the key results we use is the following version of the Thom isomorphism in constant coefficients. It allows us to express the $KO(\Pi B)$-graded cohomology of any base $B$ as the $RO(G)$-graded cohomology of a Thom space. For $\gamma \in KO(\Pi B)$ and $\star \in RO(G)$, there is an isomorphism 
\[ H^{-\gamma+\star}_B(B,\mR) \xrightarrow{\cong} \widetilde{H}^{\star}(\Th(\gamma),\mR).\]
See \cref{prop:cohBThom} and \cref{thm:KOPIB}. 

Returning to $P$, our technique is then to compute the $RO(C_2)$-graded cohomology of stunted projective spaces. Let $O(1)$ act on the $C_2$-representation $\R^{p,q}$ by multiplication by $-1$, 
and let $\bun_{p,q}$ be the $C_2$-bundle
\[ E_{C_2}O(1)\times_{O(1)}\R^{p,q} \to B_{C_2}O(1)=P. \]
For $p \geq q \geq 0$, the Thom space $\Th(\bun_{p,q})$ is the equivariant stunted projective space $P/P(\R^{p,q})$.   By the Thom isomorphism above, the cohomology of these Thom spaces make up various $RO(C_2)$-graded pages of the $KO(\Pi P)$-graded cohomology. We include the $RO(C_2)$-graded results here, which may be of independent interest. 

\begin{thm*}[c.f.\ \cref{cor:pqduality}, \cref{thm:coh_pq}]
    The $RO(C_2)$-graded cohomology of $\Th(\bun_{p,q})$ is free as an $\M_2$-module. 
There are isomorphisms
\[H^{*,*}(\Th(\bun_{p,q}), \uF_2) \cong H^{*,*}(\Th(\bun_{p,p-q}), \uF_2). \]
\begin{enumerate}[(1)]
\item When $p-q\geq q\geq 0$, the cohomology is 
generated as an $\M_2$-module by the unit and classes $x_r$ for $r\geq p$ of degree
\begin{align*}
|x_{r}| =\begin{cases} (r, p-q) & p\leq r \leq 2(p-q)\\
(r, \lceil r/2 \rceil) & r>2(p-q).
\end{cases}
\end{align*}
    \item  When $0\leq p-q \leq q$, the cohomology is generated as an $\M_2$-module by the unit and classes $x_r$ for $r\geq p$ of degree
\begin{align*}
|x_{r}| =\begin{cases} (r, q) & p\leq r \leq 2q\\
(r, \lceil r/2 \rceil) & r>2q.
\end{cases}
\end{align*}
    \end{enumerate}
\end{thm*}

Collecting the cohomology of these Thom spaces yields the $KO(\Pi P)$-graded cohomology of $P$ as an $\M_2$-module. To determine the relations, and hence establish the claimed $\M_2$-algebra structure, we use our developed theory of Euler and orientation classes, together with the equivariant Steenrod squares.

Finally, the full $RO(\Pi P)$-graded cohomology of $P$ is obtained from the $KO(\Pi P)$-graded cohomology by adjoining a unit $\vv$ in degree $(0,0,0,0,1)$. That is,
\[
H^{RO(\Pi P)}_{P}(P,\mF) \cong H^{KO(\Pi P)}_{P}(P,\mF) \otimes_{\M_2} \M_2[\vv]/(\vv^2-1).
\]
In order to show this last generator is a unit, we identify $\vv$ as a triviality class and prove that forgetfully it restricts to $1$.

\medskip

The results in this paper contribute to the growing body of recent computations in $C_2$-equivariant parametrized cohomology. In \cite{CostenobleB}, Costenoble used isotropy separation to compute the parametrized cohomology of $B_{C_2}U(1)$, the classifying space for $C_2$-equivariant complex line bundles with coefficients in the Burnside Mackey functor $\underline{A}$. There is also work computing the cohomology with Burnside coefficients of finite complex $C_2$-projective spaces \cite{CHT_C2}, of $B_{C_2}U(2)$ \cite{CH_planes},  of $B_{C_2}(U(1)^2)$ \cite{CH_BT2}, and of complex quadrics \cite{CostenobleHusdonI,CostenobleHusdonII,CostenobleHusdonIII}.

In all of the known computations, $RO(\Pi B)$ is free abelian and coincides with $KO(\Pi B)$. The cohomology computation presented here is the first complete $RO(\Pi B)$-graded computation incorporating degrees beyond $KO(\Pi B)$, and the first computation with non-free grading. It is this lack of freeness that leads us to $\mF$-coefficients, in order to avoid some technicalities and signs. Moreover, this is the first computation to directly harness the power of the Thom isomorphisms developed in \cite{CW_book}. Our approach recovers the cohomology of $B_{C_2}U(1)$ from \cite{CostenobleB} with $\mF$-coefficients.

\begin{rem}
There are various perspectives on parametrized homotopy theory, from the classical work of May--Sigurdsson \cite{MaySig} and Costenoble--Waner \cite{CW_book} to the more contemporary approaches using $\infty$-categories such as Barwick et al. \cite{BDGN} and  Ando et al. \cite{ABGHR, ABGHR2}. 
This paper is computational in nature and, even if we work with point-set models, the main results should be indifferent to the choice of framework.
\end{rem}

\subsection{Organization of the paper}

We have divided this paper into two parts. The main goal of our project was to compute the parametrized cohomology of $B_{C_2}O(1)$. \cref{part:II} of the paper is dedicated to this. In doing this computation, we also proved some general results that will apply to other bases.
\cref{sec:classes} of the paper collects these results, and so is focused on general constructions in parametrized cohomology.

We begin \cref{sec:classes} by reviewing the definition of the extended grading $RO(\Pi B)$ in \cref{sec:ROPiB}. In \cref{sec:parametrized cohomology}, we recall Costenoble--Waner's definition of parametrized cohomology. In \cref{sec:thomiso}, we review the Thom isomorphism. In  \cref{sec:char}, we discuss several kinds of characteristic classes in parametrized cohomology associated to vector bundles over the bases. The first two, Euler and orientation classes, are more familiar, but we also introduce classes called homogeneity units that exist only for homogeneous bundles and triviality classes. In order to define orientation classes, we introduce a definition of orientation for $G$-bundles.
In \cref{sec:units}, we discuss the effect of base-change when working with $\mF$-coefficients. In \cref{sec:M2module}, we specialize to the case $G=C_2$. We import the equivariant Steenrod operations to the parametrized context and examine the forgetful exact sequence for parametrized cohomology.

In \cref{part:II}, we turn to our main computation, that of the cohomology of $P=B_{C_2}O(1)$. We first compute the grading in \cref{sec:ROPiP}. Our method for computing the parametrized cohomology is to compute first the cohomology in degrees isomorphic to the dimensions of vector bundles, i.e., the cohomology of Thom spaces. We do this in \cref{section:cohomologyofthomspaces}. This section is of independent importance for those who are interested in $RO(C_2)$-graded cohomology in the non-parametrized context. \cref{sec:cohPall} is dedicated to the computation of $H^{RO(\Pi P)}_P(P,\mF)$. Finally, in \cref{sec:B}, we apply our techniques to recover a result of Costenoble \cite{CostenobleB} and recompute $H^{RO(\Pi B)}_B(B,\mF)$ for $B=B_{C_2}U(1)$.

\subsection{Conventions}
Throughout, we always assume that $G$ is a finite group. We will suppose that $B$, the base $G$-space, is a $G$-CW complex of finite type. We will always let $R$ denote a (classical) unital commutative ring, and let $H\mR$ be the Eilenberg--MacLane spectrum associated to the constant Mackey functor $\mR$. We let $\crush \colon  B \to G/G$ be the unique map. One can pull back $H\mR$ along $\crush$ to get a parametrized spectrum over $B$ which we also call $H\mR$. This is the constant parametrized spectrum which for each $b$ in $B$ has fiber (or value) $H\mR$. We let $\Pi B$ denote the equivariant fundamental groupoid of $B$.

%% file: ppaper-eulerorientation.tex

\part{Constructions in parametrized cohomology}\label{sec:classes}

In the first part of this paper, we review some important definitions and constructions in $RO(\Pi B)$-graded equivariant cohomology as in \cite{CW_book}. We relate parametrized cohomology  to twisted cohomology. 
We refer the reader to \cite{CW_book} or our survey \cite{witpaper} for a more detailed overview of the background.

\section{The extended grading}
Parametrized cohomology is graded on representations of the fundamental groupoid. In this section, we review what that means.

\subsection{The fundamental groupoid $\Pi B$}\label{sec:ROPiB}
Recall that in this paper, $G$ always denotes a finite group.
Let $\cO_G$ be the orbit category of $G$. That is, $\cO_G$ is the category whose objects are orbits $G/H$ for $H$ a subgroup of $G$, considered as left $G$-sets. The morphisms are  maps of $G$-sets $\alpha \colon G/H \to G/K$. 

Given a $G$-space $B$, the equivariant fundamental groupoid 
$\Pi B$ (also denoted $\Pi_G B$ in the literature) is a category fibered in groupoids over $\cO_G$.
The objects of $\Pi B$ are equivariant maps
\[x \colon G/H \to B\]
for $H$ a subgroup of $G$.
Given two objects $x\colon G/H \to B$ and $y \colon G/K \to B$, a morphism $x \to y$ is a pair $(\alpha, \omega)$ where 
$\alpha \colon G/H \to G/K$ is a morphism in the orbit category $\cO_G$ and $\omega$ is a $G$-homotopy 
\[ \omega \colon G/H \times [0,1] \to B \]
from $x$ to $y \circ \alpha$. Furthermore, morphisms $(\alpha,\omega)$ and $(\alpha,\omega')$ are identified if $\omega$ and $\omega'$ are homotopic relative to their endpoints. The category $\Pi B$ is a category fibered in groupoids over $\cO_{G}$ as follows: the functor
\[\phi \colon \Pi B \to \cO_G\]
sends $x \colon G/H \to B$ to $G/H$ and $\phi(\alpha,\omega)=\alpha$. The fiber over $G/H$ can be identified with the (nonequivariant) fundamental groupoid of the fixed points $B^H$.

Given a map of $G$-spaces $f\colon A \to B$, there is a functor on fundamental groupoids $f_* = \Pi f \colon \Pi A \to \Pi B$  induced by postcomposition with $f$. Given a subgroup $H\subset G$, there is an induction functor
\[G\times_H(-)\colon \Pi i^*_H B \to \Pi B \]
which applies induction $G\times_H(-)$ to objects and morphisms.

\subsection{The group $RO(\Pi B)$}
Equivariant parametrized cohomology is graded on the group $RO(\Pi B)$, which generalizes $G$-bundles. We recall the particulars below. Throughout, by a $G$-bundle we mean an actual orthogonal $G$-vector bundle over a $G$-space $B$. We also use virtual $G$-bundles in many circumstances, and we may emphasize that a bundle is not virtual by calling it an \emph{actual} $G$-bundle. 

Let $\vV_G(n)$ be the category whose objects are pairs of $G$-bundles over orbits $G/H$ of the form $(G\times_HV_1,G\times_HV_2)$ for orthogonal $H$-representations $V_i$ chosen from a fixed set of representatives such that the virtual dimension $\dim V_1-\dim V_2$ is equal to $n$. A
 morphism in $\vV_G(n)$ is a triple $(\alpha, f_1,f_2)$, where each $f_i$ is a bundle map
\[ \xymatrix{
G\times_HV_i \ar[r]^-{f_i} \ar[d] & G\times_KW_i \ar[d] \\
G/H \ar[r]^-\alpha & G/K
}\]
that is an isomorphism on fibers. 
Two morphisms are equal if they are stably equivalent or homotopic via a $G$-homotopy covering $\alpha$. Details of the equivalence relation can be found in \cite[Def. 2.2 \& \S IV.19]{CMW}. The category $\vV_G(n)$ is also a category fibered in groupoids over $\cO_G$ via the functor $\phi \colon \vV_G(n) \to \cO_G$ sending a virtual bundle to its base space and a morphism $(\alpha, f_1,f_2)$ to $\alpha$. 

For $B$ a $G$-space, a representation of $\Pi B$ of dimension $n$ is a functor $\gamma \colon \Pi B \to \vV_G(n)$ commuting with the functors $\phi$ to $\cO_G$. Using the direct sum of virtual bundles and taking representations up to natural isomorphism, representations form an abelian group, denoted $RO(\Pi B)$. This generalizes the abelian group $RO(G)$ of (possibly virtual) real orthogonal $G$-representations in the sense that for $B = G/G$ we have 
\[
RO(\Pi G/G) \cong RO(G).
\]

Given a map of $G$-spaces $f\colon A \to B$, there is a pullback
\begin{align}\label{eq:pbf}
f^* \colon RO(\Pi B) \to RO(\Pi A),\end{align}
a homomorphism induced by precomposing a representation $\gamma$ with $\Pi f$.  On the other hand, given a subgroup $H\subset G$, there is a restriction functor \[i^*_H \colon  RO(\Pi B) \to RO(\Pi i^*_HB)\] obtained by taking $\gamma$ to $i^*_H\gamma\circ (G\times_H-)$, where $i^*_H \colon \vV_G(n) \to \vV_H(n)$ takes a virtual $G$-bundle over $G/K$ to its pullback along the inclusion $H/H\cap K \to G/K$.

The tensor product of bundles  gives $RO(\Pi B)$ a ring structure and the pullbacks \eqref{eq:pbf} are ring homomorphisms. So, for $\crush \colon B \to G/G$, the ring homomorphism  
\begin{align}\label{eq:modstructureROG}
\crush^* \colon RO(G) \cong RO(\Pi G/G) \to RO(\Pi B)
\end{align}
gives $RO(\Pi B)$ the structure of an $RO(G)$-module.

\subsection{The subgroup $KO(\Pi B)$}\label{sec:KOPiB}
For any virtual $G$-vector bundle $\xi$ over $B$, there is a corresponding representation
$\dim \xi \in RO(\Pi B)$
obtained by pulling back $\xi$ along the objects $x\colon G/H \to B$ of the fundamental groupoid. We abuse notation and write $\xi=\dim(\xi)$.
This gives rise to a homomorphism
\[\dim \colon KO_G(B) \to  RO(\Pi B) \]
where $KO_G(B)$ denotes the equivariant real $K$-theory of $B$.
\begin{defn}
\label{defn:KOPiB}
    Let $KO(\Pi B)$ denote the image of $\dim$ in $RO(\Pi B)$ so
\[ KO_G(B) \xrightarrow{\dim} KO(\Pi B) \hookrightarrow  RO(\Pi B). \]
\end{defn}

\begin{rem}
    As in \cite[Remark 2.27]{witpaper}, for $\crush \colon B \to G/G$,  the image of 
    \[\crush^* \colon RO(G) \cong RO(\Pi G/G) \to RO(\Pi B)\]
    is given by trivial bundles and is always contained in $KO(\Pi B)$. 
    For $V$ a (possibly virtual) orthogonal $G$-representation, $\crush^* V = \dim(\xi_V)$, where $\xi_V = B \times V$ as described in \cite[Example 2.23]{witpaper}. By abuse, we write $\dim(\xi_V) = V$. 
\end{rem}

\begin{rem}
The map 
    \[\crush^* \colon RO(G) \cong RO(\Pi G/G) \to RO(\Pi B)\]
    is not necessarily injective. 
    For example, if $G=C_2$ and $B = S^1$ with the antipodal action, the map $\crush^*$ cannot be injective. 
    \end{rem}
    
The map $\crush^*$ is better behaved when $B$ has a fixed point.

\begin{prop}\label{prop:ROGinjects}
    Suppose $B$ has a fixed point $x \colon G/G \to B$. Then the map 
    \[\crush^* \colon RO(G) \cong RO(\Pi G/G) \to RO(\Pi B)\]
    is injective.
\end{prop}

\begin{proof}
    Consider $RO(\Pi -)$ applied to the composition $\crush x = \id$.
\end{proof}

\begin{rem}\label{rem:non-trivial-but-dim-V}
It is possible for non-trivial bundles $\gamma$ to nonetheless satisfy $\dim(\gamma)=V$, see for example $2(\gamma_{1,0} + \gamma_{1,1})$ from \cref{rem:KOC2intro} and \cref{rem:linebundles}, which has $\dim(2(\gamma_{1,0} + \gamma_{1,1})) = \R^{4,2}$ in $KO(\Pi P)$.
\end{rem}

More generally, there are homogeneous $G$-bundles that have the same representation over every fiber, but whose dimension does not land in $RO(G)$. That is, there are bundles whose representations in $RO(\Pi B)$ agree with trivial bundles on objects but not morphisms. See for example, $\gamma_{1,0} + \gamma_{1,1}$ as in \cref{rem:KOC2intro} and \cref{rem:linebundles}, which has fibers over fixed points $\R^{2,1}$ but has $\dim(\gamma_{1,0} + \gamma_{1,1}) \neq \R^{2,1}$ in $KO(\Pi P)$.

\begin{defn}\label{def:homogeneous}
Let $B$ be a $G$-CW complex. A $G$-vector bundle $\xi$ over $B$ is \emph{homogeneous} if for every point $b \in B$, the fiber $\xi_b$ over $b$ is isomorphic to the restriction of a fixed $G$-representation $\xi_0$ as a representation of the stabilizer $G_b$. We say a virtual bundle $\xi = \xi' - \xi''$ is \emph{homogeneous} if $\xi'$ and $\xi''$ are homogeneous. 
\end{defn}

\begin{rem}
Homogeneous bundles are also called $V$-bundles in the literature, where $V=\xi_0$ above. See for example \cite[Example 1.1.5 (5)]{CW_book}.
\end{rem}

\begin{defn}
    Let $HO(\Pi B)$ denote the image of $\dim$ in $RO(\Pi B)$ restricted to homogeneous bundles.
\end{defn}

It follows from \cref{prop:ROGinjects} that we have the following inclusions (which are generally not surjections, see \cref{rem:non-trivial-but-dim-V}).

\begin{prop}
Let $B$ be a $G$-space whose underlying space is path connected. Assume there is a fixed point $x \colon G/G \to B$. 
There are subgroup inclusions
\[
RO(G)\hookrightarrow HO(\Pi B)  \hookrightarrow KO(\Pi B)  \hookrightarrow RO(\Pi B).
\]
\end{prop}

When $B$ has a trivial action, there is no difference between $RO(\Pi B)$ and $KO(\Pi B)$. We showed this in \cite[Example 2.29]{witpaper} for $G=C_2$ and $B=B^{G}$. This generalizes to all finite groups as follows.
\begin{prop}\label{prop:trivKOROequal}
If $B$ is a path connected $G$-space with trivial action, then
\[KO(\Pi B) = RO(\Pi B).\] 
\end{prop}
\begin{proof}
Since $\Pi B$ has finitely many isomorphism classes of objects, $RO(\Pi B)$ is the group completion of the semi-group of isomorphism classes of actual representations. So, it's enough to prove that any actual representation is in $KO(\Pi B)$.
Let $x\in B$ and $\pi = \pi_1(B,x)$. An actual representation $\alpha$ determines, and is  determined by, a pair consisting of a representation $V$ and a homomorphism 
\[\varphi_\alpha \colon \pi \to \mathrm{Aut}_{G}(V),\]
where $\mathrm{Aut}_{G}(V)$ is the automorphism group of the $G$-representation $V$. We then emulate the construction of \cite[Example 2.29]{witpaper}. 
Let $E\pi \to B\pi$ be the universal principal $\pi$-bundle, and consider the $G$-vector bundle
\[ E\pi \times_{\pi}^\varphi V \to B\pi \]
where $\pi$ acts on $V$ via $\varphi$. Here $ E\pi$ and $B\pi$ are given the trivial $\pi$ action. 
Take the pullback of this bundle along the first Postnikov stage $B \to B\pi$ and let $\xi(\alpha)$ denote the resulting $G$-bundle over $B$.  By construction, $\dim \xi(\alpha) =\alpha$.
\end{proof}

More generally, we give a strategy for identifying the subgroup $KO(\Pi B)$.
For this, we need some classical facts about equivariant $K$-theory.

Let $I(G)$ be the kernel of the augmentation $\varepsilon \colon RO(G) \to \Z$, where $\varepsilon (V) = |V|$, the underlying dimension of $V$. 
In this section, we will use the Atiyah--Segal  map
\[ c\colon KO_G(B) \to KO_G(B \times EG) \]
 induced by $B \times EG \to B\times \pt$. By the Atiyah--Segal completion theorem \cite{AtiyahSegal}, when $B$ is compact the map $c$ is the completion at $I(G)$. Note that
 \[ KO_G(B \times EG) \cong KO(B\times_G EG), \]
 as shown in \cite[Prop. 2.1]{SegalEqKThy}. 
 
Recall from \eqref{eq:modstructureROG} that $RO(\Pi B)$ is an $RO(G)$-module, so we can also consider the completion of $RO(\Pi B)$ at $I(G)$. Furthermore, the homomorphism  $\dim \colon KO_G(B) \to RO(\Pi B) $ is a map of $RO(G)$-modules.
 \begin{lem}\label{lem:completiondiagram} 
There is a
map 
\begin{align}\label{eq:dimcomp}
\dim\colon KO_G(B\times EG) \to RO(\Pi B)^{\wedge}_{I(G)}
\end{align}
that makes the diagram
\begin{align*}
\xymatrix{
KO_G(B) \ar[d]_-\dim \ar[r]^-{c} & KO_G(B\times EG) \ar@{.>}[d]^-{\dim} \\
RO(\Pi B) \ar[r]^-{c} & RO(\Pi B)^{\wedge}_{I(G)}
}
\end{align*}
commute, where the horizontal arrows are the completion maps. In particular, the image of $KO(\Pi B)$ under the completion map is contained in the image of the map \eqref{eq:dimcomp}.
 \end{lem}
 \begin{proof}
 Whenever $B$ is compact, this follows from the Atiyah--Segal completion theorem which states that $c$ is the completion at the augmentation ideal. Recall that we always assume that $B$ is homotopy equivalent to a $G$-CW complex of finite type. Let $B^n$ be the $n$-skeleton of $B$. Then $B^n$ is compact. Further, the map 
 \begin{align}\label{eq:ROiso}
 RO(\Pi B) \to RO(\Pi B^n)
 \end{align}
 is an isomorphism for $n\geq 2$ since $\Pi B^n \to \Pi B$ is an equivalence of categories in that range.
We have a commutative diagram
  \begin{align*}
  \xymatrix@C=1.2pc{
& \displaystyle{\varprojlim_n} \ KO_G(B^n)  \ar[rr]^-{c} \ar'[d][dd]^(.45){}
&& \displaystyle{\varprojlim_n}  \ KO_G(B^n)^{\wedge}_{I(G)}  \ar[dd] \\
KO_G(B) \ar[ru] \ar[rr]^(.55){c} \ar[dd]^-\dim && KO_G(B\times EG) \ar[ru] \ar@{..>}[dd]^(.65){\dim} & \\
&  {\displaystyle \varprojlim_n} \  RO(\Pi B^n) \ar'[r]^(.65){}_(.65){}[rr] && {\displaystyle \varprojlim_n} \  RO(\Pi B^n)^{\wedge}_{I(G)}  \\ RO(\Pi B) \ar[ru]^-\cong \ar[rr]^-c && RO(\Pi B)^{\wedge}_{I(G)} \ar[ru]^-\cong &
}
\end{align*}
where the back square exists and commutes since the $B^n$ are compact, and inward pointing maps are just induced by the inclusions $B^n\subset B$ and the contravariance of the functors.
The bottom inward pointing maps are isomorphisms because the towers are eventually constant by \eqref{eq:ROiso} and thus the $\lim{}^1$ terms vanish.
The map \cref{eq:dimcomp} is the composition of the arrows around the right face of the cube.
 \end{proof}

The strategy of our main computation (which we apply in our example in \cref{sec:KOPiP} for the base $B = B_{C_2}O(1)$) is to prove that the map $c \colon RO(\Pi B) \to RO(\Pi B)^\wedge_{I(G)}$ is injective. Then 
\[ c \colon KO(\Pi B) \xrightarrow{\subset} \dim(KO_G(B\times EG))\]
and the idea is that $KO_G(B\times EG)$ is generically more accessible computationally than $KO_G(B)$. To demonstrate this, we give the following identification of the upper right corner of the diagram of \cref{lem:completiondiagram} when $B$ is an equivariant classifying space.
\begin{lem}\label{lem:Gamma}
Let $\Gamma = G \times H$, where $H$ is a compact Lie group. Let 
\[B_G H \simeq \Map(EG, EH)/H\]
be the classifying space for principal $G$-$H$-bundles (as in \cite{Lashof}).
There is an isomorphism
\[KO_{G}(B_G H  \times EG ) \cong KO(B\Gamma)  \cong RO(\Gamma)^\wedge_{I(\Gamma)}.\]
\end{lem}
\begin{proof}
It follows from \cite[p.242]{May_some}
that
\[    B_{G}H \times_{G} EG \simeq  B\Gamma.\]
This gives the first isomorphism. The second is the Atiyah--Segal completion theorem for the compact Lie group $\Gamma$. 
\end{proof}

\section{Parametrized cohomology}\label{sec:parametrized cohomology}
We now turn to parametrized cohomology, a cohomology theory graded on $RO(\Pi B)$. It is defined for ex-$G$-spaces, which play the role of pointed spaces in the equivariant parametrized setting.

\subsection{Homotopy recollections and notation}
Let $B$ be a $G$-space. An ex-$G$-space (over $B$) is a $G$-space $X$ together with an equivariant map $p\colon X \to B$ and an equivariant section $s \colon B \to X$. 
Let $\hPTopb{G}{B}$ denote the homotopy category of ex-$G$-spaces over $B$ and $\mathcal{SH}^G_B$ denote the stable homotopy category of  $G$-spectra parametrized over $B$. The smash product in both is denoted by $\wedge_B$.
There are the usual suspension and loop space functors, with an adjunction
\[\xymatrix{ \Sigma^{\infty}_B \colon \hPTopb{G}{B} \ar@<1ex>@{->}[r] & \PSH{G}{B}\ar@<1ex>@{->}[l] \colon \Omega_B^\infty } .\]
We often abuse notation and write $X$ for the parametrized suspension spectrum $\Sigma^\infty_B X$.

 If $f\colon A \to B$ is a map of $G$-spaces, there are base-change adjunctions
\begin{align*}
\xymatrix{ f_! \colon \hPTopb{G}{A} \ar@<1ex>@{->}[r] & \hPTopb{G}{B} \ar@<1ex>@{->}[l] \colon f^* }  & &
\xymatrix{ f_! \colon \PSH{G}{A} \ar@<1ex>@{->}[r] & \PSH{G}{B} \ar@<1ex>@{->}[l] \colon f^* }
\end{align*}
with $f_! \circ \Sigma_A^\infty  =  \Sigma_B^\infty\circ f_!$.
For example, if  $p\colon X \to B$ is an ex-$G$-space over $B$ with section $s \colon B\to X$ and $\crush \colon B \to G/G$, then $\crush_!(X)=X/s(B)$.
The functors $f^*$ are strong symmetric monoidal, 
\[f^*(Y\wedge_B Z) \cong f^*(Y)\wedge_A f^*(Z)\]
while the functors $f_!$ satisfy
\begin{align}\label{eq:f!smash}
f_!(f^*Y \wedge_A X) \cong Y\wedge_B f_!X.
\end{align}
See \cite[Prop. 2.2.2 and Thm. 11.4.1]{MaySig}.

If $p\colon X \to B$ is a $G$-space over $B$, then $X_{+_B} = X \sqcup B$ denotes the ex-$G$-space obtained from $X$ by adding a disjoint section. We simply write $X_+  = X_{+_B}$ when the base space $B$ is clear from context. Note that any $G$-space $X$ can be viewed as a $G$-space over itself.  
If $p\colon X \to B$ is a $G$-space over $B$, then 
  \[p_!(X_{+_X}) = X_{+_B}.\]

\subsection{Cohomology} We now recall some facts about parametrized cohomology. Equivariant parametrized cohomology of an ex-$G$-space $X$ over $B$ is graded by the group of representations of the equivariant fundamental groupoid $RO(\Pi B)$ and takes coefficients in a parametrized $G$-Mackey functor. For each $\gamma \in RO(\Pi B)$, given an ex-$CW(\gamma)$-complex structure on $X$ and a parametrized $G$-Mackey functor $\mM$, there is an integer-graded cellular cochain complex $C^{\gamma + *}(X,\mM)$ computing the parametrized cellular cohomology of $X$ with coefficients in $\mM$ as
\[
\widetilde{H}^{\gamma + *}_B(X, \mM) := H^*(C^{\gamma + *}(X,\mM), d^{\gamma + *}).
\]
We described this cellular cochain complex explicitly and computed some examples in previous work, for more details see \cite{witpaper}. 

In this paper, we take a different approach and let $\gamma$ vary to consider $RO(\Pi B)$-graded parametrized cohomology of an ex-$G$-space $X$ over $B$. 
As noted in \cite[\S3.10]{CW_book}, when $\mM$ is a Mackey ring (i.e., a Green functor), letting $\gamma$ vary, the $RO(\Pi B)$-graded parametrized cohomology of an ex-$G$-space $X$ over $B$ assembles into a graded ring. We denote this ring by
\[
\widetilde{H}^{RO(\Pi B)}_B(X,\mM).
\]

The unreduced cohomology of a $G$-space $X$ over $B$ in degree $\gamma \in RO(\Pi B)$ is defined to be
\[H^{\gamma}_B(X , \mM):=\widetilde{H}^{\gamma}_B(X_{+_B} , \mM).
\]

When $B=G/G$, we have $RO(\Pi G/G)=RO(G)$ and parametrized cohomology coincides with ordinary $RO(G)$-graded Bredon cohomology. We do not distinguish between them and simply write
\begin{align}\label{eq:ROPAREQUAL}
 H^{\star}(X , \mM) & = H^{\star}_{G/G}(X , \mM)
\end{align}
for $\star \in RO(G)$.

\subsection{Representability}
Throughout, we will be using coefficients in $\mR$, the constant parametrized Mackey functor for a commutative unital ring $R$ as defined in \cite[Def. 3.41]{witpaper}. Thus, we will study the cohomology $\widetilde{H}^{RO(\Pi B)}_B(X,\mR)$ as a graded ring. 
As usual, this cohomology theory is representable.

For each $\gamma \in RO(\Pi B)$, there is a parametrized $G$-spectrum $H\mR^\gamma$  that represents the parametrized cohomology group in the sense that 
\[   \widetilde{H}^{\gamma}_B(X,\mR) = [  \Sigma_B^\infty X, H\mR^\gamma]^G_B,\] 
for any ex-$G$-space $X$ over $B$ \cite[\S3.7]{CW_book}. 
The right-hand side denotes maps in $\mathcal{SH}_B^G$. The fibers of the spectrum $H\mR^\gamma$ are non-parametrized equivariant Eilenberg--MacLane spectra $H\mR$.

The following result is not new; it is certainly implied by the work of Costenoble--Waner, e.g., \cite[Thm. 3.8.2]{CW_book}. However, we did not find a reference for the precise statement we use and so include a proof. We thank Steve Costenoble for providing some of the details.
 \begin{theorem} \label{thm:shriek-iso}
 Let $f \colon A \to B$ be a map of $G$-spaces. Let $R$ be a commutative unital ring and $\mR$ the constant parametrized Mackey functor. Let $X$ be an ex-$G$-space over $A$. Then there is a natural ring homomorphism
\[f_! \colon  \wH^{RO(\Pi B)}_B(f_!X, \mR) 
\to \wH^{RO(\Pi A)}_A(X, \mR) \]
that restricts to an isomorphism $f_! \colon  \wH^{\gamma}_B(f_!X, \mR) 
\xrightarrow{\cong} \wH^{f^*\gamma}_A(X, \mR)$  for each $\gamma \in RO(\Pi B)$. 
 \end{theorem}
\begin{proof}
 Let $H\underline{R}^\gamma$ be the representing spectrum for the cohomology in degree $\gamma$. From \cite[Prop. 3.8.3]{CW_book} we have
\[ f^*H\underline{R}^\gamma \cong (Hf^*\underline{R})^{f^*\gamma}. \]
The isomorphism $f_!$ comes from adjunction
\begin{align*}
 \wH^{\gamma}_B(f_!X, \mR) & \cong [f_!X,H\underline{R}^\gamma ]_B^G \\
&\cong [X,f^*H\underline{R}^\gamma ]_A^G \\
&\cong [X,(Hf^*\underline{R})^{f^*\gamma} ]_A^G.
\end{align*}
To show multiplicativity, consider the pairing 
\[\mu^{\alpha,\beta}_{\uR} \colon H\uR^\alpha \wedge_B H\uR^\beta \to H\uR^{\alpha+\beta} \]
inducing the cup product.
Since $f^*$ is closed symmetric monoidal \cite[Thm. 11.4.1]{MaySig}, we get a diagram
\[ \xymatrix@C=1pc{
f^*(H\underline{R}^\alpha\wedge_B H\underline{R}^\beta) \ar[d]^-{f^*(\mu^{\alpha,\beta}_{\uR})} &  f^*(H\underline{R}^\alpha)\wedge_A f^*(H\underline{R}^\beta) \ar[l]_-{\simeq} \ar[r]^-{\simeq} & (Hf^*\underline{R})^{f^*\alpha}\wedge_A (Hf^*\underline{R})^{f^*\beta} \ar[d]^-{\mu_{f^*\uR}^{f^*\alpha,f^*\beta}}
\\
f^*(H\underline{R}^{\alpha+\beta}) \ar[rr]^-{\simeq} &  &(H  f^*\underline{R})^{  f^*(\alpha+\beta)} 
}   \]
where the vertical maps are the maps of spectra giving the cup product. 
The cup product is characterized by its effect on parametrized homotopy groups \cite[\S 3.10]{CW_book}. Comparing the effect on homotopy, we see that 
the diagram commutes.
This proves the claim.
\end{proof}

\begin{cor}
    Let $p \colon X \to B$ be a $G$-space over $B$. Whenever the induced map $p^* \colon RO(\Pi B) \to RO(\Pi X)$ is injective, there is a natural isomorphism of rings
 \[p_! \colon  H^{RO(\Pi B)}_B(X, \mR) 
\xrightarrow[]{\cong} H^{p^*(RO(\Pi B))}_X(X, \mR). \]
\end{cor}

\begin{proof}
    This follows from \cref{thm:shriek-iso} using that $p_! (X_{+X}) = X_{+B}$.
\end{proof}

Applying \cref{thm:shriek-iso} to $\crush \colon B \to G/G$ allows us to identify the $RO(G)$-graded portion of parametrized cohomology. We get the following  isomorphism, which will be used implicitly throughout the rest of the paper. 

\begin{cor}\label{cor:changeofbases}
 Let $X$ be an ex-$G$-space over $B$ and $\alpha\in RO(G)$. There is a natural ring homomorphism
  \[\wH^{RO(G)}(X/s(B), \mR) \to \wH^{\crush^*(RO(G))}_B(X, \mR) \]
  where the left-hand side is the  $RO(G)$-graded 
  cohomology of $X/s(B)$. 
  
  When $\crush^* \colon RO(G) \to RO(\Pi B)$ is injective (for example whenever $B$ has a fixed point as in \cref{prop:ROGinjects}) there is a natural ring isomorphism
  \[\wH^{RO(G)}(X/s(B), \mR) \xrightarrow{\cong} \wH^{\crush^*(RO(G))}_B(X, \mR). \]
 \end{cor}

\begin{proof}
    This follows from \cref{thm:shriek-iso} using that $\crush_!(X) = X/s(B)$.
\end{proof}
 
Thus, when $B$ has a fixed point, we have that the $RO(G)$-graded portion of the parametrized cohomology of $X$ with constant coefficients is the usual $RO(G)$-graded cohomology of the quotient $X/s(B)$. In particular, taking $X = B_{+B}$ with a fixed point, there is a natural ring isomorphism
  \[\wH^{RO(G)}(B_+, \mR) \cong \wH^{\crush^*(RO(G))}_B(B_{+_B}, \mR), \]
  where the left-hand side is the $RO(G)$-graded cohomology of the $G$-space $B$ with a disjoint basepoint. Equivalently, there is a natural ring isomorphism of unreduced cohomology
  \[H^{RO(G)}(B, \mR) \cong H^{\crush^*(RO(G))}_B(B, \mR). \]

\section{The Thom isomorphism}\label{sec:thomiso}
One of the most striking properties of equivariant parametrized cohomology is the presence of a Thom isomorphism for vector bundles in any coefficients. We spell out some of the details below since the Thom isomorphism plays a central role in our computations.

\subsection{Relative Thom space} 
If $\xi \colon E \to B$ is an (actual) orthogonal $G$-vector bundle, then the $B$-relative Thom space $\Th_B(\xi)$ is the space obtained from $E$ by fiberwise one-point compactification. We write $\xi_b$ for the fiber over $b\in B$.
The relative Thom space is an ex-$G$-space over $B$, with section the inclusion of $B$ at infinity. We also denote by $\Th_B(\xi)$ its suspension spectrum $\Sigma^\infty_B\Th_B(\xi)$ in $\mathcal{SH}^G_B$. 

The relative Thom space satisfies some nice properties. For example, given bundles $\xi$ and $\eta$, we have
\[\Th_B(\xi)\wedge_B \Th_B(\eta) \simeq \Th_B(\xi\oplus \eta), \]
since $S^{\xi_b}\wedge S^{\eta_b} \simeq S^{\xi_b+\eta_b}$ for each $b\in B$. 

The relative Thom space is closely related to the classical Thom space, which is key to translating between the parametrized and non-parametrized contexts.  In particular, we have 
\[\crush_!( \Th_B(\xi)) = \Th(\xi),\]
where the right hand side is the usual Thom space, obtained from $\Th_B(\xi)$  by collapsing the section. The relative Thom space also respects pullbacks. That is, if $f\colon A\to B$ is a $G$-map, we have
\[f^*(\Th_B(\xi)) = \Th_A(f^*\xi).\]

We can use the relative Thom space to define suspensions and loops with respect to bundles. 

\begin{defn}
For any parametrized spectrum $E$ over $B$, there is a $G$-spectrum over $B$
\[\Sigma_B^{\xi}E := E\wedge_B  \Th_B(\xi),\]
that is the fiberwise 
smash product of $E$ with $\Th_B(\xi)$, with fibers $E_b\wedge S^{\xi_b}$.
There is also a fiberwise function $G$-spectrum over $B$
 \[\Omega_B^{\xi}E := F_B(\Th_B(\xi),E),\]
whose fibers are the function spectra $F(S^{\xi_b}, E_b) \simeq \Sigma^{-\xi_b} E_b$.  
\end{defn}
 The functors $\Omega_B^{\xi}$ and $\Sigma_B^{\xi}$ are adjoint equivalences in the homotopy category of parametrized spectra $\mathcal{SH}^G_B$. For this reason, we let 
 \[\Sigma_B^{-\xi}:=\Omega_B^\xi,\]
 allowing us to make sense of suspension by a virtual bundle. 
 Therefore,
 \[[X,Y]^G_B \simeq [\Sigma^\xi_B X,\Sigma^\xi_B Y]^G_B\]
in $\mathcal{SH}^G_B$ for any orthogonal $G$-bundle $\xi$, including the virtual case.

It will be useful to make the functor $\Omega^\xi_B(-)$ more concrete and exhibit it as a relative smash product. 
For this, we can use the Madsen--Tillmann construction. Let $\VV$ be the regular real representation of $G$ so that $\VV^{\infty} = \colim_n \VV^{n}$ is a complete universe. Then 
\begin{equation}\label{eq:Grd}
Gr_d:=Gr_d(\VV^{\infty}) \simeq B_GO(d)
\end{equation}
classifies $G$-bundles of dimension $d$.
Let $\gamma$ be the universal $d$-plane $G$-bundle over $Gr_d$.

Recall we have assumed that  $B$ is a $G$-CW complex of finite type and so it is a paracompact Hausdorff space. Therefore, isomorphism classes of $G$-bundles of dimension $d$ over $B$ are classified by $G$-homotopy classes of maps to 
$B_GO(d)$.

\begin{defn}\label{defn:GrdANDthom}
Suppose that $d<n|G|$. 
\begin{enumerate}[(a)]
\item Let $Gr_d^n=Gr_d(\VV^{n})$ and  
\[\iota\colon Gr_d^n \xrightarrow{\subset} Gr_d\]
be the inclusion.
Let
$\gamma^n = \iota^*\gamma$ be the pullback of the universal bundle $\gamma$ along $\iota$. Then $\gamma^n$ is a subbundle of the trivial bundle with fibers
$\VV^{n}$ over $Gr_d(\VV^{n})$. 

\item Let $\nu^n$ be the orthogonal complement of $\gamma^n$ and define 
\[\Th_{Gr_d^n}(-\gamma^n):=\Sigma^{-\VV^{n}}_{Gr_d^n}\Th_{Gr_d^n}(\nu^n) \in \mathcal{SH}_{Gr_d^n}^G\]
and
\[\Th_{Gr_d}(-\gamma^n):=\iota_!\Th_{Gr_d^n}(-\gamma^n) \in \mathcal{SH}_{Gr_d}^G.\]

\item The restriction of $\nu^{n+1}$ over $Gr_d^{n+1}$ to $Gr_d^n$ along the inclusion is isomorphic to $\nu^n\oplus \VV$.
So the inclusions induce maps 
\[\xymatrix@R=1pc{\Th_{Gr_d}(-\gamma^n)\ar@{=}[d] \ar[r] & \Th_{Gr_d}(-\gamma^{n+1})
\ar@{=}[d] \\
\Sigma^{-\VV^{ n}}_{Gr_d}\Th_{Gr_d}(\nu^{n}) \ar[r] & \Sigma^{-\VV^{n+1}}_{Gr_d}\Th_{Gr_d}(\nu^{n+1}). } \]
Let
\[\Th_{Gr_d}(-\gamma) = \colim_n\Th_{Gr_d}(-\gamma^n) \in \mathcal{SH}^G_{Gr_d}.\]
\end{enumerate}
\end{defn}

\subsection{Invertibility} 
As we see below, the relative Thom space is invertible in the stable homotopy category.
\begin{defn}
Suppose that $\xi$ is an (actual) orthogonal $G$-bundle over $B$ classified by
$f_{\xi} \colon B \to Gr_d(\VV^{ \infty})$,
so that $f_{\xi}^*\gamma \cong \xi$. Define
\begin{align*}
\Th_{B}(-\xi) &= f_\xi^*\Th_{Gr_d}(-\gamma)  \in  \mathcal{SH}^G_B\\
\Th(-\xi)&=\rho_!(\Th_{B}(-\xi)) \in \mathcal{SH}^G.
\end{align*}
\end{defn}

\begin{lem}
If $\xi$ is a $G$-bundle over $B$, then
\[\Omega^\xi_B(E) = \Sigma^{-\xi}_B(E) \simeq E\wedge_B \Th_B(-\xi)\]
in $\mathcal{SH}^G_B$. In particular,
\[\Th_B(\xi)  \wedge_B \Th_B(-\xi) \simeq B_+ \]
and $\Th_B(\xi)$ is invertible in the symmetric monoidal category $(\mathcal{SH}^G_B, \wedge_B,B_+)$.
\end{lem}
\begin{proof}
Using that $f^*_\xi \colon \mathcal{SH}^{G}_{Gr_d} \to \mathcal{SH}^{G}_{B}$ is strong symmetric monoidal  we have
\begin{align*}
\Th_B(\xi) \wedge_B \Th_B(-\xi) 
&\simeq  f_\xi^*( \Th_{Gr_d}(-\gamma)\wedge_{Gr_d} \Th_{Gr_d}(-\gamma) )\\
 &\simeq  f^*_\xi\left(\colim_n  \iota_!\Th_{Gr_d^n}(\gamma^n)\wedge_{Gr_d^n} \iota_!\Th_{Gr_d^n}(-\gamma^n)\right) \\
  &\simeq f^*_\xi\left(\colim_n \iota_!\left( \iota^*\iota_!\Th_{Gr_d^n}(\gamma^n)\wedge_{Gr_d^n} \Th_{Gr_d^n}(-\gamma^n)\right)\right) \\
      &\simeq f^*_\xi\left(\colim_n \iota_!\left( \Th_{Gr_d^n}(\gamma^n)\wedge_{Gr_d^n} \Th_{Gr_d^n}(-\gamma^n)\right)\right) \\
          &\simeq f^*_\xi\left(\colim_n \iota_! \Sigma^{-\VV^n}_{Gr_d^n}\Th_{Gr_d^n}(\gamma^n \oplus \nu^n)\right) \\
          &\simeq f^*_\xi\left(\colim_n \iota_!( {Gr_d^n}_+)\right) \\
    &\simeq f^*_\xi ({Gr_d}_+) \simeq B_+.
\end{align*}
The third equivalence is an instance of \eqref{eq:f!smash} as in \cite[Thm. 11.4.1]{MaySig} and the fourth uses \cite[Rem. 11.4.7]{MaySig} to deduce that
\[\iota^*\iota_! \Sigma^\infty_{Gr_d^n} \Th_{Gr_d^n}(\gamma^n) \simeq \Sigma^\infty_{Gr_d^n} \iota^*\iota_! \Th_{Gr_d^n}(\gamma^n)\simeq  \Sigma^\infty_{Gr_d^n}\Th_{Gr_d^n}(\gamma^n) . \] 
So indeed $(-)\wedge_B \Th_B(-\xi)$ is an inverse for $(-)\wedge_B \Th_B(\xi)$ and the claim follows.
\end{proof}

\begin{rem}
The same proof shows that in  $\mathcal{SH}^G$
\[\Th(\xi)\wedge \Th(-\xi) \simeq B_+.\]
\end{rem}

\begin{cor}\label{cor:HRgammaHThom}
Let $R$ be a ring and $ \mR=\rho^*\mR$ be the constant parametrized Mackey functor on $R$.
For $\gamma\in RO(\Pi B)$, and
$\xi$ an actual bundle,
\begin{align*}
H\mR^{\gamma+\xi} &\simeq \Sigma_B^\xi  H\mR^\gamma \simeq H\mR^\gamma \wedge_B \Th_B(\xi) \\
H\mR^{\gamma-\xi} &\simeq \Sigma_B^{-\xi}  H\mR^\gamma \simeq H\mR^\gamma \wedge_B \Th_B(-\xi) .
\end{align*}
\end{cor}
\begin{proof}
This follows from the characterization of the Eilenberg--MacLane spectrum $H\mR^{\gamma\pm \xi}$ from \cite[\S3.7]{CW_book} as determined by its fibers.
\end{proof}

\subsection{The parametrized Thom isomorphism}
An important selling point of parametrized cohomology is the existence of Thom isomorphisms for every virtual 
bundle and coefficients. This can be thought of as a generalization of the classical fact that, so long as one uses cohomology with local coefficients, there is always a Thom isomorphism. 
\begin{theorem}[Parametrized Thom isomorphism {\cite[Thm. 3.11.3]{CW_book}}]\label{thm:thom}
Let $\gamma$ be a virtual $G$-bundle over $B$. For any $\alpha \in RO(\Pi B)$, there
is an isomorphism 
\[ {H}^{\alpha}_B(B, \mR) \xrightarrow[\cong]{t_\gamma} \widetilde{H}^{\alpha+\gamma}_B(\Th_B(\gamma), \mR).\]
Identifying $H\mR^\gamma$ with $H\mR \wedge_B \Th_B(\gamma)$ as in \cref{cor:HRgammaHThom}, the isomorphism is realized by
multiplication with the class 
\[t_\gamma \colon \Th_B(\gamma) \to   H\mR \wedge_B \Th_B(\gamma),\] 
given by the external smash product of the unit of $S^0 \to H\mR$ with $ \Th_B(\gamma)$.
\end{theorem}
\begin{proof}
We have
\begin{align*}
[B_+,  H\mR^\alpha]^G_B &\cong [\Sigma_B^\gamma B_+,  \Sigma_B^\gamma H\mR^\alpha]^G_B \\
&\cong [\Th_B(\gamma), H\mR^{\alpha+\gamma}]^G_B. \qedhere
\end{align*}
\end{proof}

The Thom isomorphism and the functor $\rho_!$ allow us to relate parametrized cohomology with $RO(G)$-graded cohomology of Thom spaces.
\begin{prop}\label{prop:cohBThom}
Let $\gamma$ be a virtual orthogonal $G$-bundle over $B$.
Let $R$ be a ring and $ \mR=\rho^*\mR$ be the constant parametrized Mackey functor on $R$. For $\star\in RO(G)$,
there are isomorphisms 
\begin{align*}
\phi \colon H^{\star-\gamma}_B(B,\mR) &\xrightarrow{\cong} \widetilde{H}^{\star}(\Th(\gamma),\mR) 
\end{align*}
where the right-hand side is $RO(G)$-graded cohomology.
\end{prop}
\begin{proof}
If $V\in RO(G)$ and we compose the Thom isomorphism of \cref{thm:thom} with the isomorphism of \cref{thm:shriek-iso}, we recover the isomorphism of \cref{prop:cohBThom}
\[ \phi \colon {H}^{V -\gamma}_B(B, \mR) \xrightarrow[\cong]{t_\gamma} \widetilde{H}^{V}_B(\Th_B(\gamma), \mR)\xrightarrow[\cong]{\rho_!}  \widetilde{H}^{V}(\Th(\gamma), \mR) .\qedhere\]
\end{proof}

For any virtual bundles $\gamma_i$ where $i=1,2$, we have pairings in $RO(G)$-graded cohomology
\begin{align}\label{eq:thompair} \xymatrix@C=2pc{\widetilde{H}^{V_1}(\Th(\gamma_1),\mR)\otimes \widetilde{H}^{V_2}(\Th(\gamma_2),\mR) \ar[r]^-{\smile} \ar[d]_-{\times}& \widetilde{H}^{V_1+V_2}(\Th(\gamma_1+\gamma_2),\mR)  \\ 
 \widetilde{H}^{V_1+V_2}(\Th(\gamma_1)\wedge \Th(\gamma_2),\mR) \ar[ur]_-{\Delta^*}
} 
\end{align}
given by composing the cross-product with the pullback along the diagonal $\Delta \colon  B \to B\times B$, using that
\[ \Th(\gamma_1 + \gamma_2)=\Delta^*(\Th(\gamma_1 \times \gamma_2)) \to \Th(\gamma_1 \times \gamma_2) = \Th(\gamma_1)\wedge \Th(\gamma_2).  \]

The Thom isomorphisms of \cref{prop:cohBThom} are multiplicative. The proof of the following result is a tedious diagram chase using the adjunction $(\rho_!,\rho^*)$ and the fact that $\crush^*$ is strong symmetric monoidal.
\begin{lem}\label{lem:phimultiplication}
Let $x_1,x_2\in H^{RO(\Pi B)}_B(B,\mR)$
 of degrees
\[|x_i| = \gamma_i + V_i\]
for $\gamma_i$ virtual bundles and $V_i\in RO(G)$.
Then for $\phi$ the Thom isomorphism as in \cref{prop:cohBThom},
\[\phi(x_1 x_2) =\phi(x_1)\phi(x_2) \in  \widetilde{H}^{V_1+V_2}(\Th(\gamma_1+\gamma_2),\mR).\]
\end{lem}

We summarize the main result of this section giving a multiplicative Thom isomorphism from parametrized cohomology to $RO(G)$-graded cohomology, the main tool for our computations. 

\begin{thm}\label{thm:KOPIB}
For $\gamma \in KO(\Pi B)$ and $\star \in RO(G)$, there is an isomorphism 
\[\phi \colon H^{\gamma+\star}_B(B,\mR) \xrightarrow{\cong} \widetilde{H}^{\star}(\Th(-\gamma),\mR).\]
The cup products in parametrized cohomology  correspond to the
pairings between Thom spectra   obtained by composing the external cross product with pullback along the diagonal. That is, the following diagram commutes
\[\xymatrix{
H^{\gamma_1+\star}_B(B,\mR) \otimes H^{\gamma_2+\star}_B(B,\mR) \ar[r]^-{\smile} \ar[d]^-\cong_-\phi&   H^{\gamma_1+\gamma_2+\star}_B(B,\mR)\ar[d]^-\cong_-\phi
\\
\widetilde{H}^{\star}(\Th(-\gamma_1),\mR)\otimes \widetilde{H}^{\star}(\Th(-\gamma_2),\mR) \ar[r]^-{ \Delta^* \circ \times} & \widetilde{H}^{\star}(\Th(-\gamma_1-\gamma_2),\mR) .
}\]
\end{thm} 

Applying the Thom isomorphism of \cref{thm:KOPIB} to a bundle with $\dim(\gamma) = V$ (see \cref{rem:non-trivial-but-dim-V}) will give an $RO(G)$-graded Thom isomorphism. 
\begin{cor}\label{thm:ROGThomdimVbundles}
Let $\gamma$ be a virtual bundle such that $\dim(\gamma) = V$ for $V$ in $RO(G)$.  Then there is a Thom isomorphism in $RO(G)$-graded cohomology
\[H^{\star}(B,\mR) \xrightarrow{\cong} \widetilde{H}^{\star+V}(\Th(\gamma),\mR).\]
\end{cor}

\begin{proof}
    We apply \cref{thm:KOPIB} to the bundle $-\gamma$ whose degree is $-\dim(\gamma)=-V$ to get
  \[H^{\star}_B(B,\mR) \xrightarrow{\cong} \widetilde{H}^{\star+V}(\Th(\gamma),\mR)\]
  and recall that if $B$ has a fixed point then in $RO(G)$ degrees 
  \[H^{\star}_B(B,\mR) \cong H^{\star}(B,\mR). \qedhere \]
  \end{proof}

\begin{rem}
Of course, it is not possible to have an $RO(G)$-graded Thom isomorphism for a bundle $\gamma$ if we cannot make sense of the dimension of $\gamma$ in $RO(G)$. Thus, \cref{thm:KOPIB} does not imply an $RO(G)$-graded Thom isomorphism (on the nose) for bundles with different representations over different fixed points. See for example the $C_2$-Möbius bundle on $S^{1,1}$ as described in \cite[Example 3.4]{Hazel_fund} and \cite[Example 2.34]{witpaper}. 
\end{rem}

\begin{rem}
The condition that $\dim(\gamma)=V$ is very restrictive. It is known there are Thom isomorphisms in $RO(G)$-graded cohomology for \emph{orientable} homogeneous bundles, where here we mean orientable in the sense of \cite[Def. 2.8]{CMW}. See \cite[Thm. 1.15.5]{CW_book} for Burnside coefficients or \cite[Thm. 1.4]{BhattZou} for a general treatment. Neither \cref{thm:KOPIB} nor \cref{thm:ROGThomdimVbundles} require orientability.
\end{rem}

Specializing to $\mF$-coefficients, we get an $RO(G)$-graded Thom isomorphism for homogeneous bundles using the theory of \cite{BhattZou}, again without any requirement of orientability. 

The following result is implicit in \cite{BhattZou} and we simply explain how to put the ingredients together. 
For $G=C_2$, this is also shown in \cite[Thm. 3.16]{Hazel_fund}.  We suspect this result has been known longer, but we were not able find an original reference. 
\begin{theorem}
\label{thm:homothom}
Let $G$ be a finite group. 
Suppose the
underlying space of $B$ is path connected and that $B$ has a fixed point.
Let $\xi$ be a virtual homogeneous bundle over $B$ with virtual fibers $\xi_0$. Then $\xi$ has a Thom class $\bar t_{\xi} \in \widetilde{H}^{\xi_0}(\Th(\xi),\mF)$ in $RO(G)$-graded cohomology and there is a Thom isomorphism
\[ H^{\star}(B,\mF) \xrightarrow[\cong]{\bar t_\xi} \widetilde{H}^{\star+\xi_0}(\Th(\xi),\mF)\]
obtained by multiplication by the Thom class.
\end{theorem}
\begin{proof}
We start with the case of $\xi$ an actual bundle.
This proof is not meant to be readable without absorbing the relevant sections of \cite{BhattZou}. 
 In the notation of \cite{BhattZou}, we work with $R=H\mF$. A homogeneous bundle is $R$-homogeneous in the sense of Definition 2.24 of \cite{BhattZou}. Further, the relative dimension of a homogeneous bundle with respect to $R$ is $\mathcal{I} = R \wedge S^{\xi_0}$. So, by \cite[Theorem 1.4]{BhattZou}, $\xi$ has an $R$-Thom class in $[\Th(\xi), \Sigma^{\xi_0} R]^G$ if and only if $\xi$ has an $R$-orientation in the sense of Definition 2.31 of \cite{BhattZou}. Each homogeneous bundle has  a first Stiefel--Whitney class relative to $R$, as in Definition 3.2 of \cite{BhattZou}, which is an element 
\[w_1^R(\xi) \in [B_+, \mathrm{bgl}_1(R)]^G. \]
By Theorem 1.15 of \cite{BhattZou}, $\xi$ is $R$-orientable if and only $w_1^R(\xi)=0$.
Now, specializing to $R=H\mF$, since $\F_2^\times =\{1\}$,  it follows that $\mathrm{bgl}_1(H\mF)$ is contractible. Thus, for any homogeneous bundle $\xi$, the Stiefel--Whitney class $w_1^{H\mF}(\xi)=0$ and so $\xi$ is $H\mF$-orientable. 

The argument for virtual bundles is standard, using the fact that cohomology behaves well under colimits.
\end{proof}

\section{Characteristic classes}\label{sec:char}
In this section, we define certain characteristic classes and units that arise from the relationship of parametrized cohomology with the cohomology of Thom spectra. 

\subsection{Euler classes}\label{sec:Euler}
The Euler classes in the parametrized context are defined in the expected way. We describe their construction below. See also \cite{CW_book} and \cite{CostenobleB}.

Let $\xi$ be an actual bundle over $B$ and consider the map of ex-$G$-spaces over $B$ 
\[\iota \colon B_+ =B\sqcup s(B)\to \Th_B(\xi)  \]
 sending $B$ to the zero section and $s(B)$ to the section at infinity. This gives rise to a stable map
\begin{align}\label{eq:agamma}\xymatrix{  B_+ \ar[r]^-{\iota} &   \Th_B(\xi) \simeq S^0 \wedge  \Th_B(\xi) \ar[r]^-{i \wedge \id} & H\underline{\mathbb{Z}} \wedge  \Th_B(\xi) \simeq   H\underline{\mathbb{Z}}^{\xi} } \end{align}
where the second arrow is the external smash product of the unit of $ H\underline{\mathbb{Z}}$ with the identity of $\Th_B(\xi)$.

\begin{defn}[Euler Class]\label{defn:euler}
The composite  \eqref{eq:agamma} is  an element 
\[a_\xi \in [  B_+ ,  H\underline{\mathbb{Z}}^{\xi}]^{G}_B \cong H^{\xi}_B(B,\underline{\mathbb{Z}}),\]
which we call the \emph{Euler class} of $\xi$. We also denote by $a_\xi$ its image in $H^\xi_B(B , \mR) $ for any ring $R$.
\end{defn}

\begin{rem}
Our notation for the Euler class comes from \cite[Definition 3.11]{HHR}, since $a_\xi$ is a direct analogue of the non-parametrized Euler classes $a_V$. If $B=G/G$ and $\xi \colon V\times G/G \to G/G$ for $V$ an orthogonal representation of $G$, the definitions agree. It also follows from the next lemma that if $\xi = B\times V \to B$ is the trivial bundle, then $a_\xi = \rho^*a_V$ for $\rho \colon B \to G/G$, and so in this case we abuse notation and simply write $a_V$. 
\end{rem}

The Euler classes satisfy a number of nice properties.
\begin{lem}\label{lem:eulerprop}
Let $\xi$ and $\gamma$ be (actual) orthogonal $G$-vector bundles over $B$, and let $H$ be a subgroup of $G$. 
\begin{enumerate}[(a)]
\item  Under the restriction
\[ i^*_H(a_\xi) = a_{i^*_H\xi}.\]
\item If $f\colon A \to B$ is a $G$-map, then
\[f^*(a_\xi) = a_{f^*\xi},\]
where the two sides are compared using the identification
\[ f_! \colon H^{\xi}_B(A,\uZ)\xrightarrow{\cong} H^{f^*\xi}_A(A,\uZ)  .\]
\item The Euler classes are multiplicative, satisfying the relation 
\[a_{\xi+\gamma} = a_{\xi}a_\gamma .\]
\item For $\crush\colon B\to G/G$, $\crush_!(a_\xi)$ is the element of $[B_+, \Th(\xi)\wedge H\uZ]^G$ induced by the zero section.
\end{enumerate}
\end{lem}
\begin{proof}
Part (a) is clear from the definition. Part (b) is shown by examining the composite
\[\xymatrix{[ B_{+} ,  H\underline{\mathbb{Z}}^{\xi}]^{G}_B \ar[r]^-{-\circ f_+}  & [ A_{+} ,  H\underline{\mathbb{Z}}^{\xi}]^{G}_B = [ f_!(A_+) ,  H\underline{\mathbb{Z}}^{\xi}]^{G}_B   \ar[r]^-{f_!}_-\cong & [ A_+ ,  H\underline{\mathbb{Z}}^{f^*\xi}]^{G}_A
} .\]  
In this equation, the first appearance of $A_+$ denotes $\Sigma^\infty_B(A \sqcup B)$ in the category of spectra over $B$. The second and third appearances mean $\Sigma^\infty_A(A \sqcup A)$ in the category of spectra over $A$. The equal sign is due to the fact that $f_!(A\sqcup A) = A\sqcup B$. For the last isomorphism, we use the isomorphism 
\[ [ f_!(A_+) ,  H\underline{\mathbb{Z}}^{\xi}]^{G}_B  =[ f_!(A_+) ,  \Th_B(\xi)\wedge_B H\underline{\mathbb{Z}}]^{G}_B  \cong [ A_+ ,  f^*(\Th_B(\xi)\wedge_B H\underline{\mathbb{Z}})]^{G}_A \]
coming from the $(f_!,f^*)$ adjunction, together with the fact that $f^*$ is symmetric monoidal and that $f^*(\Th_B(\xi)) = \Th_A(\xi)$ as in \cref{thm:shriek-iso}.

For part (c), note that 
\[\Th_B(\xi)\wedge_B \Th_B(\gamma) \cong \Th_B(\xi \oplus \gamma) ,\]
with section $B_+ \xrightarrow{\iota_{\xi\oplus \gamma}}\Th_B(\xi \oplus \gamma)$ identified with the smash product of the sections for $\xi$ and $\gamma$.
For part (d), we are applying the functor $\crush_!$ to the Euler class to get a map
\[\crush_!(a_\xi) \colon \crush_!(B_+) \to \crush_!(\Th_B(\xi)\wedge H\uZ). \]
Using the properties of the adjunction $(\crush_!,\crush^*)$ as in 
\cite[Thm. 11.4.1]{MaySig}, with the fact that, by definition, $H\uZ$ over $B$ means $\crush^*H\uZ$, we have
\[\crush_!(\Th_B(\xi)\wedge_B \crush^*H\uZ)  \simeq \crush_!(\Th_B(\xi) ) \wedge H\uZ \simeq \Th(\xi)\wedge H\uZ .\]
We see that
$\crush_!(a_\xi)$ is 
 the element of $[B_+, \Th(\xi)\wedge H\uZ]^G$ induced by the zero section. 
\end{proof}

\subsection{Orientation classes}\label{sec:orientation}

We next turn to orientation classes. In order to define these classes, we define a notion of $R$-orientability for a ring $R$. We start by discussing the nonequivariant case. 

Let $G=e$ and $\xi$ be an actual bundle over a nonequivariant path connected space $B$.  Denote by $|\xi|$ the underlying dimension of the bundle. Recall that classically
\[\widetilde{H}^{|\xi|}(\Th(\xi), R)\cong H^{0}(B, R_{\xi}) \cong  H^{0}(B, R_{-\xi}) \cong  \widetilde{H}^{-|\xi|}(\Th(-\xi), R),\]
where $R_{\xi}$ is the local coefficient system on $R$ determined by the bundle $\xi$.

We can translate the classical setting to parametrized cohomology. Let $b\colon \pt \to B$ and $\xi_b$ be the fiber over $b$.  By \cref{prop:cohBThom}, we have a commutative diagram 
\begin{align}\label{eq:noneqface}
\xymatrix{ H^{\xi-|\xi|}_B(B, \mR) \ar[r]^-{b_!b^*} \ar[d]_-\cong &   H^{\xi_b-|\xi|}(\pt, \mR) \ar[d]^-\cong \\
\widetilde{H}^{-|\xi|}(\Th(-\xi), R) \ar[r]^-{\Th(b)^*} & \widetilde{H}^{-|\xi|}(\Th(-\xi_b), R) 
}
\end{align}
where the top arrow is the composition of the induced map on cohomology $b^*$ followed by the isomorphism $b_!$ from \cref{thm:shriek-iso}
\[\xymatrix{ H^{\xi-|\xi|}_B(B, \mR) \ar[r]^-{b^*}  & H^{\xi-|\xi|}_B(\pt, \mR)  \ar[r]^{b_!}_-\cong &    H^{\xi_b-|\xi|}(\pt, \mR). }  \]
 The bottom horizontal arrow $\Th(b)^*$ is the map induced by the inclusion $S^{-\xi_b} \simeq \Th(-\xi_b) \subset \Th(-\xi)$. 
 
 Then classically, the bundle $\xi$ is $R$-orientable if and only if one of the following equivalent conditions holds (where in (c) we use path connectedness of $B$):
 \begin{enumerate}[(a)]
 \item $\widetilde{H}^{-|\xi|}(\Th(-\xi), R)\cong R$,
  \item $\Th(b)^*$ is an isomorphism for all $b \in B$,
  \item $\Th(b)^*$ is an isomorphism for some $b\in B$.
 \end{enumerate}
This translates in terms of parametrized cohomology as follows. 
\begin{lem}\label{lem:orientabilitynoneq}
Let $\xi$ be an actual bundle over a nonequivariant, path connected space $B$. The bundle $\xi$ is $R$-orientable if and only if one of the following equivalent conditions hold:
\begin{enumerate}[(a)]
\item $H^{\xi-|\xi|}_B(B,\mR) \cong R$,
\item $b_!b^*$ is an isomorphism for all $b \colon \pt \to B$,
\item $b_!b^*$ is an isomorphism for some $b \colon \pt \to B$.
\end{enumerate}
\end{lem}

We generalize to the $G$-equivariant parametrized context in a way that clearly agrees with the classical definition when $G=e$. Our definition is inspired by the first condition of \cref{lem:orientabilitynoneq}.
\begin{defn}\label{defn:orientation_class}
Let $B$ be a $G$-space whose underlying space is path connected and which has a fixed point. Let $\xi$ be an actual orthogonal $G$-bundle over $B$. Let $R$ be an ordinary commutative ring and $\uR$ the associated constant parametrized Mackey functor.
We say that the $G$-bundle $\xi$ is \emph{$R$-orientable} 
if the underlying bundle $i^*_e\xi$ is $R$-orientable and the restriction
\[i^*_e\colon  H^{\xi-|\xi|}_B(B, \mR) \to H^{i^*_e\xi-|\xi|}_{i^*_eB} (i^*_eB, \uR) \cong R\]
is an isomorphism. 
Any class $u_\xi $ in $H^{\xi-|\xi|}_B(B, \mR)$ which restricts to an $R$-module generator will be called an \emph{orientation class}.
\end{defn}

\begin{rem}
This differs from the notion of orientability for equivariant bundles given in Costenoble--May--Waner \cite[Def. 2.8]{CMW} and $R$-orientability in \cite[Def. 2.31]{BhattZou}, as well as strict orientability (of $V$-bundles) given in \cite[Def. 1.15.4]{CW_book}.  
Our aim with this notion is to define orientation classes like those for $RO(G)$-graded cohomology in \cite{HHR}, not to recover all the classical notions of orientability, say for manifolds.
\end{rem}

\begin{rem}
We do not need parametrized cohomology to phrase our definition of $R$-orientability. Let $B$ and $\xi$ be as in \cref{defn:orientation_class}. Using \cref{thm:KOPIB}, we can equivalently define a bundle $\xi$ over $B$ to be $R$-orientable if $i^*_e\xi$ is orientable over $i^*_eB$ and the restriction
\[i^*_e\colon \widetilde{H}^{-|\xi|}(\Th(-\xi),\mR) \to \widetilde{H}^{-|\xi|}(\Th(-i^*_e\xi),R)\cong R \]
from $RO(G)$-graded cohomology to the underlying singular cohomology is an isomorphism. An orientation class is then a choice of $R$-module generator in $\widetilde{H}^{-|\xi|}(\Th(-\xi),\mR)$.
\end{rem}

\begin{example}\label{ex:uVclasses}
Let $B=G/G$ and let $R$ be a ring. Let $\mR$ be the constant Mackey functor on $R$. An orthogonal $G$-bundle over the point is just an orthogonal $G$-representation $V$. Over the point, \cref{defn:orientation_class} is equivalent to $V$ being $R$-orientable as a $G$-representation and our orientation classes $u_V$ recover those in \cite[Def. 3.12]{HHR}.

Indeed, recall that in this case parametrized and non-parametrized $RO(G)$-graded cohomology coincide, see \eqref{eq:ROPAREQUAL}. An orthogonal $G$-bundle over the point is just an orthogonal $G$-representation $V$. We have a commutative diagram 
\begin{align}\label{eq:faceoverpoint}
\xymatrix{ 
H^{V-|V|}(G/G , \mR)  \ar[r]  \ar[d]_-\cong  & H^{i^*_eV-|V|}(i^*_e(G/G), R) \ar[d]^-\cong \\
\widetilde{H}^{-|V|}(S^{-V}, \mR) \ar[r] & \widetilde{H}^{-|V|}(i^*_eS^{-V}, R) 
}
\end{align}
where, noting that $S^{-V} = \Th(-V)$, the vertical arrows coincide with the isomorphisms of \cref{prop:cohBThom}.
In the diagram, the left-hand side is $RO(G)$-cohomology and the right-hand side is the underlying singular cohomology. The top horizontal map is an isomorphism if and only if  
\[\xymatrix{\widetilde{H}^{-|V|}(S^{-V}, \mR) \cong \widetilde{H}_{|V|}(S^V, \mR)\ar[r] & \widetilde{H}_{|V|}(i^*_eS^{V}, R) \cong \widetilde{H}^{-|V|}(i^*_eS^{-V}, R) }\]
is an isomorphism. It follows from \cite[Ex. 3.8 \& 3.10]{HHR} that this map is always an injection, and is an isomorphism if and only if $V$ is $R$-orientable as a $G$-representation in the sense that the action of $G$ on $H_{|V|}(i^*_eS^{V},R) \cong R$ is trivial.  Indeed, we have a commutative diagram with exact rows
\[\xymatrix{0 \ar[r]  & \widetilde{H}_{|V|}(S^V, \mR)\ar[r] \ar[d] & C^{\mathrm{cell}}_{|V|}(S^V,\mR)^G \ar[r] \ar[d]^-\subset & C^{\mathrm{cell}}_{|V|-1}(S^V,\mR)^G \ar[d]^-\subset\\
0\ar[r] & 
\widetilde{H}_{|V|}(i^*_eS^{V}, R) \ar[r] & C^{\mathrm{cell}}_{|V|}(S^V,\mR) \ar[r]  & C^{\mathrm{cell}}_{|V|-1}(S^V,\mR)
}\]
for $C^\mathrm{cell}_*(S^V,\mR)$ the complex of $R[G]$-modules as in Example 3.8 of \cite{HHR}.
Therefore, $V$ is $R$-orientable in the sense of \cref{defn:orientation_class} if and only if $V$ is an $R$-orientable $G$-representation, and we recover the definition of orientation classes $u_V$  from Definition 3.12 of \cite{HHR}. In particular, if $\uR=\mF$, then to each actual representation $V\in RO(G)$, there corresponds an orientation class
\[u_V \in H^{V-|V|}(G/G,\mF).\]
\end{example}

We have the following generalization of \cref{lem:orientabilitynoneq} to the equivariant setting.
\begin{lem}\label{lem:orientationabc}
Let $B$ be a $G$-space whose underlying space is path connected and which has a fixed point.   Let $\xi$ be an actual $G$-bundle over $B$. Suppose that $i^*_e\xi$ is $R$-orientable over $i^*_eB$. 
The following are equivalent.
\begin{enumerate}[(a)]
\item The bundle $\xi$ is $R$-orientable in the sense of \cref{defn:orientation_class}.
\item  For any fixed point $b \colon G/G \to B$, the fiber $\xi_b$ over $b$ is an $R$-orientable $G$-representation and
\[ b_!b^* \colon H^{\xi-|\xi|}_B(B,\mR) \to H^{\xi_b-|\xi|}(G/G ,\mR)\]
is an isomorphism.
\item There exists a fixed point $b\colon G/G \to B$ such that the fiber $\xi_b$ is an $R$-orientable $G$-representation and
\[ b_!b^* \colon H^{\xi-|\xi|}_B(B,\mR) \to H^{\xi_b-|\xi|}(G/G ,\mR)\] 
is an isomorphism.
\end{enumerate}
\end{lem}

\begin{proof}
We assume throughout that $i^*_e\xi$ is an $R$-orientable bundle over $i^*_eB$.
We have the following commutative diagram, where all coefficients are $\mR$.
  \begin{align*}
  \xymatrix@C=1pc@R=1pc{
& \widetilde{H}^{-|\xi|}(\Th(-\xi))  \ar[rr]^-{\Th(b)^*} \ar'[d][dd]^(.45){i^*_e}
&& \widetilde{H}^{-|\xi|}(\Th(-\xi_b))\ar[dd]^-{i^*_e}_-\subset \\
H^{\xi-|\xi|}_B(B) \ar[ru]^-\cong \ar[rr]^(.65){b_!b^*} \ar[dd]^-{i^*_e} && H^{\xi_b-|\xi|}(G/G) \ar[ru]^-\cong \ar[dd]^(.68){i^*_e}_(.68){\subset} & \\
&  \widetilde{H}^{-|\xi|}(\Th(-i^*_e\xi)) \ar'[r]^(.65){\cong}_(.65){\Th(i^*_eb)^*}[rr] && \widetilde{H}^{-|\xi|}(\Th(-i^*_e\xi_b))  \\ H^{i^*_e\xi-|\xi|}_{i^*_eB}(i^*_eB) \ar[ru]^-\cong \ar[rr]^-{i^*_eb_!b^*}_-\cong && H^{i^*_e\xi_b-|\xi|}(\mathrm{pt}) \ar[ru]^-\cong &
}
\end{align*}
The isomorphisms on the bottom face of the cube follow from the orientability of $i^*_e\xi$ and correspond to the diagram \eqref{eq:noneqface} for $i^*_e\xi$ over $i^*_eB$.
The isomorphisms pointed inwards are those of \cref{prop:cohBThom}. The vertical maps are the restrictions from equivariant to underlying singular cohomology. The right face is just an example of the commuting square in \eqref{eq:faceoverpoint}, so  falls under the situation described in \cref{ex:uVclasses}. The fact that the restrictions are inclusions is explained there. (The front and back face of the cube contain equivalent information.)

Assume (a) and let $b \colon G/G \to B$ be any fixed point. Since the left-hand restriction $i^*_e \colon H_B^{\xi-|\xi|}(B,\mR) \to H^{i^*_e\xi-|\xi|}_{i^*_eB}(i^*_eB,\mR)$ is assumed to be an isomorphism and the diagram commutes, the restriction $i^*_e \colon H^{\xi_b-|\xi|}(G/G, \mR)  \to H^{i^*_e\xi_b-|\xi|}(\pt,\mR)$ must be surjective. As explained in \cref{ex:uVclasses}, this map is injective, hence it is an isomorphism. It follows that $\xi_b$ is $R$-orientable as a $G$-representation, and that $b_!b^*$ is an isomorphism. This gives (b). 

That (b) implies (c) is clear since, by assumption, there exists at least one fixed point. Assume (c).  Since $\xi_b$ is $R$-orientable, the right-hand restriction $i^*_e \colon H^{\xi_b-|\xi|}(G/G, \mR)  \to H^{i^*_e\xi_b-|\xi|}(\pt,\mR)$ is an isomorphism. By assumption, $b_!b^*$ is an isomorphism, and therefore $i^*_e \colon H_B^{\xi-|\xi|}(B,\mR) \to H^{i^*_e\xi-|\xi|}_{i^*_eB}(i^*_eB,\mR)$ must also be an isomorphism. This gives (a).
\end{proof}

We establish some properties for orientation classes similar to those established for Euler classes. First, we consider the definition in the case when $G=e$ is trivial. 
\begin{lem}\label{lem:uclassesforGtrivial}
Taking $G=e$, let $B$ be a path connected nonequivariant space. For any (actual) $R$-orientable bundle $\xi$ over $B$, there is an orientation class
\[ u_{\xi} \in H^{\xi-|\xi|}_B(B,\mR)\] which is a unit in $R$. That is, for $\doteq$ denoting equality up to multiplication by a unit in $R$:
\begin{enumerate}[(a)]
\item if $\xi$ and $\gamma$ are $R$-orientable, $u_\xi u_\gamma \doteq u_{\xi+\gamma}$ and
\item $u_{\xi}^2 \doteq 1$. 
\end{enumerate}
\end{lem}
\begin{proof}
Since $RO(\Pi B) \cong \Z \times \Hom(\pi_1(B), O(1))$ (see \cite[Ex. 2.28]{witpaper}) 
and the right factor is 2-torsion, for any $\xi$ we have
\[2\xi  = 2|\xi|.\] 
Therefore,
\[ \xi-|\xi| + \xi = 2\xi- |\xi| = |\xi|\] 
for any $\xi \in RO(\Pi B)$. Suppose that $\xi$ is $R$-orientable. Let $t_\xi$ be the Thom class in parametrized cohomology. Consider the composition 
\[\xymatrix{
H^{\xi-|\xi|}_B(B,\mR) \ar[r]^-{t_\xi}_-\cong & \widetilde{H}^{|\xi|}_B(\Th_B(\xi),\mR) \ar[r]^-{\crush_!}_-\cong &   \widetilde{H}^{|\xi|}(\Th(\xi),R) .
}\]
Notice $u_\xi$ is an orientation class if and only if it maps to an $R$-module generator of $ H^{|\xi|}(\Th(\xi),R) \cong R$. For any (actual) $R$-orientable bundle $\gamma$, we have a commutative diagram 
\[\xymatrix{
H^{\xi-|\xi|}_B(B,\mR)  \otimes_R H^{\gamma-|\gamma|}_B(B,\mR) \ar[r]^-\cong \ar[d] &   \widetilde{H}^{|\xi|}(\Th(\xi),R) \otimes_R  \widetilde{H}^{|\gamma|}(\Th(\gamma),R) \ar[d]^-\cong \\
H^{\xi+\gamma-|\xi+\gamma|}_B(B,\mR) \ar[r]^-\cong &   \widetilde{H}^{|\xi+\gamma|}(\Th(\xi + \gamma),R).
}\]
Here the left-hand vertical map is the cup product and the right-hand vertical map is the product described in \eqref{eq:thompair} induced by the composition of the external cross-product with the pullback along the diagonal. 
The left-hand vertical map is thus an isomorphism, and so $u_\xi u_\gamma$ is an orientation class. Any two orientation classes differ by a unit, so $u_{\xi}u_\gamma \doteq u_{\xi+\gamma}$. This proves the first claim. 

Since $\dim(2\gamma) =\dim(2|\gamma|)$ in $RO(\Pi B)$, we have that 
\[u_{\gamma}^2  \doteq  u_{2\gamma} \doteq u_{2|\gamma|}\doteq 1.\]
This proves the second claim. 
\end{proof}

Returning to the equivariant setting, orientation classes satisfy properties similar to those of Euler classes shown in \cref{lem:eulerprop}. 
\begin{lem}\label{lem:orientprop}
Let $\xi$ and $\gamma$ be (actual) orthogonal $G$-vector bundles over $B$, and let $H$ be a subgroup of $G$. Suppose that $\xi$ and $\gamma$ are $R$-orientable.
\begin{enumerate}[(a)]
\item  Under restriction
\[ i^*_H(u_\xi) = u_{i^*_H\xi}.\]
\item If $f\colon A \to B$ is a $G$-map, then
\[f^*u_\xi = u_{f^*\xi},\]
where the two sides are compared using the identification
\[ f_! \colon H^{\xi-|\xi|}_B(A, \mR) \xrightarrow{\cong} H^{f^*\xi-|f^*\xi|}_A(A, \mR)  .\]
\item If  $\xi+\gamma$  is $R$-orientable, the orientation classes satisfy the relation 
\[u_{\xi+\gamma}  \doteq u_{\xi}u_\gamma,\]
where the equality is up to multiplication by a unit in $R$.
\end{enumerate}
\end{lem}
\begin{proof}
Parts (a) and (b) follow immediately from the definition, while
 (c) follows from \cref{lem:uclassesforGtrivial} and the fact that $i^*_e$ is a ring homomorphism. 
\end{proof}

\subsection{Homogeneity units}\label{sec:homounit}
This section treats units in parametrized cohomology that arise from Thom isomorphisms in $RO(G)$-graded cohomology with $\mF$-coefficients for homogeneous bundles. 

We assume our $G$-spaces $B$ are such that the underlying space of $B$ is path connected and that $B$ has
 a fixed point. 
As a corollary to the $RO(G)$-graded Thom isomorphism with $\mF$-coefficients for homogeneous bundles, we have the following isomorphism between parametrized and $RO(G)$-graded cohomology.

\begin{cor}\label{prop:homog}
 Let $B$ be a $G$-space whose underlying space is path connected and which has a fixed point.
Suppose that $\xi$ is a 
homogeneous virtual bundle with generic fiber $\xi_0$ over $B^G$.
Then there is an isomorphism 
\[H^{\star+\xi- \xi_0}_B(B,\mF) \cong H^{\star}(B ,\mF) \]
for $\star \in RO(G)$.
\end{cor}
\begin{proof}
By \cref{prop:cohBThom} and \cref{thm:homothom}, 
\begin{align*}
H^{\star+\xi- \xi_0}_B(B,\mF)&\cong \widetilde{H}^{\star- \xi_0}(\Th(-\xi),\mF) 
\\
&\cong H^{\star}(B,\mF). \qedhere
\end{align*}
\end{proof}

\begin{defn}\label{defn:homounit}
Let $B$ be a $G$-space whose underlying space is path connected and which has a fixed point.
Suppose that $\xi$ is a 
homogeneous virtual bundle over $B$ with generic fiber $\xi_0$.
The \emph{homogeneity unit} of $\xi$ is the unique class $\ee_{\xi}\in H^{\xi-\xi_0}_B(B,\mF)$ which maps to $1$ under the isomorphism of \cref{prop:homog}. 
\end{defn}

\begin{rem}\label{rem:homounit}
The homogeneity unit is the class
\[ \ee_{\xi} =t_{-\xi}^{-1}\bar{t}_{-\xi}^{}(1) \]
where $t_{-\xi}$ is as in \cref{thm:thom} and $\bar{t}_{-\xi}$ is as in \cref{thm:homothom}.
\end{rem}

The next result follows from naturality and multiplicative properties of Thom classes.
\begin{lem}\label{lem:eclassprop}
Let $B$ be a $G$-space whose underlying space is path connected and which has a fixed point as in \cref{prop:homog}. Let $\xi$ and $\gamma$ be virtual homogeneous bundles over $B$. Let $H$ be a subgroup of $G$.
\begin{enumerate}[(a)]
\item  Under restriction,
\[ i^*_H(\ee_{\xi}) = \ee_{i^*_H(\xi)}.\]
\item If $f\colon A \to B$ is a $G$-map, then
\[f^*(\ee_{\xi}) = \ee_{f^*\xi},\] 
where the two sides are compared using the identification
\[ f_! \colon H^{\xi-\xi_0}_B(A,\mF) \xrightarrow{\cong} H^{f^*\xi-f^*\xi_0 }_A(A,\mF)  .\]
\item The homogeneity units are multiplicative, satisfying the relation 
\[\ee_{\xi+\gamma} = \ee_{\xi}\ee_{\gamma}.\]
\end{enumerate}
\end{lem}

Observe that $\ee_0 = 1$ together with the multiplicative behavior implies the class $\ee_{\xi}$ is a unit since $\ee_{\xi}^{-1} = \ee_{-\xi}$, justifying the name homogeneity unit in \cref{defn:homounit}.

For actual bundles, there is a direct correspondence between orientation classes and the homogeneity units of \cref{prop:homog}.

\begin{lem}\label{lem:euclassesrel}
Let $B$ be a $G$-space whose underlying space is path connected and which has a fixed point. Let $\xi$ be an actual homogeneous $G$-bundle with generic fiber $\xi_0$. Then $\xi$ is $\F_2$-orientable in the sense of \cref{defn:orientation_class}  and 
\[ u_{\xi}  = u_{\xi_0}\ee_{\xi}\]
in $H_{B}^{RO(\Pi B)}(B;\mF)$.
\end{lem}
\begin{proof}
The classes $u_{\xi}$, $u_{\xi_0}$ and $\ee_{\xi}$ are uniquely determined by their degrees and the property that they restrict to the identity in $H^0(i^*_eB,\F_2)$. 
But on underlying dimensions, $|\xi|=|\xi_0|$, and so
\[|u_{\xi}| = \xi-|\xi| =\xi-\xi_0 + \xi_0 -|\xi_0| = |\ee_{\xi}u_{\xi_0}|.\]
Since the restriction is multiplicative, $u_{\xi_0}\ee_{\xi}$  restricts to $1$, and so is an orientation class $u_\xi$ as in \cref{defn:orientation_class}.
\end{proof}

In the case of the trivial group, the orientation classes agree with the homogeneity units.
\begin{lem}
Let $B$ be a path connected nonequivariant space.
 For an actual bundle $\xi$ over $B$, both the orientation class and the homogeneity unit exist in
 \[H^{\xi-|\xi|}_B(B,\mF) =  H^{\xi-\xi_0}_B(B,\mF)\]
 and $u_\xi = \ee_\xi$.
\end{lem}
\begin{proof}
All bundles over a path connected nonequivariant space are homogeneous and $\F_2$-orientable. So both the classes $u_\xi$ and $\ee_\xi$ exist.  Both classes are the unique class that restricts to $1$ with degree $\xi-\xi_0 = \xi -|\xi|$.
\end{proof}

\subsection{Triviality classes}
In this section, we explain how for certain representations and parametrized Mackey functors $\mM$, the $RO(\Pi B)$-graded cohomology in degree $\gamma \in RO(\Pi B)$ only depends on the value of $\gamma$ on the objects of $\Pi B$.  These results grew out of conversations with Steve Costenoble. We adopt the notation of \cite[\S 4]{witpaper} and \cite{CW_book}. Refer to these for more details.

For $\gamma \in RO(\Pi B)$ and $x \colon G/H \to B$ in $\Pi B$, we have
\[\gamma(x) = (G\times_H V_1, G\times_H V_2).\] 
Then over the identity coset, we have the $H$-representation 
\begin{align*}
\gamma_0(x) := V_1-V_2\in RO(H)
\end{align*}
and
\[G_+\wedge_H S^{\gamma_0(x),x} = x_!(G\times_H S^{\gamma_0(x)}) \in \mathcal{SH}_B^G .\]
See \cite[Fig. 3]{witpaper} for a cartoon of this construction.
Recall also the $\gamma$-twisted stable fundamental groupoid $\widehat{\Pi}_{\gamma}B$. Its objects are in one-to-one correspondence with those of $\Pi B$ and can be identified with the spectra
$G_+\wedge_H S^{\gamma_0(x),x}$.
Morphisms in $\widehat{\Pi}_{\gamma}B$ are stable maps
\[[G_+\wedge_H S^{\gamma_0(x),x}, G_+\wedge_H S^{\gamma_0(y),y}]^G_B.\]
We also have the stable orbit category $\mathcal{B}^{G}_B$  of $B$, built as a group completion of the category of spans of finite $G$-sets over $B$. There is a \emph{twisting} isomorphism
\[\Gamma_\gamma \colon \mathcal{B}^{G}_B \to \widehat{\Pi}_{\gamma}B \]
such that
\[\Gamma_\gamma(x) = G_+\wedge_H S^{\gamma_0(x),x}, \]
and 
\begin{align*}
\Gamma_\gamma( z   \xleftarrow{=} z \xrightarrow{p} y) &= \res_\gamma(p) \\
\Gamma_\gamma( x   \xleftarrow{q} z \xrightarrow{=} z) &= \tr_\gamma(q)
\end{align*}
for $x\colon X \to B$, $z \colon Z \to B$ and $y\colon Y \to B$ finite $G$-sets over $B$. See \cite[Def. 3.14 \& 3.15]{witpaper}.

Recall that a parametrized Mackey functor is a functor
\[
\mM \colon (\mathcal{B}_B^G)^\op \to \Ab .
\]
Moreover, for any \(G\)-spectrum \(E\) over \(B\), one obtains an associated parametrized Mackey functor
\[
\underline{\pi}_\gamma E
    = [\Gamma_\gamma(-),E]^G_B
    \colon (\mathcal{B}_B^G)^\op \to \Ab .
\]
By \cite[\S 3.7]{CW_book}, for every parametrized Mackey functor \(\mM\) there exists a spectrum \(H\mM^\gamma\), unique up to equivalence, characterized by the property that
\[
\underline{\pi}_{\gamma+n} H\mM^\gamma
=
\begin{cases}
\mM & n=0,\\
0 & n\neq 0.
\end{cases}
\]

Note that by definition,
\begin{align*}
\underline{\pi}_{\gamma+n} H\mM^\gamma(x) = H^{\gamma}(\Gamma_{\gamma+n}(x),\mM) \cong H^{\gamma-n}(\Gamma_{\gamma}(x),\mM)
\end{align*}
since $\Gamma_{\gamma+n}(x) \simeq \Sigma_B^n\Gamma_\gamma(x)$.

\begin{defn}\label{defn:homotrivial}
We call $\gamma \in RO(\Pi B)$ \emph{homogeneously trivial} if $\gamma_0(x) =0$ for all $x \in \Pi B$. 
\end{defn}

\begin{rem}
Every homogeneous virtual bundle $\xi$ over $B$ with fibers isomorphic to $\xi_0$ will have $\xi - \xi_0$ homogeneously trivial. So this generalizes to a representation in $RO(\Pi B)$ the notion of a representation associated to a homogeneous bundle (shifted to virtual dimension zero) in $KO(\Pi B)$.   
\end{rem}

For $\gamma$ homogeneously trivial, the $\gamma$-twisted stable fundamental groupoid $\widehat{\Pi}_{\gamma}B$ is equal to the one with $\gamma=0$ and so the two twisting functors 
\[  \Gamma_{0},\Gamma_{\gamma} \colon  \mathcal{B}^{G}_B \to \widehat{\Pi}_{0}B=\widehat{\Pi}_{\gamma}B \]
have the same source and target, though they need not agree on morphisms.

\begin{lemma}\label{lem:Mconstant}
Let $X$ be an ex-$G$-space over $B$. Let $\gamma\in \RO(\Pi B)$ be a homogeneously trivial representation such that 
$\mM \circ\Gamma_{\gamma}^{-1} = \mM \circ \Gamma_{0}^{-1}$ as functors.
 Then  
\[ H\mM^{\gamma} \simeq  H\mM. \]
\end{lemma}
\begin{proof}
Note that for all $x \in \Pi B$, $\Gamma_{\gamma}(x) = G_+\wedge_H S^{0,x}$ can be thought of as an ex-$G$-space over $B$ with a $G$-$CW(\gamma)$-complex with a single $0$-cell given by $x$. This makes sense since $\gamma_0(x)=0$ and so $x$ can indeed be taken as a $0$-cell. But this is the same as a the obvious ex-$CW(0)$-cell structure on $\Gamma_0(x)$.
Using the assumption on $\mM$ and \cite[Cor. 4.18]{witpaper}, we see that
$ \Gamma_{\gamma}\circ \Gamma_{0}^{-1}$ induces a natural isomorphism
\begin{align*}
\xymatrix@R=1pc{\underline{\pi}_{n} H\mM(x)\ar@{=}[d] \ar[r]^-\cong &  \underline{\pi}_{\gamma+n} H\mM^{\gamma}(x)\ar@{=}[d] \\
 \widetilde{H}^{-n}_B(\Gamma_{0}(x);\mM) \ar[r]^-\cong
& \widetilde{H}^{\gamma-n}_B(\Gamma_{\gamma}(x);\mM). 
 }
\end{align*}
So, $H\mM$ and $H\mM^\gamma$ satisfy the same universal property. They are thus equivalent.
\end{proof}

In particular, for $\mM=\mF$, we always have $\mF \circ\Gamma_{\gamma}^{-1} = \mF \circ \Gamma_{0}^{-1}$ for homogeneously trivial $\gamma$. The following result generalizes \cref{prop:homog}, which held for homogeneous bundles on path connected spaces with a fixed point. 

\begin{theorem}\label{thm:gammanomatter}
For any homogeneously trivial representation $\gamma\in RO(\Pi B)$ and ex-$G$-space over $B$, 
\[H^{\gamma+\star}_B(B,\mF) \cong H^{\star}_B(B,\mF). \]
In particular, there is a distinguished element $\nu_\gamma \in H^{\gamma}_B(B,\mF)$ corresponding to $1 \in H^{0}_B(B,\mF)$.
\end{theorem}

\begin{proof}
Recall that $\mF=\crush^*\mF = \mF \circ  \Span^+(\rho)$, where 
\[\Span^+(\rho)\colon \mathcal{B}_{B}^{G} \to \mathcal{B}^{G}_{G/G}\]
is the map induced on stable orbit categories by the map of fundamental groupoids $\Pi B \to \Pi G/G =\cO_G$ which  sends $x \colon G/H \to B$ to $G/H$ and $(\alpha,\omega)$ to $\alpha$. In this proof, let $\Gamma = \Gamma_{\gamma}^{-1}\circ\Gamma_{0}$. The functor $\Gamma$ is an automorphism of $\mathcal{B}_B^G$ which is the identity on objects. Studying the constructions around Definitions 3.14 and 3.15 in \cite{witpaper}, we see that since $\gamma_0(x) =0$ for all $x$, 
\begin{align*}
\Gamma(\res(\alpha,\omega)) &= \pm \res(\alpha,\omega) \\
\Gamma(\tr(\alpha,\omega)) &= \pm \tr(\alpha,\omega).
\end{align*}
Indeed, the only difference in the construction of $\res_\gamma$ and $\res_0$ is the degree of the map on a sphere introduced by the bundle map $\gamma$, and similarly for $\tr_\gamma$ and $\tr_0$. 
Therefore,
\begin{align*}
\rho^*\mF(\Gamma(\res(\alpha,\omega)) )&=\mF( \pm \res(\alpha) )= \mF(\res(\alpha)) = \rho^*\mF(\res(\alpha,\omega))
\end{align*}
and similarly for the transfer.
It follows that $\crush^*\mF \circ \Gamma = \crush^*\mF$, or equivalently, that
\begin{align*}
\crush^*\mF \circ \Gamma_{0}^{-1}  = \crush^*\mF \circ \Gamma_{\gamma}^{-1} .
\end{align*}
The claim follows from \cref{lem:Mconstant}.
\end{proof}

We call the class $\nu_\gamma$ associated to a homogeneously trivial representation $\gamma$ a \emph{triviality class}. 

\begin{rem} 
If $B$ is path connected with a fixed point and $\xi$ is a homogeneous virtual bundle, then the homogeneity unit $\ee_\xi = \nu_{\xi-\xi_0}$. Indeed, both classes are the unique nonzero class in
\[H^{\xi-\xi_0}_B(B,\mF) \cong H^{0}(B,\mF) \cong H^0(B/G,\F_2) \cong \F_2\]
where the right-hand side is the underlying singular cohomology of the orbits.
\end{rem}

\section{Change of bases}\label{sec:units}
In this section, we prove a variety of results on the effect in  parametrized cohomology of base change from $B$ to a space over $B$. In particular, we will see that often the effect is simply to adjoin units. These results and their proofs are inspired by work appearing in various proofs and statements of \cite{CostenobleB} and some of the results given in this section already appear in some form in that reference. We will not use the results in the remainder of this section until \cref{sec:forget}, when we compute the effect of base change along the inclusion $P^{C_2}\subset P$ in parametrized cohomology, so this material can be skipped on a first reading. 

\subsection{Base-change units}
We start by making a definition to simplify later notation. We work over $\F_2$ to avoid any discussion of graded commutativity and resulting signs.
\begin{defn}\label{defn:Ekappa}
For an abelian group $\kappa$, let $\F_2[\kappa]$ group algebra, viewed as a $\kappa$-graded ring with an element $\alpha \in \F_2[\kappa]$ in degree $\alpha$.
\end{defn}

\begin{defn}\label{defn:kappa}
Let $p\colon X \to B$ be a map of $G$-spaces and consider
$p^* \colon RO(\Pi B) \to RO(\Pi X)$ the induced homomorphism on representations. Let
\[\kappa_{p}:= \ker (p^*)\subset RO(\Pi B) \]
and 
\[\kappa_{p}^{KO}:= \kappa_{p} \cap KO(\Pi B).\]
\end{defn}

For any $G$-space $B$ with a fixed point $x \colon G/G \to B$, we get a splitting
\[x^* \colon RO(\Pi B) \to RO(\Pi G/G)\cong  RO(G)\]
of the map $\crush^*$ induced by the projection $\crush \colon B \to G/G$. Therefore,
\[RO(\Pi B) \cong \kappa_{x} \times RO(G) \]
where $\kappa_{x}=\ker(x^*)$ is as defined above.

As was observed by Costenoble--Waner in \cite[Remarks 3.8.5]{CW_book}, the parametrized cohomology of a space $X$ over $B$, in some sense, does not contain any information not already contained in the parametrized cohomology of 
 $X$ over itself. 
 We see this concretely in the next result.
\begin{prop}\label{prop:units}
Let $B$ be a $G$-space whose  underlying space is path-connected. Let $p\colon X \to B$ be a $G$-space over $B$. 
\begin{enumerate}[(a)]
\item For each $\alpha \in \kappa_{p}$, there is an isomorphism
 \[ \xymatrix{H^{\alpha}_B(X,\uR) \ar[r]^-{p_!}_-\cong &H^{0}_X(X,\uR)   }.\]
 The unique class which maps to 1 is a unit $\kk_\alpha \in H^{\alpha}_B(X,\uR)$. Ranging over elements of $\kappa_{p}$, these units satisfy $\kk_0=1$, $\kk_{\alpha}\kk_{\beta} = \kk_{\alpha+\beta}$. 

\item A splitting of $p^* \colon RO(\Pi B) \to p^*(RO(\Pi B))$ (if one exists) induces an isomorphism of graded rings 
 \[ 
H^{RO(\Pi B)}_B(X,\mF) \cong
 H^{p^*(RO(\Pi B))}_X(X,\mF)\otimes  \F_2[\kappa_{p}]    \]
 for $\kappa_{p}$ as in \cref{defn:kappa}.
 \end{enumerate}
 The results of (a) and (b) also hold if we replace $RO(\Pi B)$, $RO(\Pi X)$, and $\kappa_{p}$ with $KO(\Pi B)$, $KO(\Pi X)$, and $\kappa_{p}^{KO}$.
\end{prop}
\begin{proof}
Recall that $p_!(X_{+_X}) = X_{+_B}$. So applying \cref{thm:shriek-iso},
 for any $\alpha \in RO(\Pi B)$, we have a natural isomorphism
  \[ \xymatrix@R=1pc{H^{\alpha}_B(X,\uR) \ar[r]^-{p_!}_-\cong &H^{p^*\alpha}_X(X,\uR)  } .\]
Let $\alpha \in \kappa_{p}$ so that $p^*\alpha=0$ and define $\kk_\alpha$ to be the unique class so that $p_!(\kk_\alpha) = 1$ for $1\in H^{0}_X(X,\uR)$.  Uniqueness together with the fact that $p_!$ preserves products implies $\kk_0=1$ and $\kk_{\alpha}\kk_\beta = \kk_{\alpha+\beta}$. Note that each $\kk_\alpha$ is a unit since  $\kk_{\alpha}^{-1} = \kk_{-\alpha}$. This proves (a).

For part (b), assume we have a splitting $q \colon p^*(RO(\Pi B)) \to RO(\Pi B)$ of $p^*$ (restricted to its image), so that
\[RO(\Pi B) \cong \kappa_p \times q(p^*(RO(\Pi B))).\]
Using that $p^*q(\beta)=\beta$ for all $\beta \in p^*(RO(\Pi B))$ and specializing to $\mF$-coefficients, we have isomorphisms
 \[p_! \colon H^{q(\beta)}_B(X,\mF)
\xrightarrow{\cong} H^{\beta}_X(X,\mF)  .\]
These assemble into an isomorphism 
\[p_!^{-1} \colon H^{p^*(RO(\Pi B))}_X(X,\mF)
\xrightarrow{\cong}
H^{q(p^*(RO(\Pi B)))}_B(X,\mF).
\]
Now consider the homomorphism 
\[\phi \colon 
 H^{p^*(RO(\Pi B))}_X(X,\mF)\otimes \F_2[\kappa_{p}]
\to H^{RO(\Pi B)}_B(X,\mF) \]
defined by applying $p_!^{-1}$ to $H^{p^*(RO(\Pi B))}_X(X,\mF)$, and sending $\alpha \in \F_2[\kappa_{p}]$ to $\kk_\alpha$. Then $\phi$ is a homomorphism of graded rings, and an isomorphism in each degree.
\end{proof}

\begin{rem}
Note that part (a) of the above proposition is in fact formal and such units arise whenever we have a homomorphism between graded rings. See \cite[Def. 3.7 and Prop. 3.8]{CostenobleB}. 
\end{rem}

\begin{defn}\label{defn:basechangeunit}
For $p\colon X \to B$ and $\alpha \in \kappa_p = \ker(p^*)\subset RO(\Pi B)$, we call the class $\kk_\alpha \in H^{\alpha}_B(X, \uR)$  of \cref{prop:units} the \emph{base-change unit} of $\alpha$.
\end{defn}

\subsection{Applications of base-change units} Applying the above, there are a number of circumstances where parametrized cohomology is determined by something simpler.

For a connected base with a fixed point, the parametrized cohomology of $G/G$ with $\mF$-coefficients contains essentially the same information as $RO(G)$-graded cohomology.
\begin{cor}\label{cor:pointcoh}
Let $B$ be a $G$-space whose underlying space is path connected. Assume there is 
 a fixed point $x\colon G/G \to B$ and let  $\kappa_{x}$ be as in \cref{defn:kappa}. There is an isomorphism of graded rings 
\[H^{RO(\Pi B)}_B(G/G,\mF)\cong H^{RO(G)}(G/G, \mF)\otimes \F_2[\kappa_{x}].  \]
A similar isomorphism holds if we replace $RO(\Pi B)$ with $KO(\Pi B)$ and $\kappa_{x}$ with $\kappa_{x}^{KO}$.
\end{cor}
\begin{proof}
Apply \cref{prop:units} to the space $x \colon G/G \to B$ over $B$. Again note that $x^*$ is split surjective with splitting given by $\crush^*$ for $\crush \colon B \to G/G$ as usual.
\end{proof}
 
We can also identify $HO(\Pi B)$-graded cohomology with something simpler using homogeneity units and the isomorphism of \cref{prop:homog}.
\begin{rem}
Let $B$ be a $G$-space whose underlying space is path connected. Assume there is a fixed point $x \colon G/G \to B$. Recall $HO(\Pi B)$ is the subgroup of $KO(\Pi B)$ consisting of homogeneous virtual bundles. Let 
\[\kappa_{x}^{HO} =\kappa_{x} \cap HO(\Pi B)=\{ \xi-\xi_0  \mid \xi \in  HO(\Pi B) \}.\]
Then
\[HO(\Pi B) \cong RO(G) \times \kappa_{x}^{HO} .\] 
\end{rem}

\begin{prop}\label{homog:DOG}
Let $B$ be a $G$-space whose underlying space is path connected and which has a fixed point $x \colon G/G \to  B$.
There is an isomorphism of graded rings
\[ H^{RO(G)}(B,\mF)\otimes \F_2[\kappa_{x}^{HO}] \xrightarrow{\cong}  H_B^{HO(\Pi B)} (B ,\mF) \]
which maps $\xi-\xi_0 \in \kappa_{x}^{HO}$ to
the homogeneity unit $\ee_\xi$ of \cref{defn:homounit}.
\end{prop}

\begin{proof}
By \cref{prop:homog}, for any $\xi$ virtual bundle with fibers $\xi_0$, we have a Thom isomorphism
\[ H^{\star}(B,\mF) \xrightarrow{\cong} H_B^{\star+\xi-\xi_0}(B,\mF) \] 
which maps $1$ to $\ee_\xi$
as in \cref{defn:homounit}.
Consider the map
\[ H^{RO(G)}(B,\mF)\otimes \F_2[\kappa_x^{HO}] \to H_B^{HO(\Pi B)}(B,\mF) \]
which sends $a\otimes (\xi-\xi_0)$ to $\rho_!(a)\ee_\xi$
where as always,
\[\rho_! \colon  H^{RO(G)}(B,\mF)\xrightarrow{ \cong } H^{\rho^*(RO(G))}_B(B,\mF) .\]
This is clearly a ring homomorphism and it is an isomorphism in every degree. Indeed, $HO(\Pi B) \cong RO(G) \times \kappa_x^{HO}$ and for every homogeneous $\xi$, $\ee_\xi$ is a unit so that
\[ H^{\star}_B(B,\mF) \xrightarrow[\cong]{\smile \ee_\xi} H_B^{\star+\xi-\xi_0}(B,\mF) ,\]
 is an isomorphism. 
\end{proof}

For trivial $G$-spaces, parametrized cohomology with $\mF$-coefficients is determined by $RO(G)$-graded cohomology.
\begin{cor}\label{cor:ROgradedcohtrivialaction}
Let $B$ be a path connected space with a trivial $G$-action. Let $x\colon G/G \to B$. 
There are isomorphisms of graded rings
\begin{align*} 
H^{RO(\Pi B)}_B(B,\mF) &\cong H^{RO(G)}(B,\mF)\otimes \F_2[\kappa_{x}] \\
&\cong H^{RO(G)}(G/G,\mF)\otimes H^{*}(i^*_eB ,\F_2) \otimes \F_2[\kappa_{x}]  \\
&\cong H^{RO(\Pi B)}_B(G/G ,\mF) \otimes H^{*}(i^*_eB ,\F_2)  .
\end{align*}
\end{cor}

\begin{proof}
If $B$ is a path connected space with a trivial $G$-action, then $B^G=B$ is path connected. So all bundles over $B$ are homogeneous. Using \cref{prop:trivKOROequal}, we get
\[HO(\Pi B)=KO(\Pi B)=RO(\Pi B). \]
Then the first isomorphism we want to prove follows from \cref{homog:DOG}.
Since $B$ has a trivial action, 
we have an isomorphism in $RO(G)$-graded cohomology
\[H^{RO(G)}(B,\mF) \cong H^{RO(G)}(G/G ,\mF) \otimes H^{*}(i^*_eB ,\F_2) \]
where the right tensor factor is underlying singular cohomology.
By \cref{cor:pointcoh}, 
\[H^{RO(\Pi B)}_B(G/G,\mF)\cong H^{RO(G)}(G/G, \mF)\otimes \F_2[\kappa_{x}] , \]
which completes the proof.
\end{proof}

An example of the previous result occurs when $G$ is trivial. For $B$ a nonequivariant path connected space and $x\in B$, we have
\[ \kappa_{x}\cong \Hom(\pi_1(B,x), O(1)) .\]
See \cite[Example 2.28]{witpaper} for more details.
As a direct consequence of \cref{cor:ROgradedcohtrivialaction}, we  have the following result, which generalizes the non-equivariant results of \cite{Cadek} who considers cohomology graded on local coefficient systems. 

\begin{cor}\label{cor:cohBtrivunits}
Let $B$ be a path connected nonequivariant space. Let $x\in B$. There is an isomorphism of graded rings
\[ H^{RO(\Pi B)}_B(B,\mF) \cong  H^*(B,\F_2) \otimes \F_2[\kappa_{x}].\]
\end{cor}

We end with a result about the inclusion of fixed points in a $G$-space $B$, which we will use to study the $RO(\Pi P)$-graded cohomology of $P^{C_2}$. 
 \begin{thm}\label{thm: use cohomology of fixed points}
 Let $B$ be a $G$-space whose underlying space is path connected and which has a fixed point.
 Let $X$ be a nonempty path connected subcomplex of $B^G$ and $x \colon G/G \to X$.
There are isomorphisms of graded rings
\begin{align*} H^{RO(\Pi B)}_B (X,\mF) &\cong H^{RO(G)}(X,\mF) \otimes \F_2[\kappa_{px}] \\
&\cong H^*(i^*_eX,\F_2)\otimes H^{RO(\Pi B)}_B(G/G,\mF).
\end{align*}
 \end{thm}
\begin{proof}
Let $p\colon X \to B$ be the inclusion.
Let $\kappa_{px}$ and $\kappa_{x}$ be
as in \cref{defn:kappa}. We have a diagram
\[\xymatrix{ 
0\ar[r] &
\kappa_{px} \ar[d]\ar[r] & RO(\Pi B)\ar[d]^-{p^*} \ar[r]^-{(px)^*} & RO(G) \ar[d]^-= \ar[r] & 0 \\
0\ar[r] & \kappa_{x} \ar[r] & RO(\Pi X) \ar[r]^-{x^*} & RO(G) \ar[r] & 0 ,
}\]
where the rows are split by $\crush^*$. That is, 
\begin{align*}
RO(\Pi B) &\cong  RO(G) \times  \kappa_{px} \\
RO(\Pi X) &\cong  RO(G) \times  \kappa_{x} .
\end{align*}

For any $\alpha \in RO(\Pi B)$, write $\alpha = V + \beta$ for $V \in RO(G)$ and  $\beta \in \kappa_{px}$. Then 
\[p^{*}(\alpha) = p^{*}(V+\beta) = V+  p^* (\beta),\] 
with $p^{*}(\beta) \in \kappa_{x}$.
We have
\[\xymatrix{
 H^{\alpha}_B(X,\mF) \ar@{=}[d]\ar[r]^-{p_!}_-\cong & H^{p^{*}(\alpha)}_X(X,\mF) \ar@{=}[d]\\
  H^{V+\beta}_B(X,\mF) \ar[r]^-\cong &  H^{V+p^*(\beta)}_X(X,\mF).
}\]
By \cref{cor:ROgradedcohtrivialaction}, we know that
 \[H^{\star+p^*(\beta)}_X(X,\mF) \cong H^{\star}(X,\mF)\{ \kk_{p^*(\beta)} \}\]
for 
\[\kk_{p^*(\beta)} \in H^{p^*(\beta)}_X(X,\mF)\] the base-change unit of $p^*(\beta)$ along $x \colon G/G \to X$.
Consider the map
\[\phi \colon H^{RO(G)}(X,\mF)\otimes \F_2[\kappa_{px}] \to H^{RO(\Pi B)}_B (X,\mF) \]
given by
\[\phi(a\otimes\kk_{p^*(\beta)}  ) = p_!^{-1}(a\kk_{p^*(\beta)}).\]
This is a morphism of graded rings, and is an isomorphism in each degree by \cref{cor:ROgradedcohtrivialaction} and \cref{thm:shriek-iso}. This proves the first isomorphism in the statement.

For the second isomorphism, we have
\begin{align*} H^{RO(\Pi B)}_B (X,\mF) &\cong H^{RO(G)}(X,\mF) \otimes \F_2[\kappa_{px}] \\
&\cong H^*(i^*_eX,\F_2)\otimes H^{RO(G)}(G/G,\mF) \otimes \F_2[\kappa_{px}] \\
&\cong H^*(i^*_eX,\F_2)\otimes H^{RO(\Pi B)}_B(G/G,\mF)
\end{align*}
where we used that $X$ has a trivial action in the second line and \cref{cor:pointcoh} in the third.
\end{proof}

 \section{Constructions for $G=C_2$}

\subsection{$\mathbb{M}_2$-module structure}\label{sec:M2module}
 We turn to the specific case of $G=C_2$.
With various notations, 
\[\M_2  = H^{RO(C_2)}(C_2/C_2, \mF) = H^{\star}(C_2/C_2, \mF) = H^{*,*}(C_2/C_2, \mF) \] 
all denote the $RO(C_2)$-graded cohomology of the point. In the last term, the degree $(p,q)$ corresponds to $\R^{p,q} \in RO(C_2)$, the direct sum of $p-q$ copies of the one-dimensional trivial representation
and $q$ copies of $\sigma$ the one-dimensional sign representation. 
The ring $\M_2$ is depicted in \cref{fig:M2}.  It consists of two cones with a copy of $\F_2$ at each lattice point inside the cones.  The upper cone is polynomial in the classes $\aa \in H^{1,1}(C_2/C_2,\mF)$ and $\uu \in H^{0,1}(C_2/C_2,\mF)$.  The class $\theta \in H^{0,-2}(C_2/C_2,\mF)$ in the lower cone is infinitely divisible by $\aa$ and by $\uu$, and has $\theta^2 = 0$.  This is written additively as
\[\M_2 = H^{*,*}(C_2/C_2 ,\mF) \cong \F_2[\uu,\aa]\oplus \F_2\{\theta/(\uu^i\aa^j) : i,j\geq 0\}.\]
The $RO(C_2)$-graded cohomology of any $C_2$-space $X$ is an $\M_2$-algebra via the equivariant map $\crush \colon X \to C_2/C_2$.  

\begin{figure}[ht]
\begin{tikzpicture}[scale=0.7]
\draw[step=1cm,gray,very thin] (-4.5,-2.5) grid (4.5,6.5);
\draw[] (-4.5,2) -- (4.5,2) node[below, black] {\small $p$};
\draw[] (0,-2.5) -- (0,6.5) node[left, black] {\small $q$};

\foreach \y in {2,...,5}
  \foreach \x in {\y,...,2}
    \fill (\x-1.5,\y+0.5) circle (3pt);

\foreach \x in {0,...,3}
  \draw[thick,->] (\x+0.5,\x+2.5) -- (\x+0.5,6.25);
\foreach \y in {2,...,5}
  \draw[thick,->] (0.5,\y+0.5) -- (6.25-\y,6.25);

\foreach \x in {-2,...,0}
  \foreach \y in {-2,...,\x}
    \fill (\x+0.5,\y+0.5) circle (3pt);

\foreach \x in {-2,...,0}
  \draw[thick,->] (\x+0.5,\x+0.5) -- (\x+0.5,-2.25);
\foreach \y in {-2,...,0}
  \draw[thick,->] (0.5,\y+0.5) -- (-2.25-\y,-2.25);

\draw[] (0.51,3.5) node[left,black] {\small $\uu$};
\draw[] (0.5,4.5) node[left,black] {\small $\uu^2$};
\draw[] (0.5,2.25) node[left,black] {\small $1$};
\draw[] (1.8,3.4) node[black] {\small $\aa$};
\draw[] (2.95,4.5) node[black] {\small $\aa^2$};
\draw[] (1.1,0.35) node[left,black] {\small $\theta$};
\draw[] (1.1,-0.65) node[left,black] {\small $\frac{\theta}{\uu}$};
\draw[] (-0.6,-0.55) node[left,black] {\small $\frac{\theta}{\aa}$};
\end{tikzpicture}
\caption{The ring $\M_2=H^{*,*}(C_2/C_2, \mF)$.}
\label{fig:M2}
\end{figure}

In $\M_2$, both the Euler class and the orientation class of an (actual) orthogonal representation $\R^{p,q}$ for $p\geq q\geq 0$ exist. For example,
\begin{align*}
 a_{\R^{1,0}} &=  0 &    a_{\R^{1,1}} &= \aa  \\
 u_{\R^{1,0}} &=1 &   u_{\R^{1,1}} &= \uu . 
 \end{align*}
Multiplicativity determines all other Euler and orientation classes. The class $\aa$ is often denoted by $\aa_\sigma$ or $\rho$, and $\uu$ by $\uu_\sigma$ or $\tau$.

\begin{lem}
Let $B$ be a $C_2$-space.
The parametrized cohomology of an ex-$C_2$-space over $B$ with coefficients in $\mF$ is an algebra over $\M_2$.
\end{lem}
\begin{proof}
 The parametrized cohomology of an ex-$C_2$-space over $B$ with coefficients in $\mF$ is an algebra over the ring $H^{RO(\Pi B)}_B(B, \mF)$, which contains the subalgebra $H^{\crush^*RO(C_2)}_B(B, \mF)$. But by \cref{thm:shriek-iso}
\[H^{\crush^*RO(C_2)}_B(B, \mF) \cong H^{RO(C_2)}(B, \mF),  \]
where the right-hand side is the unreduced $RO(C_2)$-graded cohomology of $B$.
The $RO(C_2)$-graded cohomology of any $C_2$-space with $\mF$-coefficients is an algebra over $\M_2$, and this proves the claim.
\end{proof}

We then have the following relationship between orientation classes and homogeneity units.
\begin{lem}
 Let $B$ be a $C_2$-space whose underlying space is path connected and which has a fixed point $x \colon C_2/C_2 \to  B$. Let $\xi$ be an actual homogeneous $C_2$-bundle with generic fiber $\xi_0\cong \R^{p,q}$.
Then
\[ u_\xi = \uu^q \ee_{\xi}\]
in $H_{B}^{RO(\Pi B)}(B;\mF)$.
\end{lem}
\begin{proof}
By \cref{lem:euclassesrel}, 
 $ u_{\xi}  = u_{\xi_0}\ee_{\xi} $.
 But $ u_{\xi_0} = u_{\R^{1,0}}^{p-q}u_{\R^{1,1}}^q = u^q$.
\end{proof}

\subsection{Steenrod operations}\label{sec:bockstein}
Given a map $\varphi \colon E \to F$ in $G$-spectra, applying $\crush^*$ gives a map in $G$-spectra over $B$ as
\[ \xymatrix{
[X, \Th_B(\xi) \wedge_B \crush^*E ]^G_B \ar[r]^-{\crush^*\varphi}&  [X, \Th_B(\xi) \wedge_B\crush^*F ]^G_B 
}.\]
Therefore, we can import cohomology operations from the $RO(G)$-graded world to the parametrized context.  In the $RO(G)$-graded degrees, we recover the usual $G$-equivariant operations since 
\[ \xymatrix{
[X,  \crush^*E ]^G_B \ar[d]_-\cong \ar[r]^-{\crush^*\varphi}&  [X,  \crush^*F ]^G_B \ar[d]_-\cong \\
[\crush_! X,  E ]^G \ar[r]^-{\varphi}&  [\crush_!X,  F ]^G
}\]
commutes.  Furthermore, the identity 
\[\crush^* i^*_H\varphi=  i^*_H\crush^*\varphi\]
implies that this process is well-behaved under restriction to subgroups.

We will apply this construction to Steenrod operations from $RO(C_2)$-graded cohomology. See \cite{Caruso}, 
\cite{MR2031199}, 
\cite{DuggerIsaksen},
\cite{MR1808224}, \cite{2023arXiv230104714Y}
and 
\cite{2019arXiv190500058W} 
for various treatments of Steenrod operations in equivariant and motivic homotopy theory. 

We will consider the $C_2$-equivariant Steenrod squares
\begin{align*} 
\Sq^{1,0} &\colon H\F_2 \to \Sigma^{1,0} H\F_2 &k\geq 0 \\
\Sq^{2k,k} &\colon H\F_2 \to \Sigma^{2k,k} H\F_2  &k\geq 1
\end{align*}
For $k\geq 1$, we also have 
\[\Sq^{2k+1,k}=\Sq^{1,0}\Sq^{2k,k}.\] 
We obtain operations in parametrized cohomology
\begin{align*}
\Sq^{1,0}&:=\rho^*\Sq^{1,0} \colon H^{\xi}_B(B,\mF) \to  H^{\xi+1}_B(B,\mF)   \\
\Sq^{2k,k}&:=\rho^*\Sq^{2k,k} \colon H^{\xi}_B(B,\mF) \to  H^{\xi+k\rho_2}_B(B,\mF)   
\end{align*}
where here $\rho_2 = 1+\sigma = \R^{2,1}$ is the regular representation.
As usual $\Sq^{0,0}=\id$ and
\[\bock\circ \bock =0.\]
The Cartan formula also holds in the parametrized context in the following situation:
\begin{lem}[Cartan Formula]\label{lem:cartan}
The Cartan formula holds in $KO(\Pi B)$-graded cohomology. That is,
for $x_1,x_2 \in H^{KO(\Pi B)}_B(B,\mF)$,
\begin{align*}
\Sq^{1,0}(x_1x_2)&=\Sq^{1,0}(x_1)x_2+x_1\Sq^{1,0}(x_2) \\
\Sq^{2k,k}(x_1x_2)&=\sum_{r=0}^k\Sq^{2r,r}(x_1)\Sq^{2(k-r),k-r}(x_2) \\
 & \ \ \ \ + u\sum_{s=0}^{k-1}\Sq^{2s+1,s}(x_1)\Sq^{2(k-s-1)+1,k-s-1}(x_2) .
\end{align*}
\end{lem}
\begin{proof}
The isomorphism of \cref{prop:cohBThom}
\[ \phi \colon H^{\xi+\star}_B(B,\mF) \xrightarrow{\cong}\widetilde{H}^{\star}(\Th(-\xi),\mF)    \]
is an isomorphism of $\M_2$-modules. It also
respects the Steenrod operations in the sense that 
\[\phi(\Sq(x)) = \Sq(\phi(x)).\] Furthermore, we have $\phi(x_1x_2)= \phi(x_1)\phi(x_2)$ as shown in \cref{lem:phimultiplication}, and 
 therefore
\[\phi(\Sq(x_1x_2)) = \Sq(\phi(x_1) \phi(x_2)). \]
 The parametrized Cartan formula then follows from the equivariant Cartan formula.
\end{proof}

\begin{example}
For example, we have
\begin{align*}
\Sq^{2,1}(x_1x_2) &=  \Sq^{2,0}(x_1)x_2+u\Sq^{1,0}(x_1)\Sq^{1,0}(x_2)+x_1\Sq^{2,0}(x_2).
\end{align*}
\end{example}

Furthermore, the parametrized operations behave well with respect to restriction:
\begin{align*}
i^*_e\bock &= \Sq^1 \\
i^*_e\Sq^{2k,k} &= \Sq^{2k} .
\end{align*}
Finally, we note that in $\M_2$,
\begin{align*}
\bock(\uu)&=\aa  & \Sq^{2,1}(\uu)&=0\\
 \bock(\aa)&=0 & \Sq^{2,1}(\aa)&=0.
\end{align*}
See also Figure 1 in \cite{guillou2020cohomology} for more examples. 
In general, we have the following formulas.
\begin{prop}
On $\M_2$, the Steenrod squares $\Sq^{1,0}$ and $\Sq^{2k,k}$ act as
\begin{align*}
\Sq^{1,0}\pars{a^m u^n} &= n a^{m+1} u^{n-1},\\[0.15cm]
\Sq^{1,0}\pars{\frac{\theta}{a^m u^n}} &= (n+1) \frac{\theta}{a^{m-1} u^{n+1}}
\end{align*}
on the upper cone 
and
\begin{align*}
    \Sq^{2k,k}\pars{a^m u^n} &= {k+n-1 \choose 2k} a^{m+2k} u^{n-k},\\[0.15cm]
    \Sq^{2k,k}\pars{\frac{\theta}{a^m u^n}} &= {k+n+1 \choose 2k} \frac{\theta}{a^{m-2k} u^{n+k}}
    \end{align*}
on the lower cone, 
where all coefficients are taken mod $2$.
\end{prop}

\subsection{The forgetful long exact sequence}\label{sec:forget}
In $RO(C_2)$-graded cohomology for a based $C_2$-space $X$ with coefficients in $\mF$, the ring homomorphism  
\[i^*_e \colon \widetilde{H}^{*,*}(X,\mF) \to \widetilde{H}^{*}(i^*_eX,\F_2)\] 
induced by the restriction
has a particularly nice algebraic interpretation. The cofiber sequence
\begin{equation}\label{eq:cofiber}
\xymatrix{{C_2}_+ \ar[r] &  S^{0} \ar[r]^-{\aa}  & S^{1,1}}
\end{equation}
induces a long exact sequence
\[\xymatrix@C=1.5pc{
\cdots \ar[r] & \widetilde{H}^{p,q}(X) \ar[r]^-{ \aa} & \widetilde{H}^{p+1,q+1}(X)  \ar[r]^-{i^*_e} &\widetilde{H}^{p+1}(i^*_eX) \ar[r] & \widetilde{H}^{p+1,q}(X) \ar[r]^-{ \aa} & \cdots}
\]
where $\aa \in \M_2$ is the element in degree $(1,1)$, and so $\ker(i^*_e) = \im ( \aa)$ and $\coker(i^*_e)\cong\ker( \aa)$. In particular, when $X=C_2/C_2$, the ring map $i^*_e$ is the homomorphism 
\[\xymatrix@R=1pc{
\M_2 = H^{*,*}(C_2/C_2,\mF) \ar[r] &  H^{*}(i^*_e(C_2/C_2),\F_2) =\F_2 
}\]
that sends $\aa \mapsto 0$ and $\uu \mapsto 1$.  

We prove a generalization in the parametrized context, which reduces to the above long exact sequence when $B=C_2/C_2$.
\begin{lem}[Forgetful long exact sequence]\label{lem:ParamForgetfulLES}
Let $B$ be a $C_2$-space with a fixed point.
    Let $X$ be an ex-$C_2$-space over $B$. For $\gamma \in RO(\Pi B)$, let
     $\widetilde{H}_B^{\gamma}(X)= \widetilde{H}^{\gamma}_B(X , \mF)$, $\widetilde{H}^{i^*_e\gamma}_{i^*_eB}(i^*_eX) =\widetilde{H}^{i^*_e\gamma}_{i^*_eB}(i^*_eX,\mF)$
     and 
         \[i^*_e \colon \widetilde{H}_B^\gamma(X) \to  \widetilde{H}_{i^*_eB}^{i^*_e\gamma}(i^*_eX)\] 
         be the restriction. 
        There is a long exact sequence 
    \[\xymatrix@C=1.5pc{
\cdots \ar[r] & \widetilde{H}_{B}^{\gamma}(X) \ar[r]^-{ \aa} & \widetilde{H}_B^{\gamma + \sigma}(X)  \ar[r]^-{i^*_e} &\widetilde{H}^{i^*_e\gamma+1}_{i^*_eB}(i^*_eX) \ar[r] & \widetilde{H}_B^{\gamma+1}(X) \ar[r] & \cdots}
\]
for $\aa \in \M_2$ the class in degree $(1,1)$.
Thus, $\ker(i^*_e) = \im ( \aa)$ and $\coker(i^*_e)\cong\ker( \aa) $.
\end{lem}

\begin{proof}
Take the cofiber sequence \eqref{eq:cofiber}
in $C_2$-spaces. 
Using the external smash product of pointed $C_2$-space with ex-$C_2$-spaces over $B$, we obtain a cofiber sequence
\[
{C_2}_+\wedge X \to X \to S^{1,1} \wedge X.
\]
Applying $[-,H\mF^\gamma]^{C_2}_B$ to this cofiber sequence gives rise to a long exact sequence 
 \[
 \xymatrix@C=1.5pc{
\cdots \ar[r] & \widetilde{H}_B^{\gamma}(X) \ar[r]^-{\cdot \aa} & \widetilde{H}_B^{\gamma + \sigma}(X) \ar[r]^-{i^*_e} &  \widetilde{H}_B^{\gamma+\sigma}({C_2}_+ \wedge X) \ar[r] &  \widetilde{H}_B^{\gamma+1}(X)\ar[r] & \cdots}
\]
Now 
\[\widetilde{H}_B^{\gamma+\sigma}({C_2}_+ \wedge X) \cong \widetilde{H}_{i^*_eB}^{i^*_e(\gamma+\sigma)}(i^*_eX)\] by the Wirthm\"uller isomorphism, which is shown to be monoidal in \cite[3.9]{CW_book}. Since $i^*_e(\gamma+\sigma) = i^*_e(\gamma)+1$, this proves the claim.
\end{proof}

%% file: ppaper-ropip.tex

\part{Computations in parametrized cohomology}\label{part:II}

This part of the paper is dedicated to the computation of the parametrized cohomology of 
\[P := B_{C_2}O(1),\]
the classifying space of $C_2$-equivariant line bundles, with $\mF$-coefficients.

\section{Computation of the grading $RO(\Pi P)$}\label{sec:ROPiP}

In this section we compute the equivariant fundamental groupoid of $P$ and its representations $RO(\Pi P)$. We then identify $KO(\Pi P)$, the subgroup of $RO(\Pi P)$ that corresponds to virtual vector bundles over $P$. For simpler examples done using the same techniques, see the computation of $\Pi P^{2,1}$ and $\Pi P^{3,1}$ in Lemmas 2.31 and 2.35 of \cite{witpaper}.

\subsection{The equivariant fundamental groupoid of $P$}

Recall that  $P$ is the space of one-dimensional real subspaces of the complete universe $\mathcal{U}$, which is a direct sum of infinitely many copies of the trivial representation $\R^{1,0}$ and infinitely many copies of the sign representation $\R^{1,1}$.

\begin{notation}\label{not:vector_space}
    We write $\mathcal{U}$ as the vector space 
\[\R^{2\infty,\infty}= \R^{\infty,\infty} \oplus \R^{\infty, 0} = \colim_{n \in \N} \R^{n,n} \oplus \R^{ n,0}.\] 
We write a point in this vector space as a tuple $(v_1,v_0)$ 
with $v_1 \in \R^{\infty,\infty}$ and $v_0 \in \R^{ \infty,0}$, and with $C_2$-action $\tau \cdot (v_1,v_0) = (-v_1,v_0)$. The corresponding point in the projective space $P$ is denoted $[v_1 : v_0]$. 
If we write $P^{p,q}\subset P$, we mean the space of lines in the subspace $\R^{p,q}\subset \R^{2\infty,\infty}$ for $p \geq q \geq 0$ consisting of the points 
\[((v_{11}, v_{12}, v_{13},  \ldots, v_{1q},0, \ldots ), \ (v_{01}, v_{02}, v_{03}, \ldots, v_{0(p-q)},0, \ldots ) ). \]
\end{notation}

The underlying space $i^*_e P$ is $\R P^\infty$ and the fixed set consists of two connected components,  $P_0$ and $P_1$, with 
\begin{equation}\label{eq:fixed}
P^{C_2} =  P_0\sqcup P_1  ,
\end{equation}
corresponding to the space of lines in $\R^{\infty,0}$ and  $\R^{\infty,\infty}$ respectively.
Each component is homeomorphic to $\R P^\infty$ as well.

Let $b_0\colon C_2/C_2\to P$ be the inclusion of $P^{1,0}$ as $[0:e_1]$ and $b_1\colon C_2/C_2\to P$ the inclusion of $P^{1,1}$ as $[e_1:0]$. 
Let $b\colon C_2/e\to P$ be a free orbit in $P^{2,1}$ determined by $b(e) =[e_1:-e_1]$, as in \cref{fig:pip}.
 The category $\Pi P$ has the following skeleton, depicted here as a category over $\cO_{C_2}$, with the labels  explained below:
\[\xymatrix@R=1.5pc{
&  \Pi  P & & \ar[rrr]^-\phi  & & & & \cO_{C_2}\\ \\
\ar@(ur,ul)@{->}_-{\Z/2\{g_0\}}b_0 & & b_1\ar@(ur,ul)@{->}_-{\Z/2\{g_1\}} & & & & & C_2/C_2  \\
 & b \ar@{->}[ur]_-{\Z/2\{p_1\}} \ar@{->}[ul]^-{\Z/2\{p_0\}} \ar@(dr,dl)@{->}[]^-{\Z/2\{t\} \times \Z/2\{ g\} }  & &  & & &  & C_2/e  \ar[u]_-\crush \ar@(dr,dl)[]^{\gen}
}\]
In the orbit category, $\crush \colon C_2/e \to C_2/C_2$ is the quotient map and $\gen\colon C_2/e \to C_2/e$ is multiplication by the generator $\gen$ of $C_2$.
We abuse notation and denote the morphism $(\alpha,\omega)$ in $\Pi P$ simply by $\omega$.
Specifically, we write the following
\begin{align*}
p_0&=(\crush,p_0) &g&=(e,g) & g_0 &= (e,g_0) \\
p_1&=(\crush,p_1)  & t&=(\gen, t) & g_1&=(e,g_1)
\end{align*} 
for paths and maps that we now describe. 

The identity morphism of any object in $\Pi P$ is given by $\id=(e,c)$ where $c$ is the constant path. Observe first that in $g_0$ and $g_1$, we write $e$ for the identity map from $C_2/C_2$ to itself. 
We have denoted by 
\[g_0,g_1 \colon C_2/C_2 \times [0,1] \to P\] 
chosen generators for $\pi_1(P_0,b_0)\cong\Z/2$ and $\pi_1(P_1,b_1) \cong \Z/2$ respectively.

Since any homotopy $\omega \colon C_2/e\times [0,1] \to P$ is uniquely determined by a path $\omega(e)\colon [0,1] \to i^*_eP$, to describe a morphism in $\Pi P$ with source an object of the from $ C_2/e \to P$, it suffices to describe a path $\omega(e)$. 

We depict these paths in subspaces $P^{3,1}$ and $P^{3,2}$ of $P$ in the upper images of \cref{fig:pip}, where  $C_2$ acts by rotation in both pictures. At the bottom of \cref{fig:pip}, we show their common subspace $P^{2,1}$ with points shown as dashed lines and $C_2$ acting by reflection across the vertical axis. We will return to this in \cref{rem:chi}. Note that the different subspaces $P^{3,1}$ and $P^{3,2}$ both have $\RP^2$ as underlying spaces. 
The path $p_0(e)$ is 
the path from $b(e)$ to $b_0(C_2/C_2)$ as depicted in all three images of \cref{fig:pip}, and similarly for $p_1(e)$. The path $g(e)$ generates $\pi_1 (i^*_eP, b(e)) \cong \Z/2$ and $t(e)$ is the path from $b(e)$ to $b(\gen)$ as in \cref{fig:pip}. The chosen paths $g_0$ and $g_1$ are shown on the upper left and upper right of \cref{fig:pip} respectively.

To compute the generating  relations, it is useful to restrict to the subspaces $P^{3,1}$ and $P^{3,2}$, whose images are shown in \cref{fig:pip}. A careful computation gives the following generating relations among the morphisms:
\begin{align}\label{eq:rel1}
g\circ g &\simeq \id &  t \circ t &\simeq \id  & t\circ  g &\simeq g \circ t
\end{align}
and
\begin{align}\label{eq:rel2}
g_0\circ p_0&\simeq p_0 \circ g & g_1 \circ p_1&\simeq p_1 \circ g  \nonumber\\
p_0 \circ t  &\simeq  p_0\circ g & p_1 \circ t&\simeq p_1  \\
g_0\circ g_0&\simeq \id & g_1\circ g_1&\simeq \id . \nonumber
\end{align}
For more details of such computations in $\Pi P^{2,1}$ and $\Pi P^{3,1}$, see Examples 2.10 and 2.11 in \cite{witpaper}.
\begin{figure}[ht]
\begin{center}
\includegraphics[width=\textwidth]{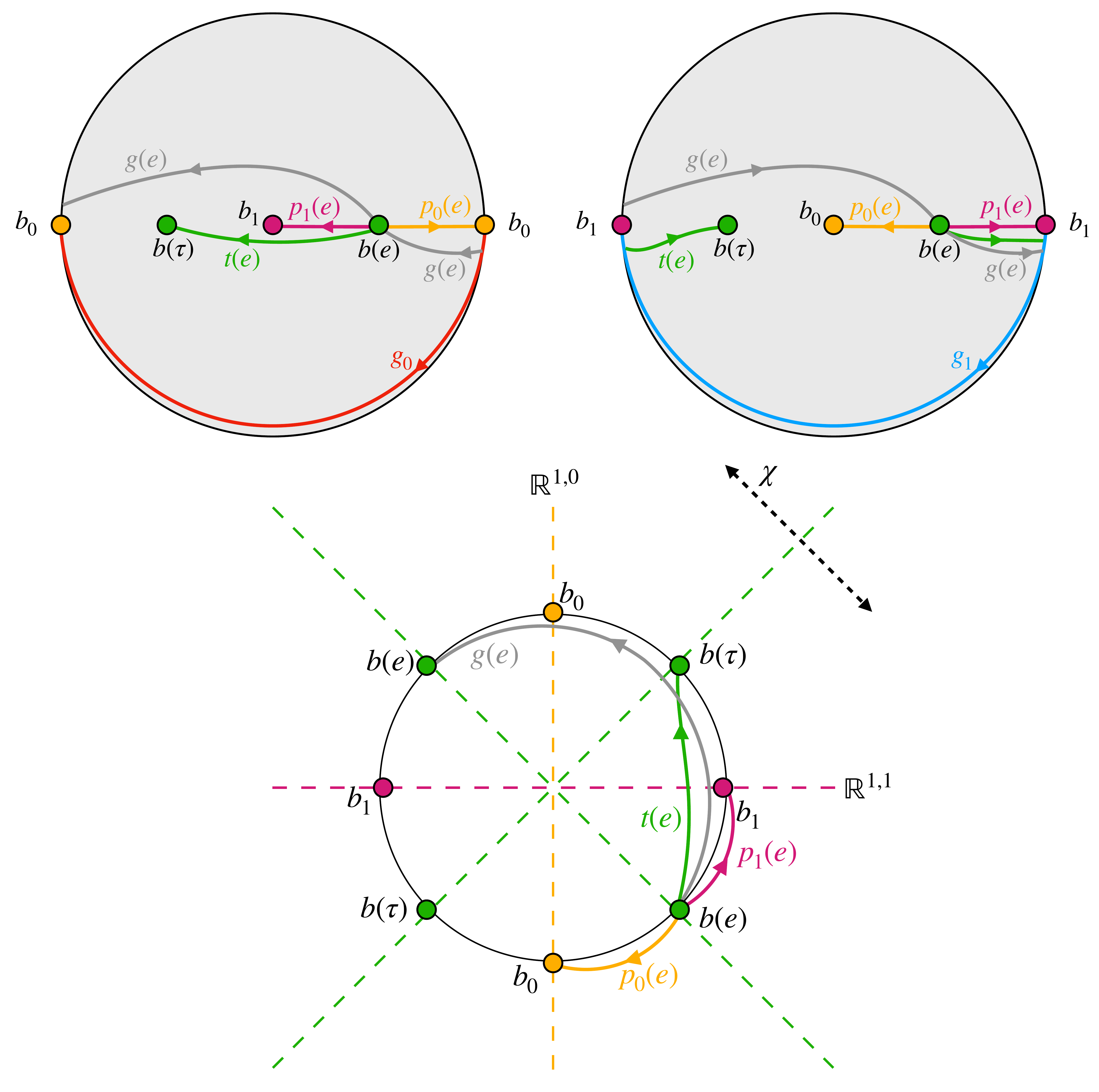}
\caption{$P^{2,1}$ (bottom),
$ P^{3,1}$ (top left) and $P^{3,2}$ (top right). The action of the automorphism $\autP$ of $\Pi P$ (see \cref{rem:chi}) is induced by reflection across the $x=y$ axis. }
\label{fig:pip}
\end{center}
\end{figure}

Next we turn to the fixed points $P^{C_2}$. Because $P^{C_2}$ is disconnected, the $C_2$-equivariant fundamental groupoid $\Pi P^{C_2}$ decomposes as the disjoint union of two categories $\Pi P_0$ and $\Pi P_1$, and we treat each separately. 
For $i=0,1$ the fundamental groupoid $\Pi P_i$ has skeleton
\[\xymatrix@R=2pc{
\Pi P_i  & \ar[r]^-\phi & & \cO_{C_2}\\
b_i \ar@(ur,ul)@{->}_-{\Z/2\{g_i\}} &&& C_2/C_2 \\
 b_i\crush \ar@{->}[u]^-{\Z/2\{c_i\}}  \ar@(dr,dl)@{->}[]^-{\Z/2\{t_i\} \times \Z/2\{h_i\}} 
& &  & C_2/e  \ar[u]_-\crush \ar@(dr,dl)[]^{\gen}
}\]
where $b_i \colon C_2/C_2 \to P$ is as above and $b_i \rho$ is the composition of $b_i$ and the quotient $\crush\colon C_2/e \to C_2/C_2$. The morphisms in this diagram are shorthand for the following
\begin{align*}
 g_i &= (e,g_i)  & c_i&= (\crush, c)\\
t_i &= (\gen,c) &  h_i &= (e,h_i)
\end{align*}
where by abuse $c$ is the constant path at $b_i$, and each $h_i(e)$ is the same path as $g_i(C_2/C_2 \times -)$. We have relations
\begin{align*}
h_i\circ h_i &\simeq \id &  t_i\circ t_i & \simeq \id  & h_i \circ t_i &\simeq t_i \circ h_i 
\end{align*}
and
\begin{align*}
   g_i \circ c_i &\simeq  c_i \circ h_i \\
   c_i \circ t_i &\simeq p_i\\
   g_i\circ g_i & \simeq \id.
\end{align*}

If $\iota_i \colon P_i \to P$ is the inclusion we can compute $(\iota_i)_*:=\Pi \iota_i$. To do this, we fix the isomorphisms $(e,p_i) \colon  b \to b_i$ and $\id=(e,c) \colon b_i \to b_i$.
By pre-composing and post-composing with these isomorphisms as needed, we get a functor between our chosen skeleta which we also denote by $(\iota_i)_*$. We have that $(\iota_i)_*(b_i\crush) = b$ 
while $(\iota_i)_*(b_i)=b_i$.  On morphisms, for $i=0$ we have 
\begin{align*}
(\iota_0)_*(h_0)&= g & (\iota_0)_*(c_0)&= p_0  \\
(\iota_0)_*(t_0)&= t\circ g & (\iota_0)_*(g_0)&= g_0
\end{align*}
and for $i=1$,
\begin{align*}
(\iota_1)_*(h_1)&= g & (\iota_1)_*(c_1)&= p_1  \\
(\iota_1)_*(t_1)&= t & (\iota_1)_*(g_1)&= g_1.
\end{align*}

We note that $\Pi i^*_eP$ is just the fundamental groupoid of the underlying $\R P^\infty$. This has a skeleton with one object, which we can take to be $b(e)$, and then morphisms are $\Z/2$ generated by the loop $g$.  The induction functor 
\[C_2\times_e (-) \colon \Pi i^*_eP \to \Pi P\]
 maps $b(e)$ to $b$ and $g$ to the morphism of the same name in $\Pi P$. 
\bigskip

\begin{remark}\label{rem:chi}
There is an important automorphism of the category $\Pi P$ described as follows. Let $\autP \colon P \to P$ 
 be the $C_2$-homeomorphism that sends a line through the vector $(v_1 , v_0) \in \R^{\infty,\infty}\oplus  \R^{\infty,0}$ to the line through $(v_0,v_1)$. Note that $\chi(P_0)=P_1$ and $\chi(P_1) = P_0$. The map $\chi$ induces a functor
\[ \autP_*=\Pi \autP \colon \Pi P \to \Pi P.\]
We describe it in terms of our skeleton for $\Pi P$.
Since $b_0(C_2/C_2) =[0 :e_1]$, $b_1(C_2/C_2)= [e_1 :0]$ and $b(e) = [e_1 : -e_1]$ for $e_1 \in \R^\infty$ the standard unit vector, the effect of $\autP_*$ on objects is given by
\begin{align*}
\autP_*(b) &= b \\
\autP_*(b_0) &= b_1 \\
\autP_*(b_1) &= b_0.
\end{align*}
On morphisms, $\autP_*(g_0)$ is the non-trivial loop in $P_1$ based at $b_1$, so it must be $g_{1}$. Similarly, $\autP_*(g_1)=g_0$. All the other morphisms in our skeleton are contained in $\Pi P^{2,1}$ so we read off the effect of $\autP_*$ from \cref{fig:pip} (bottom).  
In summary, 
\begin{align*}
\autP_*(g)&=g & \autP_*(g_0)&=g_1 & \autP_*(p_0)&=p_1 \\
\autP_*(t) &= tg  & \autP_*(g_1)&=g_0 & \autP_*(p_1)&=p_0 .
\end{align*}
The map $\autP$ can also be thought of as the map classifying the action of tensoring any line bundle with $\R^{1,1}$. See also \cite{CW_book}, near Proposition 3.12.2.
\end{remark}

\subsection{Representations of the fundamental groupoids}\label{sec:reps of fun gpd}
In this section, we compute $\RO(\Pi P)$, $\RO(\Pi P^{C_2})$ and $\RO(\Pi i^*_e P)$ and describe some homomorphisms between these groups.

First, note that
 \[RO(C_2) \cong RO(\Pi C_2/C_2)\cong \Z^2\] 
 with $(p,q)\in \Z^2$ representing the constant representation $\R^{p,q}$. 
When writing down representatives for representations, we will always choose our bundles to be of the form $C_2 \times \R^p$ over a free orbit, for some $p\in \Z$, and of the form $\R^{p,q}$ for $(p,q) \in \Z^2$ over an orbit of the form $C_2/C_2$. Therefore, for a representation $\gamma$ and a morphism $(\alpha, \gamma)$ in the fundamental groupoid $\gamma(\alpha,\omega)$ will take a value in one of the following morphism sets:
\begin{align*}
\vV_{C_2}(C_2\times \R^p, \R^{p,q} )  &\cong O(1), \\
\vV_{C_2}(C_2\times \R^p, C_2\times \R^p )  &\cong C_2\times O(1), \\
\vV_{C_2}(\R^{p,q},\R^{p,q}) &\cong O(1)\times O(1),
\end{align*}
where in the third isomorphism we write $((-1)^\epsilon,(-1)^\delta) \in O(1)\times O(1)$ for $(-1)^\epsilon$ the determinant of the morphism restricted to the sign representations $\R^{q,q}$ and $(-1)^\delta$ the determinant of the representation restricted to the trivial representation $\R^{p-q,0}$, stably.
See \cite[Example 2.17]{witpaper} for more details.

We start with the computation of $RO(\Pi P)$. Our notation is similar to that of \cite[Lemma 2.31]{witpaper}. See also \cite[Remark 2.32]{witpaper}. 
\begin{theorem}\label{thm:ROPiP}
   There is an isomorphism 
   \[RO(\Pi P)\cong \Z^3\times (\Z/2)^2\]
   that assigns to the tuple $(p,q,n,\epsilon, \mu)$  the representation $\gamma$ described by the following. On objects $\gamma$ is given by  \begin{align*}
   \gamma(b_0)&=\R^{p,q} &
   \gamma(b_1)&=\R^{p,q+n} & \gamma(b)& =C_2\times \R^{p},
   \end{align*} 
   and on morphisms by 
   \begin{align*}
   \gamma(g) &= e\times (-1)^{n} & 
   \gamma(t) &= \gen \times(-1)^{q+n}  \\
   \gamma(p_0) &=1 & \gamma(p_1)  &=1\\
    \gamma(g_0)&=\begin{pmatrix}(-1)^{\epsilon} ,
  (-1)^{n+\epsilon}
  \end{pmatrix}
   &  \gamma(g_1)&= \begin{pmatrix}(-1)^{n+\epsilon+\mu}  
, (-1)^{\epsilon+\mu}
   \end{pmatrix}
   \end{align*}
   as depicted in the following diagram: 
           \[\xymatrix@C=1pc{
&  \Pi  P& & \ar[rr]^-\gamma & & & & \vV_{C_2}\\
b_0 \ar@(ur,ul)@{->}_-{g_0}& & b_1\ar@(ur,ul)@{->}_-{g_1}  & & &&  \R^{p,q} \ar@(ur,ul)@{->}_-{{\tiny \begin{pmatrix}(-1)^{\epsilon},  (-1)^{n+\epsilon}\end{pmatrix}}} & & \R^{p,q+n}  \ar@(ur,ul)@{->}_-{\tiny{ \begin{pmatrix} (-1)^{n+\epsilon+\mu} ,(-1)^{\epsilon+\mu})\end{pmatrix}}} \\
 & b \ar@{->}[ur]_-{p_1} \ar@{->}[ul]^-{p_0} \ar@(dr,d)@{>}[]^-{g}\ar@(d,dl)@{>}[]^-{t} & && & & & C_2 \times \R^p \ar@{->}[ur]_-{1} \ar@{->}[ul]^-{1} \ar@(dr,d)@{>}[]^-{\ \ \  
 (-1)^{n}}\ar@(d,dl)@{>}[]^-{(-1)^{q+n} \ }
}\]
\end{theorem}

\begin{rem}
Before we embark on the proof, we explain our choice of coordinates. The group $RO(\Pi P)$ is the grading for parametrized cohomology. In degrees $(*,*,0,0,0)$, we recover the usual $RO(C_2)$-graded cohomology, graded according to the usual ``motivic'' convention where $p$ is the topological degree and $q$ is the ``weight''. That is, $(p,q,0,0,0) = p-q+q\sigma$ for $\sigma$ the one-dimensional sign representation. The grading $n$ measures the difference of the weights between the two fixed points components: it tells us how far the representation is from the $RO(C_2)$-graded page since in the $RO(C_2)$-graded page, these weights must be equal. More specifically, as will be discussed in \cref{rem:linebundles} below, the page $(*,*,1,0,0)$ corresponds to $RO(C_2)$-shifts of the tautological line bundle, and the degrees $(*,*,-1,1,0)$ to shifts of the pullback of the tautological line bundle along $\autP$. Finally, the vanishing of the last coordinate will indicate whether or not the representation is the dimension of a virtual bundle. This vanishing will be proved in \cref{thm:KOPiP}.
\end{rem}

\begin{proof}
We will prove that there is an isomorphism 
  $ RO(\Pi P)\cong \Z^3\times (\Z/2)^2$ via a map $ \Z^3\times (\Z/2)^2 \to RO(\Pi P)$
   that assigns to the tuple $[p,q_0,q_1,\epsilon_0, \epsilon_1]$  the representation $\gamma$ which, on objects, is given by \begin{align*}
   \gamma(b_0)&=\R^{p,q_0} &
   \gamma(b_1)&=\R^{p,q_1} & \gamma(b)& =C_2\times \R^{p},
   \end{align*} 
   and on morphisms by 
   \begin{align*}
   \gamma(g) &= e\times (-1)^{q_0+q_1} & 
   \gamma(t) &= \gen \times(-1)^{q_1} \quad &
   \gamma(p_0) &=\gamma(p_1)  =1 \end{align*}
   and
\begin{align*}
    \gamma(g_0)&=\begin{pmatrix}(-1)^{\epsilon_0} , (-1)^{\epsilon_0+q_0+q_1}
    \end{pmatrix}
   &  \gamma(g_1)&=\begin{pmatrix} (-1)^{\epsilon_1} , (-1)^{\epsilon_1+q_0+q_1}
   \end{pmatrix}.
   \end{align*}
To show this is a bijection, we do a change of coordinates:
\begin{align*}(p,q,n,\epsilon, \mu) &= p[1,0,0,0,0]+q [0,1,1,0,0] \\
& +n[0,0,1,0,1] +  \epsilon [0,0,0,1,1]  +\mu [0,0,0,0,1] \\
&= [p,q,q+n,\epsilon,n+\epsilon+\mu].\end{align*}
 That is,
  \[ (p,q,n,\epsilon,\mu) = (p,q_0, q_1-q_0, \epsilon_0,\epsilon_1+\epsilon_0+q_1-q_0).\]

We need to show the map to $RO(\Pi P)$ is surjective. Let $\gamma \in RO(\Pi P)$ be arbitrary. Then up to isomorphism 
$\gamma(b)=C_2\times \R^p$,
$\gamma(b_0)=\R^{p_0,q_0}$, and $\gamma(b_1)=\R^{p_1,q_1}$  for some integers $p,p_0,p_1, q_0,q_1$. Since $i^*_eP$ is path connected, $p=p_0=p_1$.
Next we study $\gamma$ on morphisms. 
For $i=0,1$, let 
\[\gamma_a(p_i) \colon a \times \R^p \to i^*_e\R^{p,q_i}\]
where $a \in \{e,\gen\}$ denotes the orthogonal linear maps on fibers induced by the bundle map $\gamma(p_i)$. By equivariance for the action of $\gen$, we have
\[\gamma_\gen(p_i) = \gen \circ\gamma_e(p_i)\]
where $\gen$ denotes the action on $\R^{p,q_i}$.
We also let $\gamma_a(t)\colon a \times \R^p \to \gen a \times \R^p $ be the map on fibers, and similarly for $\gamma_a(g) \colon a \times \R^p \to  a \times \R^p $. By equivariance, we have
\[ \gamma_\gen(t) = \gamma_e(t) \quad \quad \text{and} \quad \quad \gamma_\gen(g) = \gamma_e(g).\]

The relations listed in \eqref{eq:rel1} and \eqref{eq:rel2} for the morphisms in $\Pi P$ give relations among the values of $\gamma$.
Because $\gamma(p_0)\circ \gamma(g)=\gamma(p_0)\circ \gamma(t)$, we have that
 \[\gamma_e(p_0)\circ \gamma_e(g)=\gamma_\gen (p_0)\circ \gamma_e(t)=\gen \circ \gamma_e(p_0) \circ \gamma_e(t)\]
and thus we get
\[
\det(\gamma_e(g))=\det(\gen)\det(\gamma_e(t))=(-1)^{q_0}\det(\gamma_e(t)).
\]
Since $\gamma(p_1)=\gamma(p_1)\circ \gamma(t)$, we also have
\[\gamma_e(p_1)=\gamma_\gen(p_1)\circ\gamma_e(t)=\gen\circ\gamma_e(p_1)\circ\gamma_e(t)\]
and thus
\[
(-1)^{q_1}=\det (\gamma_e(t)).
\]
It follows from these relations and the equivariance of $\gamma(g)$ and $\gamma(t)$ that the integers $p$, $q_0$ and $q_1$ completely determine $\gamma(t)$ and $\gamma(g)$, namely by
\[\det (\gamma_e(t))=(-1)^{q_1} \quad \quad \text{and} \quad \quad \det (\gamma_e(g))=(-1)^{q_0+q_1}.\]

Next we look at $\gamma(g_i) =((-1)^{\epsilon_i},(-1)^{\delta_i}) \in O(1)\times O(1)$ for $i=0,1$.
Since $\gamma(g_i)\circ \gamma(p_i)=\gamma(p_i)\circ \gamma(g)$, and thus
\[i_e^*\gamma(g_i)\circ \gamma_e(p_i)=\gamma_e(p_i)\circ \gamma_e(g), \]
we have that
\[
    \det i_e^*\gamma(g_i)=\det \gamma_e(g)=(-1)^{q_0+q_1}.
\]
Since $i^*_e(g_i) = (-1)^{\epsilon_i+\delta_i}$, it follows that we have $\delta_i = \epsilon_i+q_0+q_1$ for $i=0,1$.

It is not hard to check that the choice of $\gamma(p_0)$ and $\gamma(p_1)$ can be made so that $\gamma_e(p_i)=1$ up to natural isomorphism. 
It follows that the data of $\gamma$ is completely determined by the tuple $[p,q_0,q_1,\epsilon_0,\epsilon_1]$, showing surjectivity.

Now for injectivity, if two representations $\gamma$ and $\gamma'$ in the image of the map to $RO(\Pi P)$ are isomorphic, they must agree on objects (up to isomorphism). So we must have the same values for $p$, $q_0$, and $q_1$. By considering morphisms, we can show they have the same values for $\epsilon_0$ and $\epsilon_1$ as well.

If $\eta \colon \gamma \to \gamma'$ is a natural isomorphism, then there is a commutative diagram
\[
\xymatrix{
\R^{p,q_0} \ar[r]^-{\eta(b_0)} \ar[d]_-{\gamma(g_0)} & \R^{p,q_0} \ar[d]^-{\gamma'(g_0)} \\
\R^{p,q_0} \ar[r]^-{\eta(b_0)} & \R^{p,q_0}.
}\]
Restricting to the subspace $\R^{q_0,q_0}$, we see that $\epsilon_0=\epsilon_0'$. Similarly, $\epsilon_1 = \epsilon_1'$. Therefore, the tuple $[p,q_0,q_1,\epsilon_0,\epsilon_1]$ uniquely determines the isomorphism class of 
$\gamma$, completing the proof of injectivity. 
\end{proof}

 \begin{rem}\label{rem:whythegrading}
 The map $\crush  \colon P \to C_2/C_2$
induces a homomorphism 
\[\crush^*\colon RO(C_2) \to RO(\Pi P)\]
given by
\[\crush^*(p,q) = (p,q,0,0,0).\]
This makes it easy to isolate the $RO(C_2)$-graded cohomology inside of the $RO(\Pi P)$-graded cohomology. 
Furthermore, 
\[RO(\Pi i^*_eP) \cong \Z \times \Z/2 \cong \Z \times \Hom(\pi_1 (i^*_eP), O(1)) \]  
where the tautological line bundle $\bun_1$ has degree $(1,1)$ and the trivial line bundle $\varepsilon_1$ has degree $(1,0)$.
The homomorphism
\[i^*_e \colon RO(\Pi P) \to RO(\Pi i^*_eP) \]
is given by
\[i^*_e(p,q,n,\epsilon, \mu)=(p, n).\]
\end{rem}
\begin{rem}
It will be useful to treat the coordinates as functions. For example, we write $\mu(\gamma)\in \Z/2$ for the $\mu$ coordinate of $\gamma\in RO(\Pi P)$.
\end{rem}

For the fundamental groupoid of the fixed points, the computation of each $RO(\Pi P_i)$ proceeds exactly as that of \cref{thm:ROPiP}, so we omit the details. The description of  the restrictions is also a straightforward consequence.
\begin{theorem}\label{thm:ROPIPi}
For $i=0,1$, there are isomorphisms
\[RO( \Pi P_i) \cong \Z^2 \times (\Z/2)^2\]
which assign to the tuple $(p,q,n,\epsilon)$ the representation $\gamma$ given on objects by
\begin{align*}
\gamma(b_i) &= \R^{p,q} & \gamma(b_i\crush) = C_2\times \R^p 
\end{align*}
and on morphisms
   \begin{align*}
   \gamma(h_i) &= e\times (-1)^{n} & 
   \gamma(t) &= \gen \times(-1)^{q} \\
   \gamma(c_i) &=1 &   
   \gamma(g_i)&=\begin{pmatrix}(-1)^{\epsilon}, (-1)^{n+\epsilon})
   \end{pmatrix}
   \end{align*}
Furthermore, the maps
\[\iota_i^* \colon RO(\Pi P) \to RO(\Pi P_i) \]
are given by:
\begin{align*}
\iota_0^*(p,q,n,\epsilon,\mu)&=(p,q,n,\epsilon)\\
\iota_1^*(p,q,n,\epsilon,\mu)&=( p,q+n,n,n+\epsilon+\mu).
\end{align*}
\end{theorem}

\subsection{The subgroup $KO(\Pi P)$}\label{sec:KOPiP}
The goal of this section is to identify $KO(\Pi P)$ with the subgroup of $RO(\Pi P)$ where $\mu=0$. Recall that $P =B_{C_2}O(1)$ is an equivariant classifying space. Let $\Gamma = C_2\times O(1)$.
We apply the strategy explained in \cref{sec:KOPiB}. Applying \cref{lem:completiondiagram} and \cref{lem:Gamma}  
we have the commutative diagram
\begin{align*}
\xymatrix{
KO_G(P) \ar[d]_-\dim \ar[r]^-{c} & RO(\Gamma)^{\wedge}_{I(\Gamma)} \ar[d]^-{\dim} \\
RO(\Pi P) \ar[r]^-{c} & RO(\Pi P)^{\wedge}_{I(G)}.
}
\end{align*}
We will prove that $c$ is injective so that we can identify the subgroup $KO(\Pi P) =\dim (KO_G(P))$ by studying $\dim(RO(\Gamma)^{\wedge}_{I(\Gamma)})$.

We first recall the computation of $RO(C_2)^{\wedge}_{I(C_2)}$. As usual,  $\R^{1,1} \in RO(C_2)$ denotes the one-dimensional sign representation. Let $x = \R^{0,1} = \R^{1,1}-1$. Then $I(C_2)$ is the ideal $(x)$ in
\[RO(C_2 ) \cong \Z[x]/(x^2+2x).\]
For $k\geq 2$,
\[x^k = (-2)^{k-1}x.\]
Let $\Z_2$ denote the $2$-adic integers. Then the completion is the subring
\[RO(C_2)^{\wedge}_{I(C_2)}  \subset \Z_2[x]/(x^2+2x), \] 
which additively corresponds to the subgroup
\[\Z \{1\} \oplus \Z_2\{x\}.\]

\begin{lem}
There is an isomorphism $RO(\Pi P)^{\wedge}_{I(C_2)}  \cong \Z \times \Z_2^2 \times (\Z/2)^2$, and the image of $RO(\Pi P)$ is the subgroup $\Z \times \Z^2\times (\Z/2)^2$. 
In particular, the completion map
\[ c\colon RO(\Pi P) \to RO(\Pi P)^{\wedge}_{I(C_2)}\]
is injective.
\end{lem}
\begin{proof} 
Tensoring with $\R^{1,1}$ gives a $C_2$-action on $RO(\Pi P)$ as in \cref{rem:chi}, whose effect
in the coordinates of \cref{thm:ROPiP} is given by
\[(p,q,n,\epsilon,\mu)\mapsto(p,p-q,-n,\epsilon+n,\mu)\]
The action of $x=\R^{1,1}-1$ is thus given by 
\begin{align*}x(p,q,n,\epsilon, \mu)
 &= (p,p-q,-n,\epsilon+n,\mu) - (p,q,n,\epsilon, \mu)  \\
&= (0,p-2q,-2n,n,0) .
 \end{align*}
 Observe that $x$ acts invariantly on several subspaces: the first two coordinates, the third and fourth coordinates, and the last coordinate. 
 It follows that $RO(\Pi P)$ splits as the direct sum of $RO(C_2)$-modules
\begin{equation}\label{eq:decompositionsubmodules}
 RO(\Pi P) \cong RO(C_2) \oplus M \oplus N
 \end{equation}
 where $RO(C_2)$ corresponds to the image of $\crush^*$, i.e., coordinates with $n=\epsilon=\mu=0$. The module $M= \Z \times \Z/2$ corresponds to $p=q=\mu=0$, with action $x[n,\epsilon] = [-2n,n]$ and $N=\Z/2$ is the module with $p=q=n=\epsilon=0$ with $x$ acting as zero.
 
 The completion commutes with finite direct sums so we can complete each factor separately. We know how to complete $RO(C_2)$ as discussed above. The completion of $N$ is itself. For the completion of $M$, note that for $k\geq 2$,
 \[x^k[n,\epsilon] = (-2)^{k-1}x[n,\epsilon]
 = [(-2)^kn, (-2)^{k-1}n] = [(-2)^kn, 0]=(-2)^k[n,\epsilon] .\]
 Therefore,
 \[M^\wedge_{I(C_2)} = \Z_2\times \Z/2.\]
 Notice the completion map $\Z \to \Z_2$ is injective. In fact, the completion map 
 \[ c\colon RO(\Pi P) \to RO(\Pi P)^{\wedge}_{I(C_2)}\]
 is injective on each component. 
 This finishes the proof. 
 \end{proof}

\begin{lem}
Let $\Gamma = C_2 \times O(1)$.
There is a bijection between $C_2$-line bundles over $P$ and homomorphisms $\Gamma \to O(1)$. 
\end{lem}
\begin{proof}
By \cite[Theorem 5]{May_some}, since $O(1)$ is discrete, the space $P$ is equivalent to the $C_2$-space 
\[P \simeq B_{C_2}O(1) \simeq \Map(EC_2, BO(1)).\]
Line bundles over $P=B_{C_2}O(1)$ are classified by
\begin{align*}
[B_{C_2}O(1),  B_{C_2}O(1)]^{C_2} &\cong [B_{C_2}O(1),  \Map(EC_2, BO(1))]^{C_2} \\ 
 &\cong [B_{C_2}O(1)\times_{C_2}EC_2,  BO(1)] \\
  &\cong [B\Gamma,  BO(1)] .
\end{align*}
Since $\Gamma$ and $O(1)$ are finite groups, these homotopy classes of maps are in one-to-one correspondence with conjugacy classes of group homomorphisms $\Gamma \to O(1)$.
Finally, since $O(1)$ is abelian, these are the same as homomorphisms.
\end{proof}

\begin{remark}\label{rem:linebundles}
Continuing to let $\Gamma = C_2 \times O(1)$, this implies there are four line bundles on $P$ up to isomorphism. We describe them here and compute their dimensions. Realize $P$ as the quotient of the unit sphere $E_{C_2}O(1):=S(\R^{2\infty,\infty})$ by the antipodal action of $O(1)$. Given any homomorphism $\phi \colon \Gamma \to O(1)$, we can form the $C_2$-line bundle
\[ S(\R^{2\infty,\infty})\times_{O(1)} \R^\phi \to S(\R^{2\infty,\infty})/O(1)= P.\]
Here the balanced product is over the subgroup $O(1)\subset \Gamma$, which acts on $S(\R^{2\infty,\infty})$ via the antipodal action. The $\Gamma$-representation $\R^\phi$ has underlying vector space $\R$ and action determined by the homomorphism $\phi$ via the standard action of $O(1)$ on $\R$. Running through the possibilities for $\phi$, we obtain the following isomorphism classes of $C_2$ line bundles:
\begin{enumerate}[(a)]
\item The trivial line bundle $\triv$ corresponds to the trivial homomorphism and
\[\dim \triv= 
(1,0,0,0,0).\]
\item The constant line bundle $\sign$ at the sign representation corresponds to the projection $\Gamma \to C_2$  and
\[\dim \sign=
(1,1,0,0,0). \]
\item The tautological line bundle $\taut$ corresponds to
the projection $\Gamma \to O(1)$ and has dimension
\[\dim \taut =
(1,0,1,0,0).\]
\item The tensor product $\staut = \sign \otimes \taut$ (or $\chi \taut$) corresponds to the homomorphism $\Gamma \to O(1)$ which is an isomorphism when restricted to both of the inclusions $C_2,O(1)  \subset \Gamma$ and
\[\dim \staut= 
(1,1,-1,1,0).  \] 
\end{enumerate}
\end{remark}

We can now prove the main result of this section.

\begin{theorem}\label{thm:KOPiP}
The subgroup $KO(\Pi P)\subset RO(\Pi P)$ is equal to the subgroup of those representations $\gamma \in RO(\Pi P)$ such that $\mu(\gamma)=0$. Furthermore, this coincides with the subgroup generated by $\dim \gamma$ where $\gamma$ runs over the isomorphism classes of $C_2$-line bundles over $P$.
\end{theorem}
\begin{proof}
The fact that $KO(\Pi P)$ contains any $\gamma$ with $\mu(\gamma)=0$ follows from the fact that the dimensions of line bundles listed above generate this subgroup.
It remains to prove that an element the image of
\[ KO_{C_2}(P \times EC_2)  \cong RO(\Gamma)^\wedge_{I(\Gamma)} \to RO(\Pi P)^{\wedge}_{I(C_2)}   \]
has a vanishing $\mu$-coordinate. Notice that $\Gamma=C_2\times O(1)$ has the property that $\RO(\Gamma)$ is generated by its one-dimensional irreducible representations. It follows that the image of $RO(\Gamma)^\wedge_{I(\Gamma)} $ in $RO(\Pi P)^{\wedge}_{I(C_2)} $ is the completion of the subgroup generated by the line bundles. In particular, this image is contained in the subgroup of those $\gamma$ such that $\mu(\gamma)=0$. Since
\[ c(KO(\Pi P))\subset \dim(RO(\Gamma)^\wedge_{I(\Gamma)}) \subset RO(\Pi P)^{\wedge}_{I(C_2)} \]
and $c$ is injective, any element in $ KO(\Pi P)$ has vanishing $\mu$ coordinate.
\end{proof}

\begin{rem}
 In fact, any vector bundle over $P$ splits as a direct sum of line bundles. The proof, as well as a shorter argument for \cref{thm:KOPiP}, is in preparation in \cite{BZ}.  As that approach does not generalize much beyond this example, in the spirit of providing tools for further computations, we have kept this proof.  
 \end{rem}

%% file: ppaper-cohthom.tex

\section{The $RO(C_{2})$-graded cohomology of Thom spaces}
\label{section:cohomologyofthomspaces}
Our computation of the $RO(\Pi P)$-graded cohomology of $P$ will rely on knowing the $RO(C_2)$-graded cohomology of various Thom spaces for bundles over $P$ and applying \cref{thm:KOPIB}. We compute the cohomology of the relevant Thom spaces in this section. 

\subsection{Thom spaces as stunted projective spaces} 
Nonequivariantly, certain Thom spaces over $\R P^\infty$ can be identified with stunted projective spaces \cite[Prop. 4.3]{AtiyahThom}. Namely,
\[\Th(m\bun_1) \simeq \R P^\infty / \RP^{m-1}\]
where $\bun_1$ is the tautological line bundle over $\R P^\infty$ and $m\gamma_1$ is the direct sum of $m\geq 1$ copies.
The proof generalizes to the equivariant context, as explained below. This is likely well-known. See for example \cite[\S7.1]{2025arXiv251109816B}.

Let $G$ be a finite group. We will discuss the construction of various $G$-bundles. To avoid confusion when $G=C_2$, we write 
$O(1)$ for the structure group for real line bundles.
Let $\sigma$ denote the one-dimensional sign representation of $O(1)$ and $V$ be an actual orthogonal $G$-representation. Let $S(V)$ be the unit sphere in $V$ and $P(V)$ its projective space. 
Then $V\otimes \sigma$ is a $G \times O(1)$ representation with the diagonal action, and
\[P(V) \cong S(V\otimes \sigma)/O(1) = S(V\otimes \sigma)\times_{O(1)}\pt.\]
For any actual orthogonal representation $W$ of $G$, we can form a bundle
\[ \xymatrix{
 &  S(V\otimes \sigma)\times_{O(1)} (W\otimes\sigma) \ar[d]^-{\gamma^V_W} \\ 
 & P(V)} \]
For example, if we let $\R$ be the trivial representation of $G$, then  $\gamma^V_\R$
is the tautological line bundle of the projective space $P(V)$. The Thom space of $\gamma^V_W$ can be described as follows:
 \begin{lem}
 \label{lem:thomspace}
 Let $V$ and $W$ be actual orthogonal representations of a finite group $G$.
 There is a homeomorphism 
 \[ \Th(\gamma_W^V) \cong P(V\oplus W)/P(W)\]
 where the quotient is along the canonical inclusion $P(W) \to P(V\oplus W)$. The same formula holds if $V=\colim_{n\in \N} V_n$ for finite dimensional $V_n$.
 \end{lem}
 \begin{proof}
The map
 \[f\colon S(V\otimes \sigma)\times D(W\otimes \sigma) \to S(V\otimes \sigma\oplus W\otimes \sigma) \]
determined by
 $f(v,w) = ((\sqrt{1-\|w\|^2})  v, w)$ induces the desired homeomorphism of $G$-spaces. For the infinite dimensional case, we use that the Thom space construction and the quotient commute with the colimit.
 \end{proof}

We return to the example of interest, that is $G=C_2$ and $P=P(\R^{2\infty,\infty})$. We consider bundles over $P$ and its subspaces $P^{p,q}=P(\R^{p,q})$ for $p \geq q \geq 0$.
We define $P_{p,q}$ as the quotient
\[ P^{p,q} \to P \to P_{p,q}.\]
In the notation above, define 
\[\bun_{p,q} = \gamma_{\R^{p,q}}^{\R^{2\infty,\infty}}.\]
This is a $C_2$-bundle over $P$.

\begin{example}\label{ex:equiv}
We saw two of these in \cref{rem:linebundles}, the line bundles $\bun_{1,0}$ and $\bun_{1,1}$. Since $P^{1,0}$ and $P^{1,1}$ are points,
\[P \simeq \Th(\bun_{1,0}) \simeq \Th(\bun_{1,1}). \]
\end{example}

\begin{cor}\label{cor:pqduality}
For $p \geq q \geq 0$, there is a $C_2$-equivariant homeomorphism
\[\Th(\bun_{p,q}) \cong P_{p,q}. \]
Furthermore, there is a $C_2$-equivariant homeomorphism
\[ \Th(\bun_{p,q}) \cong \Th(\bun_{p,p-q}).\]
\end{cor}

\begin{proof}
The first part follows immediately from \cref{lem:thomspace} applied to $\bun_{p,q}$. It remains to prove the second part of the claim.
 There is a commutative diagram
\[\xymatrix{
P^{p,q} \ar[r] \ar[d]^-\cong & P(\R^{p,q}\oplus 
\R^{2\infty,\infty} ) \ar[r] \ar[d]^-\cong & P_{p,q}\ar@{.>}[d] \\
P^{p,p-q} \ar[r] & P(\R^{p,p-q}\oplus 
\R^{2\infty,\infty} ) \ar[r] & P_{p,p-q}
}\]
where the vertical homeomorphisms are induced by the (non-equivariant) map
\[ \R^{p,q} \to \R^{p,p-q},  \]
which sends $(v_0,v_1)$ to $(v_1,v_0)$ (see \cref{not:vector_space}), applied component wise for the middle term.   After projectivization, these give $O(1)$-equivariant equivalences. Note that this is essentially just the automorphism $\autP$ of \cref{rem:chi} and its restriction.
It follows that
\[\Th(\bun_{p,q}) \cong P_{p,q} \cong P_{p,p-q} \cong \Th(\bun_{p,p-q}).  \qedhere\]
\end{proof}

\subsection{$RO(C_2)$-graded cohomology of $P$}\label{sec:ROC2cohomofP}
In this section we compute the $RO(C_2)$-graded cohomology of the $C_2$-space $P$ with coefficients in the constant Mackey functor $\uF_2$.  The computation is well-known, and is obtained by mimicking the motivic computation by Voevodsky \cite{MR2031199}. 
It can be found in the equivariant context  \cite[Thm. 4.2]{Kronholm}. We repeat it in detail here in order to demonstrate some useful computational techniques.  We will describe the cohomology of $P$ as an $\M_2$-algebra,
but first describe the $RO(C_2)$-graded cohomology of $P$ as an $\M_2$-module.  We use the fact that $P$ can be given the structure of a $\Rep(C_2)$-complex (as well as that of a $C_2$-CW-complex).

\begin{definition}
    A \emph{representation cell} is the unit disk $D(V)$ in an actual orthogonal $G$-representation $V$. If $n = \dim V$ then $D(V)$ has dimension $n$, the underlying dimension of the cell. 
    \end{definition}

\begin{definition}
    A \emph{$\Rep(G)$-complex} is a $G$-space $X$ with a filtration 
    \[
    X_0 \subseteq X_1 \subseteq X_2 \subseteq \cdots \subseteq X
    \]
    such that $X_0$ is a disjoint union of trivial orbits $G/G$ and $X_n$ is obtained from $X_{n-1}$ by attaching representation cells of dimension $n$ along their boundaries.  Thus $X_n$ is formed via the pushout 
    \begin{center}
    \begin{tikzcd}
        \bigsqcup_\alpha S(V_\alpha) \ar[r] \ar[d] & X_{n-1} \ar[d]\\
        \bigsqcup_\alpha D(V_\alpha) \ar[r] & X_n.
    \end{tikzcd}
    \end{center}
\end{definition}
One advantage of working with $\Rep(G)$-complexes is that the filtration quotients are wedges of representation spheres $X_n/X_{n-1} \simeq \bigvee_\alpha S^{V_\alpha}$.

In the context of $G=C_2$, up to isomorphism, all (actual) representations are of the form $\R^{p,q}$ for $p \geq q \geq 0$, and so representation cells are of the form $D(\R^{p,q})$. We say that a $\Rep(C_2)$-complex has \emph{finite type} if for each fixed-set dimension $s$, it has finitely many representation cells $D(\R^{p,q})$ with $s = p-q$ (and hence also finitely many cells of each topological dimension). 

The cohomology of a finite type $\Rep(C_2)$-complex is free as an $\M_2$-module.

\begin{theorem}[Hogle--May \cite{HogleMay}, Kronholm \cite{Kronholm}]
    Let $X$ be a $\Rep(C_2)$-complex of finite type. Then $H^{*,*}(X,\mF)$ is free as a graded $\M_2$-module.
\end{theorem}

We may apply this to $P$.  Note that
$P=Gr_1(\R^{2\infty,\infty})$ and that $\R^{2\infty,\infty}$
is a complete $C_2$-universe. Viewing  $\R^{2\infty,\infty}$ as alternating direct sums of $\R^{1,0}$ and $\R^{1,1}$,
the Schubert cell decomposition from \cite{Dugger_grassmannian} describes $P$ as a $\Rep(C_2)$-complex with a basepoint and representation cells $D(\R^{2n-1,n})$ and $D(\R^{2n,n})$ for $n\geq 1$ as below, where $\R^{1,0}$ is denoted by $+$ and $\R^{1,1}$ by $-$.
\begin{align*}
&\!\! + - + - + - + - \cdots && \\ 
& [1, 0, 0, 0,0,0,\dots, 0,\dots] && D(\R^{0,0})\\
&[*,1,0,0,0,0,\dots,0,\dots] && D(\R^{1,1})\\
&[*,*,1,0,0,0,\dots,0,\dots] && D(\R^{2,1})\\
&[*,*,*,1,0,0,\dots,0,\dots] && D(\R^{3,2})\\
&[*,*,*,*,1,0,\dots,0,\dots] && D(\R^{4,2})\\
  &\hspace{2cm} \vdots  && \qquad  \vdots 
\end{align*}

This gives a $\Rep(C_2)$-complex structure for $P$ with one representation cell of each topological dimension and at most two cells of any fixed-set dimension.  Thus the $RO(C_2)$-graded cohomology $H^{*,*}(P,\mF)$ is free as an $\M_2$-module and we need only determine the degrees of the generators. Note that the freeness of the cohomology of $P$ also follows directly from the collapse of the associated cellular spectral sequence in this particular case.

A portion of the $RO(C_2)$-graded cellular spectral sequence associated to this $\Rep(C_2)$-complex filtration is depicted in \cref{fig:PcellularSS}. 
 Only the outline of each copy of $\M_2$ is shown and the stair-step pattern continues up and to the right.  The differential $d$ has degree $(1,0)$. Since each filtration is a free module and the differential in the spectral sequence is a module map, $d$ is determined by where it sends each generator. We may depict the cellular spectral sequence even more succinctly by only drawing the generators of each $\M_2$ and suppressing the other classes as in the right image in \cref{fig:PcellularSS}.

\begin{figure}
    \centering
\includegraphics[width=0.48\textwidth]{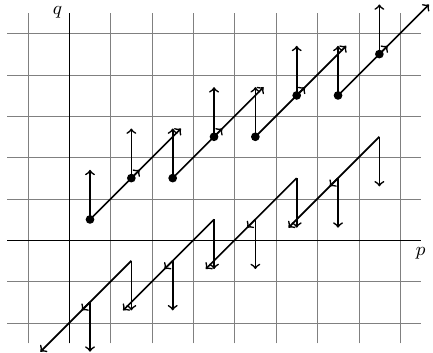}
\includegraphics[width=0.48\textwidth]{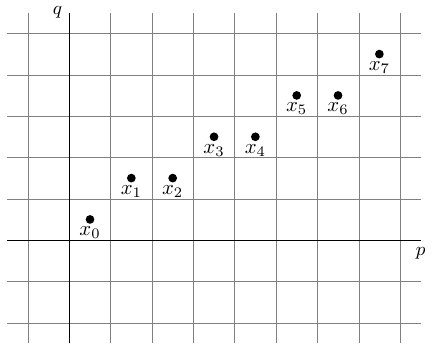}
    \caption{(Left) The $RO(C_2)$-graded cellular spectral sequence for $H^{*,*}(P,\mF)$. (Right) Generators in the $RO(C_2)$-graded cellular spectral sequence for $H^{*,*}(P,\mF)$}
    \label{fig:PcellularSS}
\end{figure}

For degree reasons, the differential on each generator is either $0$ or an isomorphism. Since $i^*_eP=\R P^\infty$,  the cohomology of the underlying infinite projective space $H^*(i^*_eP,\F_2) \cong \F_2[w_1]$, with degree $|w_1| = 1$.  Applying the forgetful long exact sequence of \cref{sec:forget} to $H^{*,*}(P,\mF)$, we inductively prove that none of the generators support a differential. In fact, this argument shows that the forgetful map 
\[i^*_e \colon H^{*,*}(P,\mF) \to H^*(i^*_eP,\F_2) \cong \F_2[w_1]\] 
sends each generator to the appropriate power of $w_1$.

We conclude the following:
\begin{prop}[{\emph{c.f.} \cite[Prop. 4.3]{Kronholm}}]\label{prop:gendegrees}
The $RO(C_2)$-graded cohomology $H^{*,*}(P,\mF)$ is free as an $\M_2$-module with one generator in each bidegree
$(0,0)$, $(2n-1,n)$, and $(2n,n)$ for $n \geq 1$. We denote by $x_p$ the generator in topological degree $p$ for each $p>0$ (and bidegree $(p,\lceil p/2 \rceil)$. Further,
$i^*_e(x_p)=w_1^p$.
\end{prop}

It remains to discuss the ring structure of $H^{*,*}(P,\mF)$. It is generated as an $\M_2$-algebra by the classes $x_1$ in $H^{1,1}(P,\mF)$ and $x_2$ in $H^{2,1}(P,\mF)$.  To show this, we use the Bockstein $\Sq^{1,0}$, which appeared in \cref{sec:bockstein}.

For degree reasons $x_1^2$ is a linear combination of $\aa^2$, $\aa x_1$, and $\uu x_2$.  So we write
\[
x_1^2 = c_1 \aa^2 + c_2 \aa x_1 + c_3 \uu x_2,
\]
where each $c_i \in \{0,1\}$.  If $c_1 \neq 0$ then
\[
(x_1 + \aa)^2 = x_1^2 + \aa^2 = c_2 \aa x_1 + c_3 \uu x_2.
\]
So up to a change of basis, replacing $x_1$ with $x_1+\aa$, we can assume $c_1 = 0$.

We now show $c_2 = c_3 = 1$.  Since $i^*_e(x_1) = w_1$ and $i^*_e(x_2) = w_1^2$, and since $i^*_e$ is multiplicative,
\[
i^*_e(x_1^2) = i^*_e(x_1)^2 = w_1^2
\]
so $c_3 \neq 0$.
Finally, $\Sq^{1,0}(x_1^2) = 0$, and one can show using $\Sq^{1,0}(\uu) = \aa$, $\Sq^{1,0}(x_1) = x_2$, and the Leibniz rule that
\[
\Sq^{1,0}(c_2 \aa x_1 + \uu x_2) = c_2 \aa x_2 + \aa x_2,
\]
so $c_2$ must be $1$.

It is clear for degree reasons, making use of the forgetful map, that all even-dimensional generators are powers of $x_2$.  Similarly, all odd-dimensional generators must be of the form $x_1x_2^k$ for some $k$.  This completes the computation of $H^{*,*}(P,\mF)$ as a ring.

\begin{thm}[{\emph{c.f.} \cite[Thm. 4.2]{Kronholm}}]\label{thm:cohPROG}
    As an $\M_2$-algebra, 
    \[
    H^{*,*}(P,\mF) \cong \M_2[x_1,x_2]/(x_1^2 - (\aa x_1 + \uu x_2)),
    \]
    where $x_1$ has degree $(1,1)$ and $x_2$ has degree $(2,1)$ (see \cref{fig:ROGcohomP}).
\end{thm}

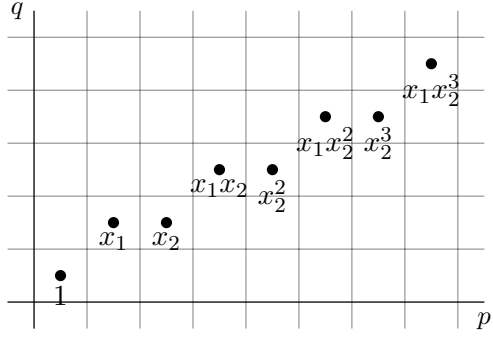
\begin{figure}
    \centering
    \begin{tikzpicture}[scale=0.7]
\draw[step=1cm,gray,very thin] (-0.5,-0.5) grid (8.5,5.5);
\draw[] (-0.5,0) -- (8.5,0) node[below, black] {\small $p$};
\draw[] (0,-0.5) -- (0,5.5) node[left, black] {\small $q$};

\fill (0.5,0.5) circle(3pt) node[below] {$1$};
\fill (1.5,1.5) circle(3pt) node[below] {$x_1$};
\fill (2.5,1.5) circle(3pt) node[below] {$x_2$};
\fill (3.5,2.5) circle(3pt) node[below] {$x_1x_2$};
\fill (4.5,2.5) circle(3pt) node[below] {$x_2^2$};
\fill (5.5,3.5) circle(3pt) node[below] {$x_1x_2^2$};
\fill (6.5,3.5) circle(3pt) node[below] {$x_2^3$};
\fill (7.5,4.5) circle(3pt) node[below] {$x_1x_2^3$};

\end{tikzpicture}
    \caption{Generators in $RO(C_2)$-graded cohomology $H^{*,*}(P,\mF)$}
    \label{fig:ROGcohomP}
\end{figure}

\subsection{$RO(C_2)$-graded cohomology of Thom spaces}\label{sec: stunted proj}

In this section, we compute the $RO(C_2)$-graded cohomology of the Thom spaces
\[H^{*,*}(\Th(\bun_{p,q}), \uF_2)  \]
for $p \geq q \geq 0$.
We show the $RO(C_2)$-graded cohomology of each Thom space $\Th(\gamma_{p,q})$ is free as an $\M_2$-module and identify the degrees of the generators. 
We focus on the case where $p-q\geq q$ and compute the cohomology of the remaining Thom spaces via the isomorphism
\[H^{*,*}(\Th(\bun_{p,q}), \uF_2) \cong H^{*,*}(\Th(\bun_{p,p-q}), \uF_2) \]
from \cref{cor:pqduality}.

\begin{rem}\label{rem:evenunique}
As noted in \cref{ex:equiv}, $P \simeq \Th(\bun_{1,0})$, so the case $p=1$ is \cref{thm:cohPROG}, 
\[H^{*,*}(\Th(\bun_{1,0}),\uF_2) \cong \M_2[x_{1},x_{2}]/(x_{1}^2-(\aa x_{1}+\uu x_{2})),\]
where $x_{1}\in H^{1,1}$ and $x_{2} \in H^{2,1}$. 
We let
\begin{align*}
x_{2n+1} &= x_{1} x_{2}^n\\
x_{2n+2} &= x_{2}^{n+1}
\end{align*}
with $x_p \in H^{*,*}(\Th(\bun_{1,0}))$ in topological degree $p$.  
While the generators in even degrees are uniquely determined by their topological degree, in odd degrees, both $x_{2n+1}$ and  $x_{2n+1}+\aa x_{2n}$ are generators.
\end{rem}

Let $\bun_p$ denote the direct sum of $p$-copies of the tautological line bundle over $i^*_eP =\R P^\infty$. We have a Thom isomorphism
\[H^*(\R P^\infty,\F_2) \xrightarrow[\cong]{t_p} \widetilde{H}^{*+p}(\Th(\bun_p),\F_2)\]
where $t_p$ is the Thom class. So, $H^*(\Th(\bun_p),\F_2)$ is generated by the unit and classes $w_1^nt_p$ in degrees $n+p$ for all $n\geq 0$.

\begin{figure}
\includegraphics[width=0.48\textwidth]{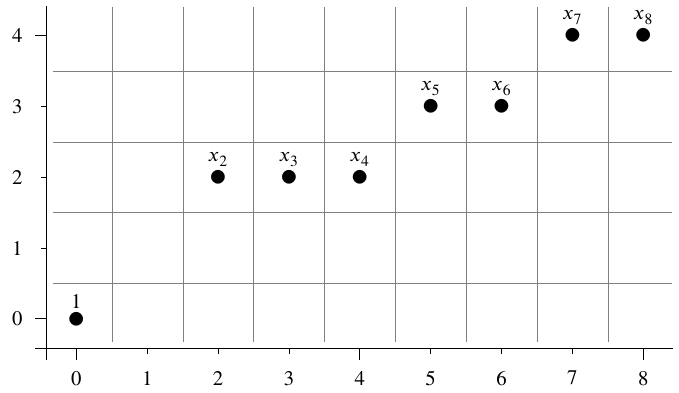}
\includegraphics[width=0.48\textwidth]{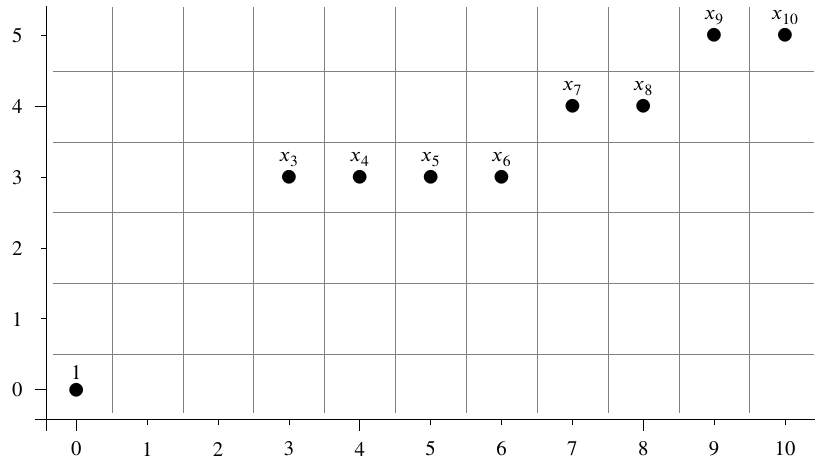}
\caption{$\M_2$-module generators for  $H^{*,*}(\Th(\bun_{p,q}),\uF_2)$ for $(2,0)$ (left)  and  $(3,0)$   (right) }
\label{fig:coh_20}
\end{figure}

\begin{figure}
\includegraphics[width=0.48\textwidth]{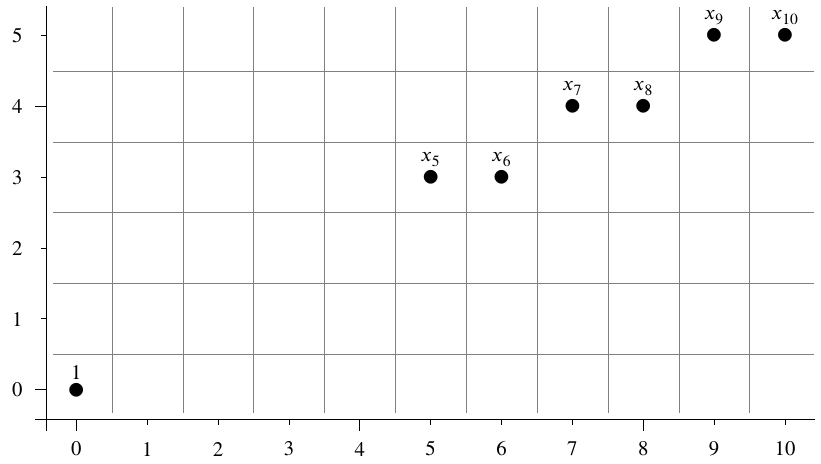}
\includegraphics[width=0.48\textwidth]{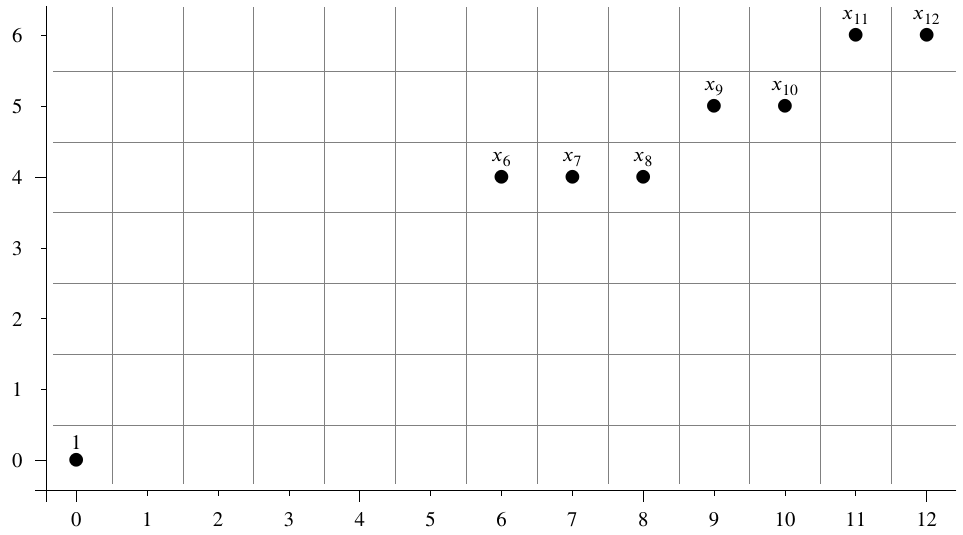}
\caption{$\M_2$-module generators for  $H^{*,*}(\Th(\bun_{p,q}),\uF_2)$ for $(p,q)=(5,2)$ (left)  and  $(p,q)=(6,2)$   (right) }
\label{fig:coh_52}
\end{figure}

\begin{theorem}
\label{thm:coh_pq}
For $p\geq q\geq 0$, we have:
\begin{enumerate}[(1)]
\item Assume $p-q\geq q$.
    The $RO(C_2)$-graded cohomology of $\Th(\bun_{p,q})$ is free as an $\M_2$-module generated by the unit and classes $x_r$ for $r\geq p$ of degree
\begin{align*}
|x_{r}| =\begin{cases} (r, p-q) & p\leq r \leq 2(p-q)\\
(r, \lceil r/2 \rceil) & r>2(p-q).
\end{cases}
\end{align*}

    \item  Assume $p-q \leq q$.
    The $RO(C_2)$-graded cohomology of $\Th(\bun_{p,q})$ is free as an $\M_2$-module generated by the unit and classes $x_r$ for $r\geq p$ of degree
\begin{align*}
|x_{r}| =\begin{cases} (r, q) & p\leq r \leq 2q\\
(r, \lceil r/2 \rceil) & r>2q.
\end{cases}
\end{align*}
    \end{enumerate}
        In both cases, the restriction 
    \[i^*_e \colon H^{*,*}(\Th(\gamma_{p,q}) ,\uF_2) \to H^{*}(\Th(\gamma_{p}) ,\F_2) \]
    sends $x_{r}$ to $w_1^{r-p}t_p$ for all $r\geq p$.
    
    See \cref{fig:coh_20} and \cref{fig:coh_52} for examples. 
\end{theorem}

\begin{proof}
We first consider the case $p-q\geq q$.
The strategy is to give a $\Rep(C_2)$-cell structure on $P_{p,q}$ that has the property that the successive boundary homomorphisms in the inductive long exact sequences building $P_{p,q}$ out of its cells are all zero, which makes the computation fairly simple. We then get a bijection between the cells and the $\M_2$-generators. 

To give $P_{p,q} = P/P^{p,q}$ the desired $\Rep(C_2)$-cell structure, we give $P$ a Schubert cell structure such that $P^{p,q}$ forms the $(p-1)$-skeleton as follows. Let $m=p-q \geq q$. 
Consider $P$ as the one-dimensional subspaces in the representation
\[
\underbrace{\R^{1,0} \oplus \cdots \oplus \R^{1,0} }_{m} \oplus \underbrace{\R^{1,1}  \oplus \cdots \oplus \R^{1,1}}_{m} \oplus \R^{1,0} \oplus \R^{1,1} \oplus \R^{1,0} \oplus\R^{1,1}\oplus \cdots
\]
with $m$ copies of the trivial representation $\R^{1,0}$ followed by $m$ copies of the sign representation $\R^{1,1}$, and then alternating trivial and sign.  The Schubert cells as described in \cite{Dugger_grassmannian} will be of the form depicted in \cref{fig:cells}.
Quotienting the $(p-1)$-skeleton gives a $\Rep(C_2)$-complex structure for $P_{p,q}$.  Using the forgetful map of \cref{sec:forget}, it is straightforward to see the attaching map for each representation cell must be trivial on $RO(C_2)$-graded cohomology with $\mF$-coefficients. Any nontrivial differential would kill off two copies of $\M_2$, giving the wrong underlying cohomology. 

For the case when $p-q\leq q$, recall the isomorphism $\Th(\bun_{p,q}) \cong \Th(\bun_{p, p-q})$ and apply the substitution $q\mapsto p-q$.
\end{proof}

\begin{rem}
As in \cref{rem:evenunique}, some of the $\M_2$-generators in the cohomology of $\Th(\gamma_{p,q})$ are uniquely determined by their topological degrees, and others are not, as one can modify them by elements in the image of multiplication with $\aa$.
\end{rem}

\begin{figure}
{\tiny
\begin{align*}
& + + + + \cdots + - - - - \cdots - + - + - \cdots\\ 
& [1,0,0,0,\dots,0,\dots] && D(\R^{0,0})\\
&[*,1,0,0,\dots,0,\dots] && D(\R^{1,0})\\
&[*,*,1,0,\dots,0,\dots] && D(\R^{2,0})\\
  & \vdots  & \vdots &  \\
 &[*,*,*,\dots, *, 1,0,\dots] && D(\R^{m-1,0}) \\
  &       [*,*,*,\dots, *, *,1,0,\dots] && D(\R^{m,m})\\
   &      [*,*,*,\dots, *, *,*,1,0,\dots] && D(\R^{m+1,m})\\
    &     [*,*,*,\dots, *, *,*,*,1,0,\dots] && D(\R^{m+2,m})\\
 & \vdots & \vdots & \\
&         [*,*,*,\dots, *, *,*,*,*,\dots,*,1,0,\dots] && D(\R^{2m,m})\\
&  [*,*,*,\dots, *, *,*,*,*,\dots,*,*,1,0,\dots] && D(\R^{2m+1,m+1})\\
&  [*,*,*,\dots, *, *,*,*,*,\dots,*,*,*,1,0,\dots] && D(\R^{2m+2,m+1})\\
& [*,*,*,\dots, *, *,*,*,*,\dots,*,*,*,*,1,0,\dots] && D(\R^{2m+3,m+2})\\
& [*,*,*,\dots, *, *,*,*,*,\dots,*,*,*,*,*,1,0,\dots] && D(\R^{2m+4,m+2})\\
& \vdots & \vdots
\end{align*}}
\caption{Schubert cell structure on $P$}
\label{fig:cells}
\end{figure}

%% file: ppaper-pcomputation.tex

\section{Parametrized cohomology of $P$}\label{sec:cohPall}
The goal of this section is to compute the $\RO(\Pi P)$-graded cohomology of $P$. Recall 
 from \cref{thm:ROPiP} that 
\[\RO(\Pi P) \cong \Z^3  \times (\Z/2)^2\]
and that we denote an element of this group by tuples $(p,q,n,\epsilon, \mu)$.
Recall that the subgroup coming from $\dim$ of virtual bundles, $KO(\Pi P)$, corresponds to those degrees with $\mu=0$.
We first compute the $KO(\Pi P)$-graded cohomology, so those gradings with $\mu=0$, using the Thom isomorphism from  \cref{thm:KOPIB}.  Then in \cref{sec:ROPIP}, we deal with degrees that are not in the image of $\dim$, those with $\mu\neq 0$.

\subsection{The $RO(\Pi i^*_e P)$-graded cohomology}
Before diving into the equivariant computation, we compute the underlying parametrized cohomology. This can also be found in \cite{Cadek}.
We recall from \cref{rem:whythegrading}
\[\RO(\Pi i^*_eP)  = \Z \times \Z/2\]
where the tautological line bundle $\bun_1$ has degree $(1,1)$. We denote
\[H^{*,*}_{i^*_e P}(i^*_e P, R) = H^{\RO(\Pi i^*_eP)}_{i^*_eP}(i^*_eP,R). \]
\begin{lem}\label{lem:underlyingPcoh}
There are isomorphisms
\[H^{*,*}_{i^*_e P}(i^*_e P, \Z) \cong \Z[a_{1}]/(2a_1).\]
where $a_1 \in H^{1,1}_{i^*_e P}(i^*_eP,\Z)$ is the Euler class of the tautological line bundle $\bun_{1}$, and
\[H^{*,*}_{i^*_e P}(i^*_e P, \F_2) \cong \F_2[a_1,u_1]/(u_1^2-1)\]
where $u_1 \in H^{0,1}_{i^*_e P}(i^*_eP,\F_2)$ is the orientation classes of $\bun_1$.
The class 
\[w_1=a_1u_1 \in H^{1,0}_{i^*_e P}(i^*_eP,\F_2) \cong H^{1}(i^*_eP,\F_2) \] 
is the classical first Stiefel--Whitney class of the tautological line bundle $\bun_1$. Furthermore, 
\[\Sq^{1}(u_1) = a_1.\]
\end{lem}
\begin{proof}
We know from \cref{cor:cohBtrivunits} that
\[H^{*,*}_{i^*_e P}(i^*_e P, \F_2) \cong \F_2[w_1] \otimes \F_2[u_1]/(u_1^2-1).\]
In the non-equivariant, non-parametrized context,  the Euler class is congruent to the top Stiefel--Whitney class modulo $2$.
This is $w_1$ for $\bun_1$. It follows from \cref{lem:eulerprop}(d)  that
\[ \crush_!(a_{\bun_1}) \equiv w_1 \mod 2\]
 is non-zero. We  must have $a_1 = u_1w_1$.  

We then use the isomorphism of \cref{thm:KOPIB}
\[ H^{*,1}_{i^*_e P}(i^*_eP,\Z) = H^{*,-1}_{i^*_e P}(i^*_eP,\Z) \xrightarrow{\cong} H^{*+1,0}_{i^*_e P}(\Th_{i^*_eP}(\bun_1),\Z) \cong H^{*+1}(\Th(\bun_1),\Z) \]
 to obtain the integral answer additively. The fact that $a_1\equiv u_1w_1$ modulo $2$ together with the ring structure of the mod 2 cohomology ring gives the multiplicative structure in the integral cohomology ring.
It then follows from the long exact sequence that
\[
\Sq^{1}(u_1) = a_1. \qedhere
\]
\end{proof}

\subsection{The $KO(\Pi P)$-graded cohomology}\label{sec:KOPII}
In this section, we will compute the $KO(\Pi P)$-graded cohomology of $P$, i.e., those degrees with $\mu=0$. 
We will use a four-tuple to denote the grading, dropping the last zero in our coordinates,
\[(p,q,n,\epsilon)= (p,q,n,\epsilon, 0).\]
Recall from \cref{rem:linebundles} that $KO(\Pi P)$ is generated by the dimensions of the four line bundles $\epsilon_{1,0}$, $\epsilon_{1,1}$, $\gamma_{1,0}$ and $\gamma_{1,1}$ and
\begin{align*}
\dim(\triv) &= (1,0,0,0) &\dim(\taut) &= (1,0,1,0) \\
\dim(\sign) &= (1,1,0,0) &\dim(\staut) &= (1,1,-1,1)  .
\end{align*}
In particular,
\[\dim(\bun_{p,q})= (p-q,0,p-q,0)+(q,q,-q,q) = (p, q, p-2q,q). \]
We denote the $KO(\Pi P)$-graded cohomology by
\[H^{*,*,*,*}_P(P, \mF  ) = H^{KO(\Pi P)}_P(P, \mF  ). \]

We start by identifying a key class in this algebra.
\begin{lem}
\label{lem:ecopy}
There is a unit $\eee =\ee_{\gamma_{2,1}} \in H^{0,0,0,1}_{P}(P, \mF  )$, such that multiplication by $\eee$ is an isomorphism
\[H^{p,q,n,1}_{P}(P, \mF  ) \xrightarrow[\cong]{e}H^{p,q,n,0}_{P}(P, \mF  ) \]
and $\eee^2=1$.
\end{lem}
\begin{proof}
The bundle $\xi=\gamma_{2,1}$ is is homogeneous with $\xi_0=\R^{2,1}$. Let $\eee= \ee_{\xi}$ be the homogeneity unit of \cref{defn:homounit}. Since $\eee^2$ is a unit in degree $(0,0,0,0)$, we must have  $\eee^2=1$.
\end{proof}

For $p \geq q \geq 0$, let $t_{p,q}=t_{\bun_{p,q}}$ be the Thom class of $\bun_{p,q}$ as in \cref{thm:thom}. From \cref{thm:KOPIB}, we have an isomorphism
 \[ \phi = \rho_!t_{p,q} \colon {H}^{p'+p, q'+q, p-2q,q}_P(P, \mR) 
 \xrightarrow{\cong}  \widetilde{H}^{p',q'}(\Th(\bun_{p,q}), \mR).\]
 
As a consequence, we have the following commutative diagrams. 
\begin{lem}\label{lem:thom-isos}
Let $n\geq 0$ and $\crush \colon P \to G/G$. There are commutative diagrams 
\[\xymatrix@C=3pc{
H^{p, q, n,\epsilon}_{P}(P, \mF  )\ar[d]_-{i^*_e}  \ar[r]^-{\crush_!  t_{n,n} \eee^{\epsilon+n} }_-\cong  & \widetilde{H}^{p+n,q+n}(\Th(\bun_{n,n}), \mF ) \ar[d]^-{i^*_e} & 
\\
H^{p,n}_{i^*_eP}(i^*_eP,\F_2)  \ar[r]^-{\crush_!t_{n}}_-\cong  & \widetilde{H}^{p+n}(\Th(\bun_{n}),\F_2) 
}\]
and
\[\xymatrix@C=3pc{
H^{p, q, -n,\epsilon}_{P}(P, \mF )\ar[d]_-{i^*_e}  \ar[r]^-{\crush_!  t_{n,0} \eee^\epsilon }_-\cong  & \widetilde{H}^{p+n,q}(\Th(\bun_{n,0}), \mF ) \ar[d]^-{i^*_e} & 
\\
H^{p,-n}_{i^*_eP}(i^*_eP,\F_2)  \ar[r]^-{\crush_!t_{n}}_-\cong  & \widetilde{H}^{p+n}(\Th(\bun_{n}),\F_2). 
}\]
\end{lem}
Now using \cref{thm:coh_pq}, we obtain the $KO(\Pi P)$-graded portion of the cohomology of $P$ as an $\M_2$-module.
\begin{cor}
\label{cor:xnr}
Let $n\geq 0$. 
\begin{enumerate}[(1)]
\item  $H^{*, *, n,0}_{P}(P, \mF )$ is a free $\M_2$-module generated by classes $x_{n,r}$ for $r\geq 0$ of degree
\begin{align*}
|x_{n,r}| =\begin{cases} (r,0,n,0) & 0\leq r \leq n\\
(r, \lceil (r+n)/2 \rceil-n,n,0) & r>n.
\end{cases}
\end{align*}
The restriction 
\[i^*_e \colon H^{r, 0, n,0}_{P}(P, \mF ) \to H^{r,n}_{i^*_eP}(i^*_eP,\F_2)  \]
sends $x_{n,r}$ to $u_1^n w_1^r =u_1^{n+r}a_1^r$.
\item $H^{*, *, -n,0}_{P}(P, \mF )$ is a free $\M_2$-module generated by classes $x_{-n,r}$ for $r\geq 0$ of degree
\begin{align*}
|x_{-n,r}| =\begin{cases} (r, n,-n,0) & 0\leq r \leq n\\
(r, \lceil (r+n)/2 \rceil,-n,0) & r>n.
\end{cases}
\end{align*}
The restriction 
\[i^*_e \colon H^{r, n, -n,0}_{P}(P, \mF ) \to H^{r,n}_{i^*_eP}(i^*_eP,\F_2)  \]
sends $x_{-n,r}$ to $u_1^n w_1^r =u_1^{n+r}a_1^r$.
\item The collection
\[\{x_{n,r}, ex_{n,r} \mid n\in \Z, r\geq 0\}\]
generates $H^{*,*,*,*}_{P}(P, \mF )$ as a free $\M_2$-module, and all generators and their $\uu$-multiples have non-zero restrictions.
\end{enumerate}
\end{cor}

\begin{rem}
As in \cref{rem:evenunique}, while the generators in even topological degrees are uniquely determined by their topological degree, in odd degrees, there can be a choice of basis.
\end{rem}

The next goal will be to identify the ring structure. There are various ways to do this, and our method will be to first identify the orientation and Euler classes of the bundles $\bun_{p,q}$. Then we will use their properties to deduce the relations in  $H^{*,*,*,*}_{P}(P, \mF )$.

We start with the orientation classes. Note that over $\F_2$, whenever orientation classes exist, they are unique. 
\begin{lem}
For any $(p,q)$ with $p\geq q\geq 0$, the bundle $\bun_{p,q}$ is $\F_2$-orientable, in the sense of \cref{defn:orientation_class}.  The orientation classes 
\[u_{\bun_{p,q}}\in H^{0,q,p-2q,q}_{P}(P, \mF )\] 
are given by
\[ u_{\gamma_{p,q}} = \begin{cases} \uu^q \eee^q x_{p-2q,0}  &p-2q\geq 0 \\
\uu^{p-q} \eee^q x_{p-2q,0} & p-2q \leq 0.
\end{cases}\]
\end{lem}
\begin{proof}
We have
\[\dim (\bun_{p,q})-|\bun_{p,q}| = (0, q, p-2q,q). \]
Since any nonequivariant bundle is $\F_2$-orientable, the underlying bundle is $\F_2$-orientable, so we need to prove that the restriction
\[H^{0, q, p-2q,q}_{P}(P, \mF )  \to H^{0, p-2q}_{i^*_eP}(i^*_eP,\F_2)\cong \F_2 \]
is an isomorphism.
This follows since 
 \[H^{0, q, p-2q,q}_{P}(P, \mF ) \cong\begin{cases} \F_2\{\uu^q e^q x_{p-2q,0}\} & p-2q\geq 0\\
 \F_2\{ \uu^{p-q} e^q  x_{p-2q,0}\} &p-2q \leq 0
 \end{cases} 
 \]
 and the generators restrict to generators.
\end{proof}

The orientation classes of $\taut$ and $\staut$ play a central role, so we define the following notation.
\begin{defn}
Let 
\begin{align*}
\uu_{10} &:=u_{\taut} \in H^{0,0,1,0}_{P}(P, \mF ) \\
\uu_{11} &:=u_{\staut} \in H^{0,1,-1,1}_{P}(P, \mF ).
\end{align*}
\end{defn}

Multiplicativity of orientation classes gives the following.
\begin{cor}
In $H^{*,*,*,*}_{P}(P, \mF )$, 
\[u_{\bun_{p,q}} = \uu_{10}^{p-q}\uu_{11}^q\]
and $i^*_eu_{\bun_{p,q}} = u_{\bun_p}= u_1^p$ for $u_1$ as in \cref{lem:underlyingPcoh}.
\end{cor}

We turn to the Euler classes. 

\begin{prop}
For $p \geq q \geq 0$, the Euler class $a_{\bun_{p,q}}$ is given by
\[a_{\bun_{p,q}} = \eee^q x_{p-2q,p}.\]
\end{prop}
\begin{proof}
First, we note that
\[H^{ p, q, p-2q,q}_{P}(P, \mF ) = \F_2\{e^qx_{p-2q, p}\}.
\]
To see this, let $n=p-2q$, and recall from \cref{lem:thom-isos} and \cref{thm:coh_pq}, that 
\[H^{ p, q,n,q}_{P}(P, \mF ) \cong 
\begin{rcases}
\begin{dcases}
\widetilde{H}^{2(p-q),p-q}(\Th(\gamma_{n,n}), \mF ) & n \geq 0 \\
\widetilde{H}^{2(p-q),p-q}(\Th(\gamma_{-n,0}), \mF ) & n < 0 
\end{dcases}
\end{rcases} \cong \F_2\{x_{2(p-q)}\}.
\]
Here we are in the special case where the generator of the cohomology of the Thom space is the only nonzero element in its bidegree, with no $\aa$-multiples intruding from lower degrees. Since $i^*_e(a_{\bun_{p,q}}) = a_{\bun_p}$, the Euler classes are nonzero. The claim follows.
\end{proof}

The Euler classes of $\taut$ and $\staut$ play a central role, so again, we give them a special name.
\begin{defn}
Let 
\begin{align*}
\aa_{10} &:=a_{\taut} \in H^{1,0,1,0}_{P}(P, \mF ) \\
\aa_{11} &:=a_{\staut} \in H^{1,1,-1,1}_{P}(P, \mF ).
\end{align*}
\end{defn}
\begin{cor}
In $H^{*,*,*,*}_{P}(P, \mF )$,
\[ a_{\bun_{p,q}} = \aa_{10}^{p-q}\aa_{11}^q\]
and $i^*_e a_{\bun_{p,q}}=a_{\bun_{p}} = a_1^p =u_1^pw_1^p$. 
\end{cor}

We compute a few Steenrod squares, which we use to compute relations between the generators of the cohomology of $P$.
\begin{lem}
In $H^{*,*,*,*}_{P}(P, \mF )$, we have
\begin{align*}
\bock (\uu) &= \aa &
\bock (\uu_{10}) &= \aa_{10} &
\bock (\uu_{11}) &= \aa_{11} .
\end{align*}
\end{lem}
\begin{proof}
The first square is well-known (c.f. \cite[Fig.1]{guillou2020cohomology}).
The others follow from
\[i^*_e\bock  = \Sq^1\]
and the fact that $\Sq^1(u_1)=a_1$ as shown in \cref{lem:underlyingPcoh}.
\end{proof}

\begin{lem}
\label{lem:relations}
In $H^{*,*,*,*}_{P}(P, \mF )$, 
we have the relations
\[ \eee \uu_{10} \uu_{11} = \uu \in H^{0,1,0,0}_{P}(P, \mF )\]
and
\[\eee(\uu_{10} \aa_{11} + \uu_{11} \aa_{10}) = \aa  \in H^{1,1,0,0}_{P}(P, \mF ). \]
Furthermore,
\[ \eee\aa_{10}\aa_{11} \in H^{2,1,0,0}_{P}(P, \mF )\]
is the (unique) generator.
\end{lem}
\begin{proof}
From the fact that $\eee= \eee_{\bun_{2,1}}$ and $\taut\staut=\bun_{2,1}$, we have from \cref{lem:euclassesrel} 
that
\[\uu_{10}\uu_{11}=u_{\taut}u_{\staut}=u_{\bun_{2,1}}=\uu \eee.\]
Since $\eee^2=1$, we get the first relation.
The second relation follows from the Cartan formula of \cref{lem:cartan}:
\begin{align*}
\eee\aa&= \bock (\eee\uu) \\
&= \bock (\uu_{10} \uu_{11})\\
&=\bock (\uu_{10}) \uu_{11}+ \uu_{10}\bock ( \uu_{11}) \\
&=\aa_{10} \uu_{11} + \uu_{10} \aa_{11}.
\end{align*}
The class
\[ \eee\aa_{10}\aa_{11} = \eee a_{\bun_{2,1}}  \in H^{2,1,0,0}_{P}(P, \mF ) \cong H^{2,1}_{P}(P, \mF ) \]
is nonzero since $\eee$ is a unit, and there is a unique nonzero element in $H^{2,1}_{P}(P, \mF )$.
\end{proof}
\begin{rem}\label{rem:edivide}
The class $\eee$ thus deserves the name
\[\eee =\frac{\uu_{10}\uu_{11}}{\uu} \]
even if $\uu$ is not a unit.
\end{rem}

We are finally ready for the main result of this section.

\begin{theorem}\label{thm:KOcohomologyofP}
    As an $\M_2$-algebra, the $KO(\Pi P)$-graded cohomology of $P$ is isomorphic to 
    \[H^{*,*,*,*}_{P}(P, \mF )\cong \frac{\M_2[\uu_{10},\uu_{11},\aa_{10},\aa_{11},\eee]}{(\uu_{10}\uu_{11}-\uu \eee,\aa_{10}\uu_{11}+\aa_{11}\uu_{10}-\aa \eee,\eee^2-1)},\] 
    where
    \begin{align*}
|\eee| &= (0,0,0,1) &  & \\ 
|\uu_{10}| &= (0,0,1,0)  & |\uu_{11}| &= (0,1,-1,1)  \\ 
|\aa_{10}|&=(1,0,1,0) & |\aa_{11}| &= (1,1,-1,1).
\end{align*}
\end{theorem}

\begin{proof}
Let $M$ be the $\M_2$-algebra on the right-hand side of the equation. From what we have shown so far, we obtain a homomorphism of $\M_2$-algebras  
\[
M \to H^{*,*,*,*}_{P}(P, \mF )
\]
sending the ring generators to the orientation classes, Euler classes, and homogeneity unit with the same names. 
To prove that this is an isomorphism of $\M_2$-modules, we show that a basis of $M$ as a free $\M_2$-module maps bijectively to a basis of $H^{*,*,*,*}_{P}(P, \mF )$ as a free $\M_2$-module. We have already shown the relations are satisfied and so this map is well-defined (see \cref{lem:ecopy} and \cref{lem:relations}).

Recall from part (3) of \cref{cor:xnr} that the collection 
\[\{x_{n,r}, ex_{n,r} \mid n\in \Z, r\geq 0\}\]
generates $H^{*,*,*,*}_{P}(P, \mF )$ as a free $\M_2$-module. So we have identified a basis for $H^{*,*,*,*}_{P}(P, \mF )$. In degrees $(*,*,n,0)$ and $(*,*,-n,0)$, the elements $x_{n,r}$ and $x_{-n,r}$ are generators. Similarly, for degrees $(*,*,n,1)$ and $(*,*,-n,1)$, the elements $ex_{n,r}$ and $ex_{-n,r}$ are generators.

We map the ring generators of $M$ to $H^{*,*,*,*}_{P}(P, \mF )$ as follows.
\begin{align*}
1 &\xmapsto{\phantom{testing}} x_{0,0}\\
u_{10} &\xmapsto{\phantom{testing}} x_{1,0}\\
u_{11} &\xmapsto{\phantom{testing}} ex_{-1,0}\\
a_{10} &\xmapsto{\phantom{testing}} x_{1,1}\\
a_{11} &\xmapsto{\phantom{testing}} ex_{-1,1}\\
\eee &\xmapsto{\phantom{testing}} ex_{0,0}
\end{align*}
This determines the map $M \to H^{*,*,*,*}_{P}(P, \mF )$. We need to identify the preimages of the free generators $x_{n,r}$, $x_{-n,r}$, $\eee x_{n,r}$, and $\eee x_{-n,r}$. Using the $\M_2$-algebra structure of $M$, we can determine the elements in those degrees. In some degrees, the generator is uniquely determined, but in other degrees there is a choice, and the generator is only determined up to adding $a$ times another element. Up to a change of basis for $H^{*,*,*,*}_{P}(P, \mF )$, we can choose the $\M_2$-module generator $x_{n,r}$ accordingly. 
So, for degree reasons, after possibly changing our choice of module generators for $H^{*,*,*,*}_{P}(P, \mF )$, multiples of the ring generators of $M$ map as follows. 
\begin{itemize}
    \item For $0\leq r\leq n$: 
       \begin{align*}
    \eee^n\uu_{11}^{n-r}\aa_{11}^r &\xmapsto{\phantom{testing}} x_{-n,r}\\
    \uu_{10}^{n-r}\aa_{10}^r &\xmapsto{\phantom{testing}} x_{n,r}.
    \end{align*}
    \item For $r > n$:
    \begin{itemize}
    \item If $r+n$ is even, then
       \begin{align*}
    \aa_{10}^{\frac{r-n}{2}}\aa_{11}^{\frac{r+n}{2}}\eee^{\frac{r+n}{2}} &\xmapsto{\phantom{testing}} x_{-n,r}\\
    \aa_{10}^{\frac{r+n}{2}}\aa_{11}^{\frac{r-n}{2}}\eee^{\frac{r-n}{2}}
    &\xmapsto{\phantom{testing}} x_{n,r}.
    \end{align*}
    \item If $r+n$ is odd, then 
    \begin{align*}
   \uu_{11}\aa_{10}^{\frac{r-n-1}{2}}\aa_{11}^{\frac{r+n-1}{2}}\eee^{\frac{r+n+1}{2}} &\xmapsto{\phantom{testing}} x_{-n,r}\\
    \uu_{10}\aa_{10}^{\frac{r+n-1}{2}}\aa_{11}^{\frac{r-n-1}{2}}\eee^{\frac{r-n+1}{2}}
    &\xmapsto{\phantom{testing}} x_{n,r}.
    \end{align*}
    \end{itemize}
\end{itemize}
Furthermore, these elements multiplied by $\eee$ are mapped to $\eee x_{n,r}$ and $\eee x_{-n,r}$, respectively. 

It remains to show that the collection
\begin{align}
\begin{split}
\label{eq:basisforM}
    & \{ \uu_{11}^{n-r}\aa_{11}^r,  \uu_{10}^{n-r}\aa_{10}^r, \eee\uu_{11}^{n-r}\aa_{11}^r,  \eee\uu_{10}^{n-r}\aa_{10}^r : 0 \leq r \leq n \} \\
    & \cup \{\aa_{10}^s\aa_{11}^t, \eee\aa_{10}^s\aa_{11}^t: s,t \geq 0 \} \\
    & \cup \{\uu_{10}\aa_{10}^s\aa_{11}^t, \eee\uu_{10}\aa_{10}^s\aa_{11}^t: 0 \leq t < s\}\\
     & \cup \{\uu_{11}\aa_{10}^s\aa_{11}^t, \eee \uu_{11}\aa_{10}^s\aa_{11}^t : 0 \leq s < t\}
\end{split}
\end{align}
generates $M$ as an $\M_2$-module. 
For this we need to find a basis of $M$ as an $\M_2$-module.

Instead of directly finding a basis for $M$ as an $\M_2$-module, we first determine a basis for $M/(\eee-1)$. Once obtained, a basis for $M$ as an $\mathbb{M}_2$-module is given by multiplying each basis element of $M/(\eee-1)$ by $\eee$.

We note that $\{\uu_{10}\uu_{11}+\uu, \uu_{10}\aa_{11}+\uu_{11}\aa_{10}+\aa\}$ is a Gr\"obner basis (with lexicographic ordering $\uu_{10} < \uu_{11} < \aa_{10} < \aa_{11}$) for the ideal $(\uu_{10}\uu_{11}+\uu, \uu_{10}\aa_{11}+\uu_{11}\aa_{10}+\aa)$ in $\mathbb{M}_2[\uu_{10},\uu_{11},\aa_{10},\aa_{11}]$. Thus, the quotient space
\[
\frac{\mathbb{M}_2[\uu_{10},\uu_{11},\aa_{10},\aa_{11}]}{(\uu_{10}\uu_{11}+\uu, \uu_{10}\aa_{11}+\uu_{11}\aa_{10}+\aa)} \cong M/(\eee-1)
\]
has the following basis:
\[
\{ \uu_{10}^s \aa_{10}^t, \uu_{11}^s \aa_{11}^t, \aa_{10}^s \aa_{11}^t, \uu_{11} \aa_{10}^s \aa_{11}^t : s,t \geq 0 \}.
\]
By taking this collection, together with the collection formed by multiplying each element by $\eee$, we obtain a basis for $M$. 

To establish a bijection with our desired basis, we need to identify this collection with \eqref{eq:basisforM} using relations in the ring. We can immediately identify the elements in the first two lines and the fourth line of \eqref{eq:basisforM}. For the third line, we use the relation
\[
\uu_{10}\aa_{11}+\uu_{11}\aa_{10}+\aa = 0.
\]
Multiplying by $a_{10}^s$ and $a_{11}^{t-1}$, this gives
\[
\uu_{10} \aa_{10}^{s} \aa_{11}^{t} = \uu_{11} \aa_{10}^{s+1} \aa_{11}^{t-1} + \aa \aa_{10}^{s} \aa_{11}^{t-1}.
\]
Thus, we can replace $\uu_{11} \aa_{10}^{s+1} \aa_{11}^{t-1}$ by $\uu_{10} \aa_{10}^{s} \aa_{11}^{t}$ when $1\le t< s$ to get the generators in the third row of \eqref{eq:basisforM}. 
We note that none of these elements are zero in $H^{*,*,*,*}_P(P,\mF)$. Indeed, their restrictions are nonzero in $H^{*,*}_{i^*_eP}(i^*_eP,\mF)$. There are no $\F_2$-linear relations among the generators for degree reasons. Therefore, there are no $\M_2$-relations since $H^{*,*,*,*}_P(P,\uF_2)$ is free.
Thus, the bijection is established. 
\end{proof}

%% file: ppaper-fullrocomputation.tex
\subsection{The  $RO(\Pi P)$-graded cohomology}\label{sec:ROPIP}
We can now turn to the case when $\mu\neq 0$ and compute the full $RO(\Pi P)$-cohomology of $P$. Recall that the full grading is
\[RO(\Pi P)\cong \Z^3\times(\Z/2)^3,\] and that we denote the elements by $(p,q,n,\epsilon,\mu)$.
From \cref{thm:KOcohomologyofP}, we know the subring of the cohomology associated to gradings with $\mu=0$.  To compute the rest of the $RO(\Pi P)$-graded cohomology we first look at degrees 
\[(p,0,0,0,\mu).\]
  In particular, we will see there is a unit  $\vv$ in degree $(0,0,0,0,1)$ making the $RO(\Pi P)$-graded cohomology the tensor product 
\[
H^{RO(\Pi P)}_{P}(P,\mF) \cong H^{KO(\Pi P)}_{P}(P,\mF) \otimes_{\M_2} \M_2[\vv]/(\vv^2-1).
\]

In order to produce the class $\vv$, we use the fact that any representation in degree $(0,0,0,0,1)$ is homogeneously trivial as in \cref{defn:homotrivial}.

\begin{lem}
Let $\gamma=(0,0,0,0,1)$. There is an isomorphism 
\[H^{*,*,0,0,1}_{P}(P,\mF) \cong H^{*,*}(P,\mF)  .\]
In particular,
\begin{align*}
H^{0,0,0,0,1}_{P}(P,\mF) &= \F_2, & H^{-1,-1,0,0,1}_{P}(P,\mF) &= 0.
\end{align*}
\end{lem}
\begin{proof}
We have $\gamma(x)=0$ for all $x\in \Pi P$, so $\gamma$ is homogeneously trivial. 
From \cref{thm:gammanomatter} and \cref{thm:shriek-iso} we get 
\[H^{*,*,0,0,1}_{P}(P,\mF) \xrightarrow{\cong} H^{*,*,0,0,0}_P(P,\mF) \xrightarrow[\rho_!]{\cong}H^{*,*}(P,\mF).\qedhere  \]
\end{proof}

\begin{defn}
Let 
\[ \vv\in H^{0,0,0,0,1}_P(P,\mF) = \F_2 \] denote the nonzero element. 
\end{defn}
\begin{cor}
Under the map
\[i^*_e \colon H^{0,0,0,0,1}_{P}(P,\mF) \to H^{0,0}_{i^*_eP}(i^*_eP,\mF) \]
the element $\vv$ satisfies
\[i^*_e\vv=1.\]
\end{cor}
\begin{proof}
From the forgetful long exact sequence of \cref{lem:ParamForgetfulLES}, the kernel of $i^*_e$ is the image of multiplication by $\aa$
\[  H^{-1,-1,0,0,1}_{P}(P,\mF)  \xrightarrow{ \aa}  H^{0,0,0,0,1}_{P}(P,\mF). \]
The source is zero, so $\vv$ is not in the kernel of $i^*_e$. Since there is a unique nonzero element in the target, the claim follows.
\end{proof}

Finally, we have the following complete computation.
\begin{theorem}\label{thm:ROcohomologyofP}
    As $\M_2$-algebras the $RO(\Pi P)$-graded cohomology of $P$ is isomorphic to
    \[H^{*,*,*,*,*}_{P}(P,\uF_2)\cong \frac{\M_2[\uu_{10},\uu_{11},\aa_{10},\aa_{11},\eee,\vv]}{(\uu_{10}\uu_{11}-\uu \eee,\aa_{10}\uu_{11}+\aa_{11}\uu_{10}-\aa \eee,\eee^2-1,\vv^2-1)},\]
where
\begin{align*}
|\eee| &= (0,0,0,1,0) & |\vv| &= (0,0,0,0,1)   \\ 
|\uu_{10}| &= (0,0,1,0,0)  & |\uu_{11}| &= (0,1,-1,1,0)  \\ 
|\aa_{10}|&=(1,0,1,0,0) & |\aa_{11}| &= (1,1,-1,1,0).
\end{align*}
\end{theorem}
\begin{proof}
All that remains is to prove that $\vv^2=1$. We have
\[\vv^2\in  H^{0,0,0,0,0}_{P}(P,\uF_2)=\F_2 \]
so $\vv^2$ is either zero or one.
Since
$i^*_e(\vv^2) = i^*_e(\vv)^2=1$,
it must be $1$.
\end{proof}

\begin{rem}
The classes $\eee$ and $\vv$ both restrict to $1$ and can be thought of as orientation classes for the virtual representations $(0,0,0,1,0)$ and $(0,0,0,0,1)$ respectively.
\end{rem}

%% file: ppaper-restriction.tex

\subsection{The $RO(\Pi P)$-graded cohomology of $P^{C_2}$}
We finish this section by computing the cohomology of the fixed points and the restriction. Recall from \eqref{eq:fixed} that $P^{C_2}=P_0\sqcup P_1$ for $P_i \simeq \R P^\infty$ and that $\iota_i\colon P_i\rightarrow P$ denote the inclusions. In this section we compute the
maps
\[\iota_i^* \colon H^{\RO (\Pi P)}_P(P, \uF_2) \to H^{\RO (\Pi P)}_P(P_i, \uF_2).  \]
First, we compute $H^{\RO (\Pi P)}_P(P_i, \uF_2)$ using \cref{thm: use cohomology of fixed points}.

To use the theorem we have to compute the kernels $\kappa_{b_i}$ of the maps 
\[b_i^*\colon \RO( \Pi P) \to \RO (\Pi G/G)\cong RO(C_2)\] 
induced by the inclusions of the points $b_i\colon G/G \to P$ for $i=0,1$. 
Let $(p,q,n,\epsilon, \mu)\in \Z^3\times (\Z/2)^2$. For $i=0$, one computes that
\[b_0^*(p,q,n,\epsilon, \mu)=(p,q)\]
and so 
\[\kappa_{b_0}\cong \Z\times (\Z/2)^2\]
with generators $(0,0,1,0,0),(0,0,0,1,0),(0,0,0,0,1)\in \Z^3\times (\Z/2)^2$.
For $i=1$ we have 
\[b_1^*(p,q,n,\epsilon, \mu)=(p,q+n)\]
and so we get
\[\kappa_{b_1}\cong \Z\times (\Z/2)^2\]
with generators 
$(0,1,-1,1,0),(0,0,0,1,0),(0,0,0,1,0)\in \Z^3\times (\Z/2)^2$.
Now we apply \cref{thm: use cohomology of fixed points} and get
\begin{align*}H^{\RO (\Pi P)}_{P}(P_0, \uF_2) &\cong H^*(P_0,\F_2)\otimes \F_2[\kappa_{b_0}]\otimes \M_2\\
&\cong \F_2[w_1]\otimes \F_2[\uu_{10}^{\pm1}, \eee,\vv]/(\eee^2-1,\vv^2-1)\otimes \M_2\\
&\cong \M_2[w_1,\uu_{10}^{\pm1},\eee,\vv]/(\eee^2-1,\vv^2-1) \\
&\cong \M_2[\aa_{10},\uu_{10}^{\pm1},\eee,\vv]/(\eee^2-1,\vv^2-1) \\
\end{align*}
where $\aa_{10} = w_1 \uu_{10}$. The degrees are
\begin{align*}
\vert \aa_{10}\vert&=(1,0,1,0,0),  &  \vert \uu_{10}\vert &=(0,0,1,0,0), \\
  \vert \eee\vert &=(0,0,0,1,0), &  \vert \vv\vert&=(0,0,0,0,1).
\end{align*}
Similarly, 
\begin{align*}H^{\RO (\Pi P)}_{P}(P_1,\uF_2)
&\cong \M_2[\aa_{11},\uu_{11}^{\pm1},\eee,\vv]/(\eee^2-1,\vv^2-1)
\end{align*}
with 
\begin{align*}
\vert \aa_{11} \vert &=(1,1,-1,1,0), &  \vert \uu_{11}\vert &=(0,1,-1,1,0),  \\
\vert \eee\vert &=(0,0,0,1,0), & \vert \vv\vert &=(0,0,0,0,1).
\end{align*}
\begin{theorem}
For $i=0,1$,
\[\iota_i^* \colon H^{RO(\Pi P)}_{P}(P ,\mF) \to H^{RO(\Pi P)}_{P}(P_i ,\mF) \]
 induces an isomorphism
\[H^{RO(\Pi P)}_{P}(P_0 ,\mF)  \cong  \uu_i^{-1} H^{RO(\Pi P)}_{P}(P ,\mF).\]
The classes  $\aa_i,\uu_i, \eee$ and $\vv$ are mapped to the same named elements and
\begin{align*}
\iota_0^*(\uu_{11})  &= \uu_{10}^{-1}\uu \eee & \iota_0^*(\aa_{11})  &= \uu_{10}^{-1}\aa \eee+ \uu_{10}^{-2} \aa_{10}\uu \eee \\
\iota_1^*(\uu_{10})  &= \uu_{11}^{-1}\uu \eee & \iota_1^*(\aa_{10})  &= \uu_{11}^{-1}\aa \eee+ \uu_{11}^{-2} \aa_{11}\uu \eee .
\end{align*}
\end{theorem}
\begin{proof}
We prove this when $i=0$. The proof for $i=1$ is similar.
The restriction $\iota_0^*$ sends $\aa_{10},\uu_{10}, \eee$ and $\vv$ to the same named classes. For $\aa_{10}$ and $\uu_{10}$, this follows from the naturality of Euler and Thom classes.  For $\eee$ and $\vv$, we can use the commutativity of
\[\xymatrix{ 
H^{\RO (\Pi P)}_{P}(P,\uF_2)\ar[r]^-{\iota_0^*}\ar[d]^-{i^*_e}& H^{\RO (\Pi P)}_{P}(P_0,\uF_2) \ar[d]^-{i^*_e}\\
 H^{\RO (\Pi i^*_eP)}_{i^*_eP}(i^*_eP,\uF_2) \ar[r]^-{i^*_e\iota_0^*}_-\cong &  H^{\RO (\Pi i^*_eP)}_{i^*_eP}(i^*_eP_0,\uF_2)
}  \]
and use the defining property of these classes, i.e., they are the unique classes that restrict to $1$ in their respective degrees.
Since $\iota_0^*(\uu_{10})$ is a unit, $\iota_0^*$ factors through the localization at $\uu_{10}$,
\[\uu_{10}^{-1}H^{RO(\Pi P)}_{P}(P ,\mF) \to H^{RO(\Pi P)}_{P}(P_0 ,\mF). \]
From the relations in $H^{RO(\Pi P)}_{P}(P ,\mF) $, we deduce that in $\uu_{10}^{-1}H^{RO(\Pi P)}_{P}(P ,\mF)$
\begin{align*}
\uu_{11}  &= \uu_{10}^{-1}\uu \eee \\ \aa_{11}  &= \uu_{10}^{-1}\aa \eee+ \uu_{10}^{-2} \aa_{10}\uu \eee.
\end{align*}
This renders the generators $\aa_1$ and $\uu_{11}$ redundant and we see that 
\begin{align*}\uu_{10}^{-1}H^{RO(\Pi P)}_{P}(P ,\mF) &\cong \M_2[\aa_{10}, \uu_{10}^{\pm 1}, \eee, \vv]/(\eee^2-1, \vv^2-1) \\
&\cong H^{RO(\Pi P)}_{P}(P_0 ,\mF) \qedhere.
\end{align*}
\end{proof}

\section{Parametrized cohomology of $B_{C_2}U(1)$}\label{sec:B}

Let 
\[B := B_{C_2}U(1) = P(\mathbb{C}^{2\infty,\infty}),\]
the classifying space for complex $C_2$-equivariant line bundles.
The $\RO (\Pi B)$-graded cohomology of $B$ was computed by Costenoble in \cite{CostenobleB} with coefficients in the parametrized Burnside Mackey functor 
(see \cite[II.11]{CostenobleB}), as well as other coefficients such as $\underline{\Z}$ (see \cite[II.13]{CostenobleB}).  These computations used isotropy separation sequences. In this section, we explain how to compute the parametrized cohomology  of $B$ with constant parametrized  $\uF_2$-coefficients using the techniques we used to compute the parametrized cohomology of $P$.
We adapt the statement to our notation  rather than that of \cite{CostenobleB}, incorporating the orientation classes. Costenoble computed the following cohomology of $B$.

\begin{theorem}[Costenoble]\label{thm:costenobleB}
Let $B=B_{C_2}U(1) = P(\C^{2\infty,\infty})$ be the classifying for $C_2$-equivariant complex line bundles. Let $\tautB$ be the tautological complex line bundle and $\tautBchi$ be its pullback along the automorphism of $B$ obtained by interchanging the trivial and sign representations. Then 
\[RO(\Pi B)= KO(\Pi B) \cong \Z^3\] 
and the parametrized cohomology of $B$ is given by
\[H^{*,*,*}_B(B,\uF_2) \cong
\frac{\M_2[\uu_{\tautB},\uu_{\tautBchi},\aa_{\tautB},\aa_{\tautBchi}]}{(\uu_{\tautB}\uu_{\tautBchi} = \uu^2,\aa_\tautB \uu_{\tautBchi} + \uu_{\tautB}  \aa_{\tautBchi}=\aa^2)} 
\]
where $\uu_{\tautB},\uu_{\tautBchi}$ are the orientation classes and $\aa_{\tautB},\aa_{\tautBchi}$ the Euler classes of $\tautB$ and $\tautBchi$, with degrees 
\begin{align*}
|\uu_{\tautB}|  &= (0,1,1)  & |\aa_{\tautB}|  &= (2,1,1) \\ 
|\uu_{\tautBchi}| &= (0,1,-1)  & |\aa_{\tautBchi}|  &= (2,1,-1) .
\end{align*}
\end{theorem}

We show how to reprove this using the methods of this paper. The skeleton of $\Pi B$ is less complicated than that of $\Pi P$. 
Similarly to the fixed set of $P$, we have $B^{C_2} = B_1 \sqcup B_0$, with 
$B_1 = P(\mathbb{C}^{\infty,\infty}) = \mathbb{C}P^\infty$ and 
$B_0 = P(\mathbb{C}^{\infty,0}) = \mathbb{C}P^\infty$, and also 
$i_e^* B = \mathbb{C}P^\infty$.  
Since $\mathbb{C}P^\infty$ is simply connected, there are no analogs of 
$g$, $g_0$, and $g_1$. 
Hence, the only nontrivial morphisms in the skeleton are those corresponding to $t$, $p_0$, and $p_1$, and we continue to denote them by the same names.
Costenoble proves in Proposition 6.1 of \cite{CostenobleB} that for any $\gamma \in \RO( \Pi B)$, the morphisms $\gamma(t)$, $\gamma(p_0)$, and $\gamma(p_1)$ are uniquely determined by the values $\gamma(b_0) = \R^{p,q_0}$ and $\gamma(b_1) = \R^{p,q_1}$, with the additional constraint that $q_0$ and $q_1$ must have the same parity.
It follows that
\[
\RO (\Pi B) \cong \mathbb{Z}^3, \quad (p,q,n) \in \mathbb{Z}^3,
\]
with $q_0 = q - n$ and $q_1 = q + n$.  We depict this as follows, using the same conventions as in \cref{sec:reps of fun gpd}. 
\[\xymatrix@C=1pc{
&  \Pi  B& & \ar[rr]^-\gamma & & & & \vV_{C_2}\\
b_0 & & b_1  & & &&  \R^{p,q-n} & & \R^{p,q+n}  \\
 & b \ar@{->}[ur]_-{p_1} \ar@{->}[ul]^-{p_0} \ar@(dl,dr)@{>}[]_-{t} & && & & & C_2 \times \R^p \ar@{->}[ur]_-{1} \ar@{->}[ul]^-{1}   
\ar@(dl,dr)@{>}[]_-{C_2\times (-1)^{q-n}}
 }\]

Let $\gamma$ denote the tautological complex line bundle. Define $\autP\colon B\rightarrow B$ in a way completely analogous to the automorphism $\autP$ of $P$ discussed in \cref{rem:chi}. That is, $\autP$ is the $C_2$-homeomorphism that sends a line $[v_1,v_0]$ through $(v_1,v_0)\in \mathbb{C}^{\infty,\infty}\oplus \mathbb{C}^{\infty,0}$ to the line through $(v_0,v_1)$. Let $\tautBchi$ be the pullback of $\tautB$ along $\autP$.
Then 
\begin{align*}
\dim (\tautB) &= (2,1,1)\\ 
\dim ( \tautBchi) &= (2,1,-1).
\end{align*}
 Together with the trivial line bundles with fibers $\R^{1,0}$ and $\R^{1,1}$ (which have degrees $(1,0,0)$ and $(1,1,0)$ respectively), these 
generate $RO(\Pi B)$.
So $KO(\Pi B) = RO(\Pi B)$.

Let $\uu_{\tautB}$  and  $\uu_{\tautBchi}$ be the orientation class of $\gamma$ and  $\autP \gamma$ respectively. Similarly, let $\aa_{\tautB}$ and $\aa_{\tautBchi}$ be the associated Euler classes.
Then
\begin{align*}
|\uu_{\tautB}|  &= (0,1,1)  & |\aa_{\tautB}|  &= (2,1,1) \\ 
|\uu_{\tautBchi}| &= (0,1,-1)  & |\aa_{\tautBchi}|  &= (2,1,-1) .
\end{align*}

We can deduce the relations between these elements using a similar approach to the computation in  \cref{sec:KOPiP}.
First, it follows from \cref{lem:uclassesforGtrivial} that the product $\uu_{\tautB}\uu_{\tautBchi}$ corresponds to the orientation class of the bundle of dimension $(0,2,0)$, and hence we have a relation
\[
\uu_{\tautB}\uu_{\tautBchi} = \uu^2 \in H^{0,2,0}(B,\uF_2) \cong H^{2,0}(G/G,\uF_2).
\]
Applying the operation $\Sq^{2,1}$ to both sides of the identity $\uu_{\tautB}\uu_{\tautBchi} = \uu^2$ and using the equivariant Cartan formula of \cref{lem:cartan}
gives
\begin{align*}
\uu \aa^2 &=\Sq^{2,1}(\uu^2)\\
&=\Sq^{2,1}(\uu_{\tautB}\uu_{\tautBchi}) \\
&= \Sq^{2,1}(\uu_{\tautB}) \uu_{\tautBchi} + \uu \Sq^{1,0}(\uu_{\tautB}) \Sq^{1,0}(\uu_{\tautBchi}) + \uu_{\tautB} \Sq^{2,1}(\uu_{\tautBchi})\\
&= \uu (\aa_\tautB \uu_{\tautBchi} + \uu_{\tautB}  \aa_{\tautBchi}).
\end{align*}
Multiplication by $\uu$ is injective in the relevant bidegrees, and so we deduce that 
\[\aa_\tautB \uu_{\tautBchi} + \uu_{\tautB}  \aa_{\tautBchi}=\aa^2.\]

Thus we have a ring homomorphism
\begin{equation}
\label{eq:cohBC2U(1)}
\frac{\M_2[\uu_{\tautB},\uu_{\tautBchi},\aa_{\tautB},\aa_{\tautBchi}]}{(\uu_{\tautB}\uu_{\tautBchi} = \uu^2,\aa_\tautB \uu_{\tautBchi} + \uu_{\tautB}  \aa_{\tautBchi}=\aa^2)} \to H^{*,*,*}(B,\uF_2).
\end{equation}
In fact, it is an isomorphism, which can be shown by comparing the dimension in each degree on both sides. For this, we need to know $H^{p,q,n}(B,\uF_2)$ for all $p,q,n$, which we get from the Thom isomorphism as we did for $P$. Indeed, from \cref{thm:KOPIB} and the fact that $RO(\Pi B)=KO(\Pi B)$, we have a ring isomorphism
\[ H^{\star+\gamma}_B(B,\mF) \cong \widetilde{H}^{RO(G)}(\Th(-\gamma),\mF).\]

So, for $n \ge 0$, we have
\begin{align}
\begin{split}
\label{eq:ThomIsosB}
    &H^{p,q,n}_B(B,\uF_2)
    \cong \widetilde{H}^{p+2n,q+n}(\Th(\chi \gamma^{\oplus n}),\uF_2)\\
    &H^{p,q,-n}_B(B,\uF_2)
    \cong \widetilde{H}^{p+2n,q+n}(\Th(\gamma^{\oplus n}),\uF_2).
\end{split}
\end{align}

\begin{proposition}
    The reduced $RO(C_2)$-graded cohomologies of $\Th(\gamma^{\oplus n})$ and $\Th(\chi\gamma^{\oplus n})$ are isomorphic in every bidegree:
    \[\widetilde{H}^{p+2n,q+n}(\Th(\gamma^{\oplus n}))\cong \widetilde{H}^{p+2n,q+n}(\Th(\chi\gamma^{\oplus n}))\]
    and have $\M_2$-generators in degrees 
\begin{align*}&(2n,2n),(2n+2,2n),(2n+4,2n),\ldots,(4n,2n),\\
&(4n+2,2n+2),(4n+4,2n+2),(4n+6,2n+4),(4n+8,2n+4)\\&\ldots,(4n+4i-2,2n+2i),(4n+4i,2n+2i),\ldots\end{align*}
    
\end{proposition}
\begin{proof}
    Let $B_{n}$ be the quotient of the map 
    \[B^n\coloneqq P(\mathbb{C}^{n,0})\rightarrow B=P(\mathbb{C}^{2\infty,\infty}),\]
    and let $B_{-n}$ be the quotient of the map
    \[B^{-n}\coloneqq P(\mathbb{C}^{n,n})\rightarrow B=P(\mathbb{C}^{2\infty,\infty}).\]
    Then there are $C_2$-homeomorphisms
\[\Th(\bun^{\oplus n}) \cong B_{n} \]
and
\[\Th(\chi\gamma^{\oplus n}) \cong B_{-n} \]
 See \cref{cor:pqduality} for the corresponding result for $P$.
The $RO(C_2)$-graded cohomology of $B_n$ and $B_{-n}$ can be computed analogously to the computation in \cref{section:cohomologyofthomspaces}. Start with a $\Rep(C_2)$-structure on $B_n$ and $B_{-n}$. For this we give $B$  Schubert cell structures such that $B^n$ (respectively $B^{-n}$) forms the $(2n-1)$-skeleton, which is the same as the $(2n-2)$-skeleton because all cells are even-dimensional. 
For $B^n$ that means we take a cell structure for $B$ given by choosing the order of representations in $\mathcal{U}$ to be 
\[
\underbrace{\R^{2,0} \oplus \cdots \oplus \R^{2,0} }_{n} \oplus \underbrace{\R^{2,2}  \oplus \cdots \oplus \R^{2,2}}_{n} \oplus \R^{2,0} \oplus \R^{2,2} \oplus \R^{2,0} \oplus\R^{2,2}\oplus \cdots.
\]
Taking the quotient by the $(2n-2)$-skeleton 
gives a $\Rep(C_2)$-complex structure for $B_n$.
Similarly, for $B^{-n}$ we take a cell structure for $B$ using the order of representations
\[
\underbrace{\R^{2,2} \oplus \cdots \oplus \R^{2,2} }_{n} \oplus \underbrace{\R^{2,0}  \oplus \cdots \oplus \R^{2,0}}_{n} \oplus \R^{2,2} \oplus \R^{2,0} \oplus \R^{2,2} \oplus\R^{2,0}\oplus \cdots,
\]
and quotient by the $(2n-2)$-skeleton. 

A similar degree argument as for $P$ shows the attaching maps all induce trivial maps in $H^{*,*}(-,\mF)$.  So the $RO(C_2)$-graded cohomology of $B_n$ is free over $\M_2$ with one generator for each cell, and similarly for $B_{-n}$. This gives the $\M_2$-generators in the degrees listed above.
\end{proof}

The Thom isomorphisms and identifications in \eqref{eq:ThomIsosB} yield the following cohomology.
\begin{cor}
\label{cor:generatorsB}
The cohomology
$H^{*,*,*}_B(B,\uF_2)$ has generators in degrees
\begin{align*}&(0,n,n),(2,n,n),(4,n,n),\ldots,(2n,n,n),\\
&(2n+2,n+2,n),(2n+4,n+2,n),(2n+6,2n+4,n),(2n+8,n+4,n)\\&\ldots,(2n+4i-2,n+2i,n),(2n+4i,n+2i,n),\ldots\end{align*} and
\begin{align*}&(0,n,-n),(2,n,-n),(4,n,-n),\ldots,(2n,n,-n),\\
&(2n+2,n+2,-n),(2n+4,n+2,-n),(2n+6,2n+4,-n),\\
&(2n+8,n+4,-n),\ldots,(2n+4i-2,n+2i,-n),(2n+4i,n+2i,-n),\ldots\end{align*} 
for all $n\ge 0$.
\end{cor}
To finish the proof of \cref{thm:costenobleB} from this perspective, one has to prove that these generators are in bijection with the $\mathbb{M}_2$-generators of the left-hand-side of \eqref{eq:cohBC2U(1)}: 
\[\frac{\M_2[\uu_{\tautB},\uu_{\tautBchi},\aa_{\tautB},\aa_{\tautBchi}]}{(\uu_{\tautB}\uu_{\tautBchi} = \uu^2,\aa_\tautB \uu_{\tautBchi} + \uu_{\tautB}  \aa_{\tautBchi}=\aa^2)}.\]
Recall that 
\begin{align*}
|\uu_{\tautB}|  &= (0,1,1)  & |\aa_{\tautB}|  &= (2,1,1) \\ 
|\uu_{\tautBchi}| &= (0,1,-1)  & |\aa_{\tautBchi}|  &= (2,1,-1) 
\end{align*}
Note that $\{\uu_{\tautB}\uu_{\tautBchi} + \uu^2,\aa_\tautB \uu_{\tautBchi} + \uu_{\tautB}  \aa_{\tautBchi}+\aa^2\}$  is a Gröbner basis with lexicographic ordering $\uu_{\tautB}<\uu_{\tautBchi}<\aa_\tautB<\aa_{\tautBchi}$ for the ideal $(\uu_{\tautB}\uu_{\tautBchi} + \uu^2,\aa_\tautB \uu_{\tautBchi} + \uu_{\tautB}  \aa_{\tautBchi}+\aa^2)$ in $\M_2[\uu_{\tautB},\uu_{\tautBchi},\aa_\tautB,\aa_{\tautBchi}]$ and thus \[\frac{\M_2[\uu_{\tautB},\uu_{\tautBchi},\aa_{\tautB},\aa_{\tautBchi}]}{(\uu_{\tautB}\uu_{\tautBchi} = \uu^2,\aa_\tautB \uu_{\tautBchi} + \uu_{\tautB}  \aa_{\tautBchi}=\aa^2)}\]
has $\M_2$-basis
\[\mathcal{B}=\{\uu_{\tautB}^s\aa_{\tautB}^t, \uu_{\tautBchi}^s \aa_{\tautBchi}^t, \aa_{\tautB}^s\aa_{\tautBchi}^t, \uu_{\tautBchi}\aa_{\tautB}^s\aa_{\tautBchi}^t: s,t\ge 0\}.\]
Note that $RO(\Pi i^*_eB )\cong \Z$ as $i^*_eB$ is simply connected and that the restriction 
\[i^*_e \colon H^{*,*,*}_B(B,\uF_2) \to H^*_{i^*_eB}(i^*_eB,\F_2) \cong H^*(i^*_eB,\F_2) \cong \F_2[c_1],\]
is given by
\begin{align*}
i^*_e(\uu_{\tautB}) &= i^*_e(\uu_{\tautBchi}) =1 \\ i^*_e(\aa_{\tautB}) &= i^*_e(\aa_{\tautBchi})=c_1,
\end{align*}
where $c_1$ is the first Chern class of $\tautB$.
This shows that none of our basis elements are zero in $H^{*,*,*}_B(B,\uF_2)$. There are no $\F_2$-linear relations for degree reasons. Therefore, there are no $\M_2$-relations since $H^{*,*,*}_B(B,\uF_2)$ is free.

It remains to write down a bijection between the elements in $\mathcal{B}$ and the $\M_2$-generators of $H^{*,*,*}_B(B,\uF_2)$ in \cref{cor:generatorsB}.
\begin{itemize}
    \item $\uu_{\tautB}^s\aa_{\tautB}^t$ corresponds to the generator of $H^{*,*,*}(B,\uF_2)$ in degree 
     \begin{itemize}
        \item
    $(2t,s,s+t)=(2i,n-i,n)$
     with $t=i$, $n=s+t$, $i=0,\ldots,n$
     \end{itemize}
    \item $\uu_{\tautBchi}^s\aa_{\tautBchi}^t$ corresponds to the generator of $H^{*,*,*}(B,\uF_2)$ in degree 
      \begin{itemize}
        \item $(2t,s,-s+t)=(2i,n-i,-n)$ with $t=i$, $n=s+t$, $i=0,\ldots,n$
        \end{itemize}
    \item $\aa_{\tautB}^s\aa_{\tautBchi}^t$ corresponds to the generator of $H^{*,*,*}(B,\uF_2)$ in degree
    \begin{itemize}
        \item $(2n+4i,n+2i,n)=(2s+2t,s+t,s-t)$ for $n=s-t$ and $i=t$ when $s-t>0$
        \item $(2n+4i,n+2i,-n)=(2s+2t,s+t,s-t)$ for $n=t-s$ and $i=s$ when $s-t<0$
    \end{itemize}
    \item $\uu_{\tautBchi}\aa_{\tautB}^s\aa_{\tautBchi}^t$ corresponds to the generator of $H^{*,*,*}(B,\uF_2)$ in degree
    \begin{itemize}
        \item $(2n+4i-2,n+2i,n)=(2s+2t,s+t+1,s-t-1)$ for $n=s-t-1$ and $i=t+1$ when $s-t-1>0$
        \item $(2n+4i-2,n+2i,-n)=(2s+2t+1,s+t+1,s-t-1)$ for $n=t+1-s$ and $i=s$ when $s-t-1<0$.
    \end{itemize}
\end{itemize}